\documentclass[a4paper,11pt]{amsart}
\usepackage[all]{xy}
\SelectTips{cm}{}
\usepackage{graphicx}
\usepackage{psfrag}
\usepackage{tikz}
\usepackage{tikz-cd}
\usepackage{mathtools}
\usepackage{amsmath,amssymb,tikz-cd,amsthm}
\usepackage[export]{adjustbox}
\usepackage{calligra}
\usepackage{comment}
\usetikzlibrary{decorations.pathreplacing}
\usetikzlibrary{matrix,arrows}
\usepackage[colorlinks, linkcolor=blue,anchorcolor=blue,
citecolor=blue,urlcolor=Emerald]{hyperref}

\newcommand{\ko}{\: , \;}

\numberwithin{equation}{subsection}
\newtheorem{theorem}{Theorem}[section]
\numberwithin{equation}{theorem}

\newtheorem{definition}[theorem]{Definition}
\newtheorem{examples}[theorem]{Examples}

\newtheorem{classification-theorem}[subsubsection]{Classification Theorem}
\newtheorem{decomposition-theorem}[subsubsection]{Decomposition Theorem}
\newtheorem{proposition-definition}[theorem]{Proposition-Definition}
\newtheorem{definition-proposition}[theorem]{Definition-Proposition}
\newtheorem{example-definition}[theorem]{Example-Definition}
\newtheorem{periodicity-conjecture}[subsubsection]{Periodicity Conjecture}
\newtheorem{lemma}[theorem]{Lemma}
\newtheorem{proposition}[theorem]{Proposition}
\newtheorem{corollary}[theorem]{Corollary}

\newtheorem{example}[theorem]{Example}
\newtheorem{remark}[theorem]{Remark}

\newtheorem{notation}[theorem]{Notation}

\newtheorem*{theoremA}{Theorem~A}
\newtheorem*{theoremB}{Theorem~B}
\newtheorem*{theoremC}{Theorem~C}
\newtheorem*{theoremD}{Theorem~D}

\newtheorem{Definition-Proposition}[theorem]{D\'efinition-Proposition}

\newcommand{\reminder}[1]{}

\renewcommand{\mod}{\mathrm{mod}}

\newcommand{\rep}{\mathrm{rep}}

\newcommand{\inj}{\mathrm{inj}\,}
\newcommand{\Mod}{\mathrm{Mod}\,}
\newcommand{\CM}{\mathrm{CM}}

\newcommand{\proj}{\mathrm{proj}\,}

\newcommand{\per}{\mathrm{per} }
\newcommand{\thick}{\mathrm{thick} }
\newcommand{\pvd}{\mathrm{pvd} }
\newcommand{\add}{\mathrm{add} }

\renewcommand{\Im}{\mathrm{Im} }

\newcommand{\tr}{\mathrm{tr}}

\newcommand{\Cone}{\mathrm{Cone}}

\newcommand{\dgcat}{\mathrm{dgcat}}
\newcommand{\Hqe}{\mathrm{Hqe}}

\renewcommand{\rep}{\mathrm{rep}}

\newcommand{\pretr}{\mathrm{pretr} }
\newcommand{\ex}{\mathrm{ex} }

\newcommand{\colim}{\mathrm{colim}}
\newcommand{\cok}{\mathrm{cok} }

\renewcommand{\ker}{\mathrm{ker} }
\newcommand{\obj}{\mathrm{obj} }

\newcommand{\Q}{\mathcal{Q}}

\newcommand{\iso}{\xrightarrow{_\sim}}

\newcommand{\Sq}{\mathrm{Sq}}

\newcommand{\Id}{\mathrm{id}}

\newcommand{\Ex}{\mathrm{Ex}}
\newcommand{\rel}{\mathrm{rel}}
\newcommand{\Ab}{\mathrm{Ab}}

\newcommand{\Def}{\mathrm{def}\kern 0.1em}
\newcommand{\ctr}{\mathrm{ctr}}
\newcommand{\dg}{\mathrm{dg}}
\renewcommand{\H}{\mathcal H}
\newcommand{\D}{\mathcal {D}}
\newcommand{\A}{\mathcal {A}}
\newcommand{\B}{\mathcal {B}}
\newcommand{\C}{\mathcal {C}}
\newcommand{\E}{\mathcal {E}}
\newcommand{\F}{\mathcal {F}}

\newcommand{\cn}{\mathrm{cn}}
\newcommand{\J}{\mathcal {J}}
\newcommand{\I}{\mathcal {I}}

\newcommand{\N}{\mathcal {N}}

\newcommand{\T}{\mathcal T}
\renewcommand{\P}{\mathcal P}
\newcommand{\X}{\mathcal X}
\newcommand{\Y}{\mathcal Y}

\renewcommand{\sp}{\mathrm{sp}}
\newcommand{\Fun}{\mathrm{Fun}}
\newcommand{\Hom}{\mathrm{Hom}}

\newcommand{\RHom}{\mathrm{RHom}}

\newcommand{\Ext}{\mathrm{Ext}}

\newcommand{\Mor}{\mathrm{Mor}}

\renewcommand{\S}{\mathcal S}

\renewcommand{\phi}{\varphi}

\renewcommand{\tilde}[1]{\widetilde{#1}}

\usetikzlibrary{positioning}

\begin{document}
\title[Exact dg categories II : The Embedding Theorem]{Exact dg categories II : The Embedding Theorem}
\author[Xiaofa Chen]{Xiaofa Chen}
\address{University of Science and Technology of China, Hefei, P.~R.~China}
\email{cxf2011@mail.ustc.edu.cn}
\date{\today}

\begin{abstract}
For an exact dg category $\A$, we introduce its bounded dg derived category $\D^b_{dg}(\A)$ and establish the universal exact morphism from $\A$ to $\D^b_{dg}(\A)$. 
We prove that the dg quotient of an exact dg category by a subcategory of projective-injectives carries a canonical exact structure. 
 We show that exact dg categories reproduce under tensor products and functor dg categories.
We apply our results to 0-Auslander extriangulated categories and confirm a conjecture by Fang--Gorsky--Palu--Plamondon--Pressland for the algebraic case.
\end{abstract}
\maketitle

\tableofcontents

\section{Introduction} 

In \cite{Chen24} the notion of exact dg category is proposed. 
Roughly speaking, an exact structure on an additive dg category $\A$
is a class $\mathcal S$ of homotopy short exact sequences in $\A$, satisfying certain axioms analogous
to those of Quillen. 
For exact dg categories $(\A,\mathcal S)$ and $(\A',\mathcal S')$,
a morphism $F:\A\rightarrow \A'$ in the localisation $\Hqe$ of the category of dg categories with respect to quasi-equivalences is exact if it sends the objects in $\mathcal S$ to objects in $\mathcal S'$.
The subcategory of $\Hqe$ consisting of exact dg categories with exact morphisms is denoted by $\Hqe_{\ex}$.
The main motivation is to provide dg enhancements for Nakaoka--Palu's extriangulated categories. 
One of the main results in~\cite{Chen24} is that for an exact dg category $(\A,\mathcal S)$, the category $H^0(\A)$ carries a canonical extriangulated structure $(H^0(\A),\mathbb E,\mathfrak s)$, cf.~\cite[Theorem 4.26]{Chen24}.
 In analogy with the notion of algebraic triangulated category \cite{Keller06d},  an extriangulated category of this form will be called {\em algebraic}.
  
In this paper, we prove several theorems concerning exact dg categories.
Our first result relates exact dg categories to pretriangulated dg categories.
 \begin{theoremA}[Theorem~\ref{main}]\label{intro:main}
Let $(\A,\S)$ be a small exact dg category. 
There exists a universal exact morphism $F:\A\rightarrow \D^b_{\dg}(\A,\S)$
in $\Hqe$ from $\A$ to a pretriangulated dg category $\D^b_{\dg}(\A,\S)$. 
If moreover $\A$ is connective, this morphism satisfies:
\begin{itemize}
\item[1)]It induces a quasi-equivalence from $\tau_{\leq 0}{\A}$ to $\tau_{\leq 0}\D'$ for an extension-closed dg subcategory $\D'$ of $\D^b_{\dg}(\A,\S)$;
\item[2)]It induces a natural bijection $\mathbb E(C,A)\xrightarrow{\sim} \Ext^1_{\D^b(\A,\S)}(FC,FA)$ for each pair of objects $C,A$ in $H^0(\A)$ where $\D^b(\A,\S)=H^0(\D^b_{\dg}(\A,\S))$.
\end{itemize}
We call $\D^b_{\dg}(\A,\S)$ the {\em bounded dg derived category} of $(\A,\S)$. 
\end{theoremA}
We construct $\D^b_{\dg}(\A,\S)$ as the dg quotient
of the pretriangulated hull $\pretr(\A)$ by a full dg subcategory $\N$, which generalises the dg category of acyclic
complexes over a Quillen exact category. 
As a consequence, for a connective exact dg category $(\A,\mathcal S)$, the morphism $F:\A\rightarrow \D^b_{\dg}(\A,\S)$ induces a canonical isomorphism $K_{0}(H^0(\A),\mathbb E,\mathfrak s)\iso K_0(\D^b(\A,\S))$ between Grothendieck groups, cf.~Proposition \ref{prop:Grothendieckgroups}.
Based on Theorem~\hyperref[intro:main]{A}, we obtain a characterisation of algebraic extriangulated categories, cf.~Definition-Proposition~\ref{algebraic}.
 Using a general result concerning $\delta$-functors with weakly effaceable bimodules, cf.~Corollary~\ref{cor:effaceablebimodule}, we obtain that the morphism $F$ also induces natural bijections between higher extension groups, cf.~Proposition~\ref{higher}.


For an extriangulated category $\C$ with a subcategory $\P$ consisting of (not necessarily all) projective-injective objects in $\C$, 
Nakaoka--Palu showed that the ideal quotient $\C/[\P]$ has the structure of an extriangulated category, 
induced from that of $\C$, cf.~\cite[Proposition 3.30]{NakaokaPalu19}. 
Our second result enhances theirs. 
\begin{theoremB}[Theorem~\ref{quot}]\label{intro:quot}
Let $(\A,\S)$ be a small connective exact dg category and $\P$ a full dg subcategory of $\A$ consisting of projective-injective objects in $\A$.
Let $\mathcal T_{\dg}$ be the canonical dg enhancement of $\D^b(\A,\S)/\tr(\P)$.
Then the dg quotient $\A/\P$ carries a canonical exact dg structure $(\A/\P,\overline{\S})$ induced from that of $\A$ 
and its dg derived category is quasi-equivalent to $\mathcal T_{\dg}$.
\end{theoremB}
As a consequence, we obtain that for a Quillen exact category $\E$ and a full subcategory $\P_0$ of projective-injectives in $\E$, 
the dg quotient $\E/\P_0$ carries a canonical exact structure.
In~\cite[7.3.15]{Cisinski19}, it is proven that every $\infty$-category is equivalent to the localization of the nerve of a small category by a subcategory.
We have an analogous result in the context of exact dg categories: each connective exact dg category is exactly quasi-equivalent to 
the dg quotient $\E/\P_0$ of a Quillen exact category $\E$ by a full dg subcategory $\P_0$ of projective-injectives in $\E$, cf.~Corollary~\ref{cor:connectivequotient}.


Let $(\A,\mathcal S)$ and $(\B,\mathcal S')$ be small exact dg categories. 
Let $(\C,\S'')$ be another small exact dg category. 
A morphism $\mu: \A\otimes \B\rightarrow \C$ in $\Hqe$ is {\em biexact}, if for each object $A\in \A$ and $B\in \B$, the induced morphisms 
\[
\mu_{A,-}:\B\rightarrow \C,\;\;\mu_{-,B}:\A\rightarrow \C
\]
are both exact.
A famous slogan \cite[Remark 16]{Neeman20} by Amnon Neeman is: triangulated categories do not reproduce. 
Our third result reveals a key feature of exact dg categories.
\begin{theoremC}[Theorem~\ref{fun}, Proposition~\ref{prop:universalbilinear}]\label{intro:fun}
Exact dg categories do reproduce under tensor products and functor dg categories. 
More precisely:
\begin{itemize}
\item[1)] If $(\B,\S')$ is connective, then $\rep_{\dg}(\B,\C)$ carries a canonical exact structure. 
\item[2)]If both $(\A,\S)$ and $(\B,\S')$ are connective, then there exists the universal biexact morphism $\A\otimes \B\rightarrow \A\boxtimes \B$.
\end{itemize}
\end{theoremC}
Let $\Hqe_{\ex}^{\cn}$ denote the full subcategory of $\Hqe_{\ex}$ consisting of connective exact dg categories.
Building upon Theorem~\hyperref[intro:fun]{C}, we show that a certain monoidal structure on $\Hqe_{\ex}^{\cn}$ is closed, cf.~Corollary~\ref{cor:exactadjunction}.

 Let $(\C,\mathbb E,\mathfrak s)$ be an extriangulated category.
 We denote by $\proj \C$ respectively $\inj \C$ its full subcategory of projectives respectively injectives.
 If the extriangulated category $\C$ is algebraic, then the ideal quotient $\C/[\inj \C\rightarrow\proj \C]$
is also algebraic, cf.~Proposition~\ref{prop:idealquotientinjectiveprojective}. 
When $\C$ is furthermore {\em $0$-Auslander} in the sense of \cite{GorskyNakaokaPalu23}, 
our last result confirms a conjecture by Fang--Gorsky--Palu--Plamondon--Pressland for
the algebraic case.
\begin{theoremD}[Theorem~\ref{thm:FGPPP}]\label{intro:thmD}
For each algebraic $0$-Auslander extriangulated category $\C$, we have
an equivalence of extriangulated categories
\[
\C/[\inj \C\rightarrow\proj \C]\iso \H^{[-1,0]}(\proj \C/[\I])
\]
where $\I$ is the full subcategory of $\C$ consisting of projective-injectives.
\end{theoremD}
The key ingredient in the proof of Theorem~\hyperref[intro:thmD]{D}  is the $0$-Auslander correspondence between the equivalence classes of certain pairs of connective dg categories and the equivalence classes of $0$-Auslander exact dg categories, cf.~Theorem~\ref{thm:0-Auslander correspondence}.

The paper is structured as follows. 
In Section~\ref{sec:prelim}, we collect basic results about extriangulated categories, 
and some basic notations and terminology needed in this paper. 
We prove a general result concerning $\delta$-functors with weakly effaceable bimodules, cf.~Corollary~\ref{cor:effaceablebimodule}. 
It is a key ingredient in comparing higher extension groups in Proposition~\ref{higher}.

In Section~\ref{sec:embedding}, we prove the embedding theorem for exact dg categories.
We employ it to provide a characterisation of algebraic extriangulated categories, cf.~Proposition-Definition~\ref{algebraic}.
Additionally, we  study the relation between higher extensions.
Furthermore, we study the dg quotients of connective exact dg categories by a class of projective-injective objects. 
We show that each connective exact dg category can be obtained  in this manner from the dg quotient of a Quillen exact category.
We end this section with several classes of examples including stable dg categories, Cohen--Macaulay dg modules and Higgs categories.
Relative cluster categories are then described as the bounded derived categories of the corresponding Higgs dg categories.

In Section~\ref{sec:reproduction}, we show that ``functor dg categories" with exact target naturally inherit exact structures from the target.
Building upon this, we show that a specific monoidal structure on the category of small connective exact dg categories is closed.
In Section~\ref{sec:applications}, we apply our theoretic framework to study the class of $0$-Auslander extriangulated categories.
Within this context, we resolve  the algebraic case of a conjecture posed by Fang--Gorsky--Palu--Plamondon--Pressland.
\subsection*{Acknowledgement}
This paper is a revised and augmented version of the second part of the author's Ph.D thesis \cite{Chen23}. 
The author is grateful to his Ph.D supervisor Bernhard Keller and cosupervisor Xiao-Wu Chen for their constant support and interest in this work.
 He thanks Matthew Pressland for his talk at the ARTA IX on his joint work with Xin Fang, Mikhail Gorsky, Yann Palu, and Pierre-Guy Plamondon.
He thanks Mikhail Gorsky, Gustavo Jasso, Yann Palu, Matthew Pressland, Xiao-Wu Chen, and Yilin Wu for helpful discussions and useful comments. 
While writing this paper, the author was financially supported by Xiaomi Youth Scholar.

\section{Preliminaries}\label{sec:prelim}
In this section, we collect basic results on extriangulated categories, 
and some basic notations and terminology needed in this paper. 
 For more details on extriangulated category,
we refer to \cite{NakaokaPalu19, Palu23,GorskyNakaokaPalu21, GorskyNakaokaPalu23}.
\subsection{$\delta$-functors in extriangulated categories}
Let $(\mathcal C,\mathbb E,\mathfrak s)$ be a skeletally small extriangulated category. 
\begin{lemma}[\cite{LiuNakaoka19}, Proposition 1.20]\label{LN}
Let  
\[
\begin{tikzcd}
A\ar[r,"x"]&B\ar[r,"y"]&C\ar[r,dashed,"\delta"]&\ 
\end{tikzcd}
\]
be any $\mathbb E\mbox{-}$triangle, let $f:A\rightarrow D$ any morphism, and let
\[
\begin{tikzcd}
D\ar[r,"d"]&E\ar[r,"e"]&C\ar[r,dashed,"f_{*}\delta"]&\ 
\end{tikzcd}
\]
be any $\mathbb E\mbox{-}$triangle realizing $f_{*}\delta$.
Then there is a morphism $g$ which gives a morphism of $\mathbb E\mbox{-}$triangles
\[
\begin{tikzcd}
A\ar[r,"x"]\ar[d,"f"swap]&B\ar[r,"y"]\ar[d,"g"red,dashed]&C\ar[d,equal]\ar[r,dashed,"\delta"]&\ \\
D\ar[r,"d"swap]            &E\ar[r,"e"swap]              &C\ar[r,dashed,"f_{*}\delta"swap]           &\ 
\end{tikzcd}
\]
and moreover, the sequence
\[
\begin{tikzcd}
A\ar[r,"\begin{bmatrix} {-}f \\ x \end{bmatrix}"] & D\oplus B\ar[r,"{[}d{,}\; g{]}"]&E\ar[r,dashed,"e^{*}\delta"]&\ 
\end{tikzcd}
\]
becomes an $\mathbb E\mbox{-}$triangle.
\end{lemma}

The following diagram lemma presents a slightly stronger version of \cite[Proposition 3.15]{NakaokaPalu19}.
It is often used as a replacement for the non-existent push-outs or pull-backs in extriangulated categories \cite{Palu23}.
\begin{lemma}\label{lem:diagramlemmaextriangulated4}
Suppose we are given $\mathbb E\mbox{-}$triangles
\[
\begin{tikzcd}
A_1\ar[r,"x_1"]&B_1\ar[r,"y_1"]&C\ar[r,dashed,"\delta"]&,
\end{tikzcd}
\]
\[
\begin{tikzcd}
A_2\ar[r,"x_2"]&B_2\ar[r,"y_2"]&C\ar[r,dashed,"\delta'"]&.
\end{tikzcd}
\]
Then there is a commutative diagram in $\mathcal C$: 
\[
\begin{tikzcd}
                                               &A_2\ar[r,equal]\ar[d,"m_2"]          &A_2\ar[d,"x_2"]\\
A_1\ar[r,"m_1"]\ar[d,equal]     &M\ar[r,"e_1"]\ar[d,"e_2"]              &B_2\ar[d,"y_2"]\\
A_1\ar[r,"x_1"swap]                        &B_1\ar[r,"y_1"swap]                             &C
\end{tikzcd}
\]
which satisfies
\[
\mathfrak s(y_2^{*}\delta)=[A_1\xrightarrow{m_1} M\xrightarrow{e_1} B_2],
\]

\[
\mathfrak s(y_1^{*}\delta')=[A_2\xrightarrow{m_2} M\xrightarrow{e_2} B_1],
\]
\[
(m_1)_{*}(\delta)+(m_2)_{*}(\delta')=0
\]
and moreover, the sequence
\[ 
\begin{tikzcd}
M\ar[r,"\begin{bmatrix}{e_1}\\e_2\end{bmatrix}"]&B_1\oplus B_2\ar[r,"{[}y_1{,}-y_2{]}"]&C\ar[r,dashed,"(m_1)_*(\delta)"]&\,
\end{tikzcd}
\]
is an $\mathbb E\mbox{-}$triangle.
\end{lemma}
\begin{proof}
We apply the dual of Lemma \ref{LN} to the diagram
\[
\begin{tikzcd}
A_1\ar[d,equal]\ar[r,"m_1"]&M\ar[r,"e_1"]\ar[d,dashed,"e_2"]&B_2\ar[r,dashed,"y_2^{*}(\delta)"]\ar[d,"y_2"]&\,\\
A_1\ar[r,"x_1"swap]&B_1\ar[r,"y_1"swap]&C\ar[r,dashed,"\delta"swap]&\,
\end{tikzcd}
\]
Then there exists a morphism $e_2:M\rightarrow B_1$ such that the above diagram commutes in $\mathcal C$ and
\[
\begin{tikzcd}
M\ar[r,"\begin{bmatrix}{e_2}\\e_1\end{bmatrix}"]&B_1\oplus B_2\ar[r,"{[}y_1{,}-y_2{]}"]&C\ar[r,dashed,"(m_1)_*(\delta)"]&\,
\end{tikzcd}
\]
is an $\mathbb E\mbox{-}$triangle.

We then apply \cite[Proposition 3.17]{NakaokaPalu19} to the following diagram
\[
\begin{tikzcd}
A_2\ar[r,"x_2"]\ar[d,"m_2"{swap},dashed]&B_2\ar[r,"-y_2"]\ar[d,"\begin{bmatrix}0\\1\end{bmatrix}"]                    &C\ar[d,equal]\\
M\ar[r,"\begin{bmatrix}e_2\\e_1\end{bmatrix}"]\ar[d,"k"{swap},dashed]                         &B_1\oplus B_2\ar[r,"{[}y_1{,}-y_2{]}"] \ar[d,"{[}1{,}0{]}"]             &C\\
B_1\ar[r,equal]&B_1&
\end{tikzcd}.
\]
So we have an $\mathbb E\mbox{-}$triangle
\[
\begin{tikzcd}
A_2\ar[r,"m_2"]&M\ar[r,"k"]&B_1\ar[r,dashed,"\theta"]&\,
\end{tikzcd}
\]
such that the above diagram commutes (so we have $k=e_2$) and 
\begin{itemize}
\item[(i)] $(m_2)_{*}(-\delta')=(m_1)_*(\delta)$,
\item[(ii)] $(x_2)_{*}(\theta)=0$,
\item[(iii)] $[1,0]^{*}(\theta)+[y_1,-y_2]^{*}(-\delta')=0$.
\end{itemize}
Applying the functor $\begin{bmatrix}1\\0\end{bmatrix}^{*}$ to the last equality, we obtain $\theta=y_1^*(\delta')$.
\end{proof}

\begin{definition}
A {\em contravariant connected sequence of functors} is a pair $(T,\epsilon)$ where $T=(T^i)_{i\geq 0}$ is sequence of right $\C$-modules and $\epsilon$ is a collection of morphisms $\epsilon_{\delta}^i:T^i(A)\rightarrow T^{i+1}(C)$ for each $\mathbb E\mbox{-}$extension $\delta\in\mathbb E(C,A)$ and $i\geq 0$ which is natural with respect to morphisms of $\mathbb E$-extensions.
It is a {\em right $\delta$-functor} if for any $\mathfrak s\mbox{-}$triangle 
\[
\begin{tikzcd}
A\ar[r,"k"]&B\ar[r,"p"] &C\ar[r,dashed,"\delta"]&\;
\end{tikzcd},
\]
 the associated sequence (which is a complex by the naturality of $\epsilon$, cf.~\cite[Proposition 3.20]{GorskyNakaokaPalu21})
\[
\ldots \rightarrow T^{n}(C)\xrightarrow{T^n(p)} T^n(B)\xrightarrow{T^n(k)}T^n(A)\xrightarrow{\epsilon^{n}_{\delta}} T^{n+1}(C)\rightarrow \ldots 
\] 
is exact.
A morphism $(T,\epsilon)\rightarrow (\tilde{T},\tilde{\epsilon})$ of right $\delta$-functors is a family of morphisms of right $\C$-modules $\theta^i:T^i\rightarrow \tilde{T}^i$ which is compatible with the connecting morphisms, i.e.,~for each $\delta\in \mathbb E(C,A)$ and each $i\geq 0$, the following diagram is commutative
\[
\begin{tikzcd}
T^{i}(A)\ar[r,"\epsilon^i_{\delta}"]\ar[d,"\theta^i(A)"swap]&T^{i+1}(C)\ar[d,"\theta^{i+1}(C)"]\\
\tilde{T}^i(A)\ar[r,"\tilde{\epsilon}^i_{\delta}"swap]&\tilde{T}^{i+1}(C)
\end{tikzcd}.
\]
Dually one defines the notion of {\em covariant connected sequence of functors} and {\em left $\delta$-functors}.
\end{definition}

We have the following extriangulated analogue of \cite[Proposition 2.1]{Grothendieck57}.
\begin{proposition}\label{prop:effaceableuniversal}
Let $(T,\epsilon)$ and $(R,\eta)$ be right $\delta$-functors such that $R^i$ is weakly effaceable for each $i>0$.
Then each morphism $R^0\rightarrow T^0$ extends uniquely to a morphism of right $\delta$-functors $(R,\eta)\rightarrow (T,\epsilon)$.
\end{proposition}
\begin{proof}
Suppose we have constructed natural morphisms $\theta^j:R^j\iso T^j$ which are compatible with the connecting morphisms for $0\leq j<i$.

Let $C$ be an object in $\C$ and $x\in R^i(C)$ an arbitrary element.
Since $R^i$ is effaceable, there exists a $\mathbb E$-triangle 
\[
\begin{tikzcd}
A\ar[r,"k"]&B\ar[r,"p"] &C\ar[r,dashed,"\delta"]&\;
\end{tikzcd},
\]
 such that $R^i(p)(x)=0$.
We have the following diagram
\[
\begin{tikzcd}
R^{i-1}(B)\ar[r,"R^{i-1}(k)"]\ar[d,"\theta^{i-1}_{B}"swap]&R^{i-1}(A)\ar[r,"\eta_{\delta}^{i-1}"]\ar[d,"\theta^{i-1}_{A}"swap]&R^i(C)\ar[r,"R^i(p)"]\ar[d,dashed]&R^i(B)\\
T^{i-1}(B)\ar[r,"T^{i-1}(k)"swap]&T^{i-1}(A)\ar[r,"\epsilon_{\delta}^{i-1}"swap]&T^i(C)\ar[r]&T^i(B)\\
\end{tikzcd}.
\]
By the exactness of the top horizontal sequence, there exists an element $y\in R^{i-1}(A)$ such that $\eta_{\delta}^{i-1}(y)=x$.
Since the left square commutes by the construction of $\theta^{i-1}$, it is clear that the element $\epsilon_{\delta}^{i-1}\circ \theta_{A}^{i-1}(y)$ is independent of the choice of $y$. 
Let 
\[
\begin{tikzcd}
A'\ar[r,"l"]&B'\ar[r,"q"] &C\ar[r,dashed,"\delta'"]&\;
\end{tikzcd}
\] 
be another such $\mathbb E$-triangle. 
By Lemma~\ref{lem:diagramlemmaextriangulated4}, we have the following diagram
\[
\begin{tikzcd}
&A'\ar[r,equal]\ar[d,"f"]&A'\ar[d,"l"]&\;\\
A\ar[r,"m"]\ar[d,equal]&E\ar[r,"r"]\ar[d,"g"]&B'\ar[d,"q"]\ar[r,dashed,"q^{*}(\delta)"]&\;\\
A\ar[r,"k"swap]&B\ar[r,"p"swap] &C\ar[r,dashed,"\delta"swap]&\;
\end{tikzcd}
\]
where $m_*(\delta)=-f_*(\delta')$ and the sequence
\[
\begin{tikzcd}
E\ar[r,"\begin{bmatrix}r\\g\end{bmatrix}"]&B'\oplus B\ar[r,"{[}-q{,}\;p{]}"] &C\ar[r,dashed,"m_{*}(\delta)"]&\;
\end{tikzcd}
\] 
is an $\mathbb E$-triangle. So we have the following diagram made of $\mathbb E$-triangles
\[
\begin{tikzcd}
A\ar[r,"k"]\ar[d,"m"swap]&B\ar[r,"p"] \ar[d,"\begin{bmatrix}1\\0\end{bmatrix}"swap]&C\ar[r,dashed,"\delta"]\ar[d,equal]&\;\\
E\ar[r,"\begin{bmatrix}g\\r\end{bmatrix}"swap]&B\oplus B'\ar[r,"{[}-q{,}\;p{]}"swap] &C\ar[r,dashed,"m_{*}(\delta)"swap]\ar[d,equal]&\;\\
A'\ar[r,"l"swap]\ar[u,"-f"]&B'\ar[r,"q"swap] \ar[u,"\begin{bmatrix}0\\-1\end{bmatrix}"]&C\ar[r,dashed,"\delta'"swap]&\;
\end{tikzcd}
\]
This implies that $\epsilon_{\delta}^{i-1}\circ \theta_{A}^{i-1}(y)$ is independent of the choice of the $\mathbb E$-triangle. 
Similarly as above, one shows that the map $\theta^{i}_{C}: R^{i}(C)\rightarrow T^{i}(C)$ is a morphism of abelian groups.
It is direct to verify that $\theta^i:R^i\rightarrow T^{i}$ is a natural transformation 
and it follows directly from the construction that $\theta^i$ is compatible with the connecting morphisms. 
By induction on $i$, we have natural transformations $\theta^i:R^i\rightarrow T^i$ for $i\geq 0$ which are compatible with the connecting morphisms. The uniqueness is also clear from the construction of the $\theta^i$.
\end{proof}

\begin{definition}[\cite{GorskyNakaokaPalu21}, Definitions 4.5]\label{def:deltafunctor}
A {\em $\delta$-functor} is a triple $(T,\epsilon,\eta)$ where $T=(T^i)_{i\geq 0}$ be a sequence of $\C$-$\C$-bimodules
and $\epsilon$ and $\eta$ are collections of morphisms $\epsilon_{\delta}^i:T^i(A,-)\rightarrow T^{i+1}(C,-)$ and $\eta_{\delta}^i:T^i(?,C)\rightarrow T^{i+1}(?,A)$ for each $\mathbb E$-extension $\delta\in\mathbb E(C,A)$ and $i\geq 0$ which are natural with respect to morphisms of $\mathbb E$-extensions, and such that for each $\mathfrak s$-triangle
\[
\begin{tikzcd}
A\ar[r,"k"]&B\ar[r,"p"] &C\ar[r,dashed,"\delta"]&\;
\end{tikzcd},
\]
 the associated complexes
 \[
\ldots \rightarrow T^{n}(C,-)\xrightarrow{T^n(p,-)} T^n(B,-)\xrightarrow{T^n(k,-)}T^n(A,-)\xrightarrow{\epsilon^{n}_{\delta}} T^{n+1}(C,-)\rightarrow \ldots 
\] 
\[
\ldots \rightarrow T^{n}(?,A)\xrightarrow{T^n(?,k)} T^n(?,B)\xrightarrow{T^n(?,p)}T^n(?,C)\xrightarrow{\eta^{n}_{\delta}} T^{n+1}(?,A)\rightarrow \ldots 
\] 
are exact in $\Mod \C$, resp.,~in $\C\Mod$. 
A morphism $(T,\epsilon,\eta)\rightarrow (\tilde{T},\tilde{\epsilon},\tilde{\eta})$ of $\delta$-functors is a family of morphisms of $\C$-$\C$-bimodules $\theta^i:T^i\rightarrow \tilde{T}^i$ which is compatible with the connecting morphisms.
\end{definition}

Let $(T,\epsilon,\eta)$ a $\delta$-functor. 
We assume that
\begin{equation}\label{assumption:delta1}
 \text{$T^0=\Hom(?,-)$ and $T^i$ are effaceable for $i>0$.}
 \end{equation}

By Proposition~\ref{prop:effaceableuniversal}, we have that $T^i$ is isomorphic to $\mathbb E^i$ for $i>0$ which is compatible with the connecting morphisms $\epsilon$.
We see that $\epsilon_{\delta}^0:\Hom(A,-)\rightarrow \mathbb E(C,-)$ is given by 
\[
(\epsilon_{\delta}^0)_U:\Hom(A,U)\rightarrow \mathbb E(C,U),\;\; f\mapsto f_{*}\delta
\]
for each object $U\in\C$. 
\begin{equation}\label{assumption:delta2}
\text{We assume that a similar description holds for $\eta_{\delta}^0$.}
\end{equation}  
For each $i\geq 2$ and any two $\mathfrak s$-triangles
\[
\begin{tikzcd}
A\ar[r,"k"]&B\ar[r,"p"] &C\ar[r,dashed,"\delta"]&\;
\end{tikzcd},
\]
\[
\begin{tikzcd}
F\ar[r,"l"]&G\ar[r,"q"] &H\ar[r,dashed,"\psi"]&\;
\end{tikzcd}
\]
\begin{equation}\label{assumption:delta3}
\text{we assume the following diagram~(\ref{dia:bivariantdeltafunctor}) is commutative}
\end{equation}
\begin{equation}\label{dia:bivariantdeltafunctor}
\begin{tikzcd}
 T^{i-2}(A,H)\ar[r,"\epsilon_{\delta}^{i-2}"]\ar[d,"\eta_{\psi}^{i-2}"swap]&T^{i-1}(C,H)\ar[d,"\eta_{\psi}^{i-1}"]\\
T^{i-1}(A,F)\ar[r,"\epsilon_{\delta}^{i-1}"swap]&T^i(C,F)
\end{tikzcd}\;.
\end{equation}

Put $\mathbb E^0=\Hom_{\C}(?,-)$. 
From the definition of $\mathbb E^n$ and the construction of the connecting morphisms, cf.~\cite[Claim 3.11]{GorskyNakaokaPalu21}, it is clear that $(\mathbb E^n)_{n\geq 0}$ together with its connecting morphisms satisfy the above assumptions~(\ref{assumption:delta1}--\ref{assumption:delta3}).
\begin{corollary}\label{cor:effaceablebimodule}
Keep the notations as above. 
Let $(R,\gamma,\mu)$ be another $\delta$-functor which satisfies the above assumptions (\ref{assumption:delta1}--\ref{assumption:delta3}). 
 Then there is a unique isomorphism of $\delta$-functors $\theta^i: R^i\rightarrow T^i$
 for $i\geq 0$ such that $\theta^0=\Id_{\Hom_{\C}(?,-)}$.
\end{corollary}
\begin{proof}
By Proposition~\ref{prop:effaceableuniversal}, we have that for each $i\geq 0$, the right $\C$-module $R^i(?,A)$ is naturally isomorphic to $T^i(?,A)$ in $\Mod \C$ and the isomorphism is natural in the variable $A$ and compatible with the connecting morphisms. 
Hence we have an isomorphism of $\C$-$\C$-bimodules $\theta^i: R^i\rightarrow T^i$ for each $i\geq 0$.
Similarly $R^i(C,-)$ is naturally isomorphic to $T^i(C,-)$ and this also gives an isomorphism of $\C$-$\C$-bimodules $\lambda^i:R^i\rightarrow T^i$ for each $i\geq 0$.

We identify both $R^1$ and $T^1$ with $\mathbb E$ and by assumption the natural isomorphisms $\theta^1$ and $\lambda^1$ are both given by the identity.
We only need to check that the natural isomorphisms $\theta^i$ and $\lambda^i$ coincide with each other for each $i\geq 2$ and we proceed by induction on $i$.

Let $C$ and $F$ be objects in $\C$ and $x\in R^{i}(C,F)$ for $i\geq 1$.
Since $R^i$ and $R^{i-1}$ are both effaceable, there exists $\mathfrak s$-triangles
\[
\begin{tikzcd}
A\ar[r,"k"]&B\ar[r,"p"] &C\ar[r,dashed,"\delta"]&\;
\end{tikzcd},
\]
\[
\begin{tikzcd}
F\ar[r,"l"]&G\ar[r,"q"] &H\ar[r,dashed,"\psi"]&\;
\end{tikzcd}
\]
such that there exists an element $y\in R^{i-2}(A,H)$ with $x=\mu_{\psi}^{i-1}\gamma_{\delta}^{i-2}(y)=\gamma_{\delta}^{i-1}\mu_{\psi}^{i-2}(y)$ where the second equality is given by the commutativity of the diagram~\ref{dia:bivariantdeltafunctor} for the bimodules $R^{i}$, $R^{i-1}$ and $R^{i-2}$.
Since by induction $\theta^{j}$ coincides with $\lambda^{j}$ for $j=i-2$ and $i-1$, we obtain that the isomorphisms $\theta^{i}$ and $\lambda^i$ coincide with each other and this finishes the proof. 
\end{proof}

Let $(\C,\mathbb E,\mathfrak s)$ be a skeletally small extriangulated category.
The underlying category $\C$ being additive, one has the {\em split Grothendieck group} $K_{0}^{\sp}(\C)$ defined as the quotient of the free abelian group generated by the isomorphism classes $[X]$ of objects in $\C$ by the subgroup generated by the elements $[X]-[Y]+[Z]$ for each split short exact sequences
\[
\begin{tikzcd}
X\ar[r,tail]&Y\ar[r,two heads]&Z
\end{tikzcd}
\]
in $\C$. The image of $[X]$ in $K_0^{\sp}(\C)$ is denoted by $[X]^{\sp}$.
\begin{definition}[\cite{ZhuZhuang21}]
The {\em Grothendieck group} $K_0(\C,\mathbb E,\mathfrak s)$ of the extriangulated category $(\C,\mathbb E,\mathfrak s)$ is the abelian group
\[
K_0(\C,\mathbb E,\mathfrak s)\coloneqq K_0^{\sp}(\C)/\langle [X]^{\sp}-[Y]^{\sp}+[Z]^{\sp}\mid \text{(\ref{conf:C}) is an $\mathfrak s$-triangle} \rangle
\]
\begin{equation}\label{conf:C}
\begin{tikzcd}
X\ar[r,tail]&Y\ar[r,two heads]&Z.
\end{tikzcd}
\end{equation}
\end{definition}
	\subsection{Recollections on dg categories}\label{subsection:notations}
	In this subsection, we collect basic notations and terminology needed in this paper. 
	The standard references for dg categories are \cite{Keller94, Keller06d, Drinfeld04,Toen11, BondalKapranov90}.

Throughout we fix a commutative ring $k$. 
	We write $\otimes$ for the tensor product over $k$.  
	We denote by $\C_{\dg}(k)$ the dg category of complexes of $k$-modules and by $\C(k)$ the category of complexes of $k$-modules. 
	Let $M$ be complex of $k$-modules. We put
	\[
	\begin{tikzcd}
	\tau_{\leq 0}M=(\cdots\ar[r]&M^{-2}\ar[r]&M^{-1}\ar[r]&Z^0M\ar[r]&0\ar[r]&\cdots).
	\end{tikzcd}
	\] 
                 For a dg category $\A$, denote by $\tau_{\leq 0}\A$ the dg category with the same objects as $\A$ 
                 and whose morphism complexes are given by
                 \[
                 (\tau_{\leq 0} \A)(A_1, A_2) = \tau_{\leq 0}(\A(A_1, A_2)).
                  \]
                 The composition is naturally induced by that of $\A$. 
A dg category $\A$ is {\em connective} if for each pair of objects $A_1, A_2\in \A$, the 
 complex $\Hom_{\A}(A_1, A_2)$  has cohomology concentrated in non-positive degrees.
It is called {\em strictly connective} if the components of the complex $\Hom_{\A}(A_1, A_2)$ vanish
in all positive degrees.
Following To\"en, the dg category $\tau_{\leq 0}\A$ is called the {\em connective cover} of $\A$.

	Let $\A$ be a dg $k$-category.
	For two objects $A_1$ and $A_2$, the Hom complex is denoted by $\Hom_{\A}(A_1,A_2)$ or 
	$\A(A_1,A_2)$.	An element $f$ of $\A(A_1,A_2)^{p}$ will be called a {\em homogeneous} morphism of degree $p$ with the notation $|f|=p$. 
	A homogeneous morphism $f:A_1\rightarrow A_2$ is {\em closed} if we have $d(f)=0$.
	
	We denote by $Z^0(\A)$ the category with the same objects as $\A$ and whose morphism
spaces are defined by
	\[
	(Z^0\A)(A_1,A_2)=Z^0(\A(A_1,A_2)),
	\]
 where $Z^0$ is the kernel of $d:\A(A_1,A_2)^0\rightarrow \A(A_1,A_2)^1$.
 Similarly, we denote by $H^0(\A)$ the category with the same objects as $\A$ and whose 
 morphism spaces are given by 
 \[
 (H^0\A)(A_1,A_2)=H^0(\A(A_1,A_2)),
 \]
 where $H^0$ denotes the zeroth cohomology of the complex.
 A dg category $\A$ is {\em additive} if $H^0(\A)$ is additive as a category.
 
	The {\em oppositie dg category} $\A^{op}$ has the same objects as $\A$ and its morphism spaces are defined by
	\[
	\A^{op}(X,Y)=\A(Y,X);
	\] 
                 the composition of $f\in\A^{op}(Y,X)^{p}$ with $g\in \A^{op}(Z,Y)^{q}$ is given by $(-1)^{pq}gf$.
	By a dg $\A$-module $M$, we mean a right dg $\A$-module, i.e.,~a dg functor $\A^{op}\rightarrow \C_{\dg}(k)$. 
	We denote by $\C_{\dg}(\A)$ the dg category of right dg $\A$-modules. The category of dg 
	$\A$-modules is
	\[
	\C(\A)=Z^0(\C_{\dg}(\A)).
	\] 
	The category up to homotopy of dg $\A$-modules is 
	\[
	\mathcal H(\A)=H^0(\C_{\dg}(\A)).
	\]

For each object $A\in\A$, we have the right dg module {\em represented by} $A$
 \[
 A^{\wedge}=\Hom_{\A}(-,A).
 \]
 A dg $\A$-module $M$ is {\em representable} if it is isomorphic to $A^\wedge$ for some
 $A\in\A$ and {\em quasi-representable} if it is quasi-isomorphic to $A^\wedge$ for some $A \in \A$.
 The {\em Yoneda dg functor}
 \[
 \A\rightarrow \C_{\dg}(\A),\;\;f:A\rightarrow A' \mapsto f^{\wedge}:A^{\wedge}\rightarrow A'^{\wedge}
 \]
 is fully faithful.
 
 A dg functor $F:\A\rightarrow \B$ is a {\em quasi-equivalence} if
 \begin{itemize}
 \item[a)] it is {\em quasi-fully faithful}, i.e.,~for all objects $A, A'\in A$, the morphism
  \[
 F_{A,A'}:\A(A,A')\rightarrow \B(FA,FA').    
 \] 
 is a quasi-isomorphism, and
 \item[b)] the induced functor  $H^0(F):H^0(\A)\rightarrow H^0(\B)$ is an equivalence of categories.
 \end{itemize}
The category $\dgcat$ of small dg categories admits the Dwyer-Kan model structure (cf.~\cite{Tabuada05}), 
 whose weak equivalences are the quasi-equivalences. Its homotopy category is denoted by $\Hqe$. 
 There exists a cofibrant replacement functor $Q$ on $\dgcat$ such that for any $\A\in\dgcat$, the natural dg functor $Q(\A)\rightarrow \A$ is the identity on the set of objects. 
 In particular, the Hom-complexes of $Q(\A)$ are cofibrant over $k$.
 The {\em tensor product} $\A\otimes \B$ of two dg categories $\A$ and $\B$ has the class of objects  $\obj(\A)\times \obj(\B)$ and the morphism spaces
 \[
 (\A\otimes\B)((A,B),(A',B'))=\A(A,A')\otimes \B(B,B')
 \]
 with the natural compositions and units. 
 This defines a symmetric monoidal structure $-\otimes -$ on $\dgcat$ which is closed.
 For $\A, \B\in\dgcat$, put $\A\otimes^{\mathbb L}\B=\A\otimes Q(\B)$. 
 This extends to a bifunctor $-\otimes^{\mathbb L}-:\dgcat\times \dgcat\rightarrow \dgcat$ and then passes through the homotopy categories
 \[
 -\otimes^{\mathbb L}-:\Hqe\times \Hqe\rightarrow \Hqe.
 \]

For a dg category $\A$, we denote by $\D(\A)$ its {\em derived category}, a triangulated category. 
 By definition, $\D(\A)$ is the localization of $\C(\A)$ at the class of {\em quasi-isomorphisms}, i.e.,~morphisms of dg $\A$-modules which induce quasi-isomorphisms of complexes when evaluated at any object in $\A$.
 Let $\pi:\C(\A)\rightarrow \D(\A)$ be the quotient functor. 
 For a morphism $j:M\rightarrow N$ in $\C(\A)$, we denote by $\overline{\jmath}=\pi(j):M\rightarrow N$ the corresponding morphism in $\D(\A)$.
 A dg functor $F:\A\rightarrow \B$ induces a triangle functor $F_*:\D(\A)\rightarrow \D(\B)$ which is an equivalence of triangulated categories if $F$ is a quasi-equivalence.
 The {\em dg derived category} $\D_{\dg}(\A)$ of $\A$ is defined to be the full dg subcategory of $\C_{\dg}(\A)$ consisting of cofibrant dg $\A$-modules in the {\em projective model structure} of $\C(\A)$ (cf.~\cite[Theorem 3.2]{Keller06d}).
 The canonical functor $H^0(\D_{\dg}(\A))\rightarrow \D(\A)$ is an equivalence of triangulated categories.
For dg categories $\A$ and $\B$, we define $\rep(\B,\A)$ to be the full subcategory of $\D(\A\otimes^{\mathbb L} \B^{op})$ whose objects are the dg bimodules $X$ such that $X(-,B)$ is quasi-representable for each object $B$ of $\B$. 
 We define the canonical dg enhancement $\rep_{\dg}(\B,\A)$ to be the full dg subcategory of $\D_{\dg}(\A\otimes^{\mathbb L}\B^{op})$ whose objects are those of $\rep(\B,\A)$.
A dg category $\A$ is {\em pretriangulated} if the canonical inclusion $H^0(\A)\rightarrow \D(\A)$ is a triangulated subcategory.
Note that $\rep_{\dg}(\B,\A)$ is pretriangulated if $\A$ is pretriangulated.
For a dg category $\A$, we have its {\em pretriangulated hull} $\pretr(\A)$ (\cite{BondalKapranov90, Drinfeld04, BondalLarsenLunts04}). 
We denote the triangulated category $H^0(\pretr(\A))$ by $\tr(\A)$.
The Yoneda dg functor $\A\rightarrow \C_{\dg}(\A)$ extends to a fully faithful dg functor $\pretr(\A)\rightarrow \D_{\dg}(\A)$.
It induces a fully faithful triangle functor $\tr(\A)\rightarrow \D(\A)$ whose essential image is the triangulated subcategory of $\D(\A)$ generated by the representable dg modules.

	For a small category $\I$, we have the $k$-category $k\I$ whose set of objects is the same as that of $\I$ and for each pair of objects $x,y$ in $\C$, the space $k\I(i,j)$ is the free $k$-module generated by the set $\I(i,j)$. 
	The {\em square} category $\mathrm{Sq}$ is the path category of the quiver 
\[
\begin{tikzcd}
00\ar[r,"f"]\ar[d,"g"swap] & 01\ar[d,"j"] \\
10\ar[r,"k"swap] & 11
\end{tikzcd}
\]
with the commutativity relation $jf\sim kg$;
the {\em cospan} category $\mathrm{Cosp}$ is the path category of the quiver
\[
\begin{tikzcd}
&01\ar[d]\\
10\ar[r]&11;
\end{tikzcd}
\]
the {\em span} category $\mathrm{Sp}$ is the path category of the quiver
\[
\begin{tikzcd}
00\ar[r]\ar[d]&01\\
10&
\end{tikzcd}.
\]

\begin{lemma}[\cite{KellerNicolas13,BeligiannisReiten07}] \label{lemma:tstructure}
 Let $\A$ be a dg category.
There is a canonical t-structure 
\[
(\D(\A)^{\leq 0},\D(\A)^{\geq 0})
\]
on $\D(\A)$ such that $\D(\A)^{\geq 0}$ is formed by those dg modules whose cohomology is concentrated in non-negative degrees. 
\end{lemma}
Note that we have $A^{\wedge}\in\D(\A)^{\leq 0}$ for each $A\in\A$.
When $\A$ is connective, then $\D(\A)^{\leq 0}$ is formed by those dg modules whose cohomology is concentrated in non-positive degrees.
If $\A$ is strictly connective, then for each dg $\A$-module $M$, the truncation $\tau_{\leq 0}M$ becomes a dg $\A$-submodule and the canonical triangle for $M$ is induced by the short exact sequence $0\rightarrow \tau_{\leq 0}M\rightarrow M\rightarrow \tau_{\geq 1}M\rightarrow 0$ of dg $\A$-modules.

\begin{proposition}$($\cite[Proposition 6.2.1]{Bondarko10}$)$ \label{cot}
Suppose $\A$ is a connective dg category. 
Put $\T=\tr(\A)$. 
Let $\T^{\leq 0}$ (resp.,~$\T^{\geq 0}$) be the full subcategory of $\T$ which consists of objects isomorphic to those of the form $A=(\oplus_{i=1}^{n} A_i[r_i],q)$ for $n\geq 0$ and $r_i\geq 0$ (resp.,~$r_i\leq 0$) for each $i$. 
Then $(\T^{\geq 0}, \T^{\leq 0})$ is a co-t-structure on $\T$, the {\em canonical co-t-structure}. 
\end{proposition}

\subsection{Exact dg categories}
In this subsection, we briefly recall the definitions of homotopy short exact sequences and exact dg categories.
For more details, we refer to \cite{Chen24}.

Let $\I$ be the dg $k$-path category of the following quiver with relations:
\begin{equation}\label{quiv:3term}
\begin{tikzcd}
0\ar[r,"f",""{swap, name=1}]&1\ar[r,"g",""{swap,name=2}]&2,\ar[r,from=1,to=2,dashed,no head,swap,bend right =8ex]
\end{tikzcd}
\end{equation}
where $f$ and $g$ are closed of degree 0, and $gf=0$.

\begin{definition}[{\cite[Definition 3.14]{Chen24}}]
The category $\mathcal H_{3t}(\A)$ of {\em 3-term homotopy complexes} (={\em 3-term h-complexes}) over $\A$ is defined to be the 0-th homology category $H^0(\Fun_{A_\infty}(\I,\A))$ of the dg category $\Fun_{A_\infty}(\I,\A)$ of strictly unital $A_{\infty}$-functors 
 from $\I$ to $\A$, cf.~\cite{Lefevre03}. 
\end{definition}

We unpack the definition that we will actually work with as follows.
We identify objects $X$ in $\mathcal H_{3t}(\A)$ with diagrams in $\A$
\begin{equation}\label{dia:F}
\begin{tikzcd}
&A_0\ar[r,"f"]\ar[rr,bend right = 8ex,"h"swap]&A_1\ar[r,"j"]&A_2,
\end{tikzcd}
\end{equation}
where $|f|=|j|=0$, $|h|=-1$ and $d(f)=0$, $d(j)=0$ and $d(h)=-jf$.
Suppose $X'$ is given by the diagram where a superscript $'$ is added to each object and each morphism. 
A morphism from $X$ to $X'$ is given by an equivalence class of 6-tuples $(r_i,s_j,t)$, $i=0$, $1$, $2$, $j=1$, $2$, where 
\[
r_i:A_i\rightarrow A_i', |r_i|=0, d(r_i)=0,
\]
\[
s_j:A_{j-1}\rightarrow A_{j}', |s_j|=-1, d(s_1)=f'r_0-r_1f,
\]
\begin{equation*}
 d(s_2)=j'r_1-r_2j, 
\end{equation*} 
and 
\begin{equation*}
t:A_0\rightarrow A_2', |t|=-2, d(t)=r_2\circ h-h'\circ r_0-s_2\circ f-j'\circ s_1,
\end{equation*}
which can be identified with the following diagram:
\begin{equation*}\label{mordegreen}
\begin{tikzcd}
&A_0\ar[r,"f"]\ar[d,"r_0"swap]\ar[rd,"s_1"red,swap ,red]\ar[rrd,"t"blue,blue]\ar[rr,bend left = 8ex,"h"]&A_1\ar[d,"r_1"swap]\ar[rd,"s_2"red,red]\ar[r,"j"]&A_2\ar[d,"r_2"]\\
&A_0'\ar[r,"f'"swap]\ar[rr,bend right = 8ex,"h'"swap]&A_1'\ar[r,"j'"swap]&A_2'
\end{tikzcd}.
\end{equation*}

A 3-term homotopy complex over $\A$ (\ref{dia:F}) yields a canonical morphism in $\C(\A)$
\[
u: A_0^{\wedge}\rightarrow \Sigma^{-1}\Cone(j^{\wedge}:A_1^{\wedge}\rightarrow A_2^{\wedge}),
\]
where $u=\begin{bmatrix}f^{\wedge}\\h^{\wedge}\end{bmatrix}$.
\begin{definition}[{\cite[Lemma 3.22]{Chen24}}] \label{intro:homotopyleftexact}
A $3$-term h-complex (\ref{dia:F}) is homotopy left exact if and only if the map $\tau_{\leq 0}(u)$ is a quasi-isomorphism of dg $\tau_{\leq 0}\A$-modules. Dually, we define the notion of homotopy right exact sequence and then the notion of homotopy short exact sequence.
\end{definition}
A {\em homotopy square} \cite[Definition 3.21]{Chen24} in $\A$ is a diagram in $\A$
\begin{equation}\label{square:A}
\begin{tikzcd} 
A\ar[r,"{i}"]\ar[d,"{j}"swap]\ar[rd,"h"blue,blue]&{B}\ar[d,"{p}"]\\
{B'}\ar[r,"{q}"swap]&{C}
\end{tikzcd}
\end{equation}
where $i$, $j$, $p$ and $q$ are closed morphisms of degree $0$, and $h$ is a morphism of degree $-1$ such that $d(h)=pi-qj$, yields a canonical $3$-term h-complex
\begin{equation}\label{sequence:A}
\begin{tikzcd}
A\ar[r,"\begin{bmatrix}i\\j\end{bmatrix}"]\ar[rr,bend right=8ex,"h"swap]&B\oplus B'\ar[r, "{[}p{,}-q{]}"]&C.\\
\end{tikzcd} 
\end{equation}
The homotopy square (\ref{square:A}) is a {\em homotopy pullback square} (={\em homotopy cartesian square}) if the $3$-term h-complex (\ref{sequence:A}) is homotopy left exact.
Dually, we have the notions of {\em homotopy pushout square} (={\em homotopy cocartesian square}) and {\em homotopy bicartesian square}.

\begin{theorem}[\cite{CanonacoOrnaghiStellari18} cf.~also \cite{Faonte17a}]
We have a canonical isomorphism 
\[
\Fun_{A_{\infty}}(-,-)\rightarrow \rep_{\dg}(-,-)
\]
of bifunctors $\Hqe_{k-\mathrm{cf}}^{op}\times \Hqe\rightarrow \Hqe$, where $\Hqe_{k-\mathrm{cf}}$ is the
full subcategory of $\Hqe$ whose objects are the dg categories whose morphism
complexes are cofibrant dg $k$-modules.
\end{theorem}
\begin{corollary}\label{strictify}
Let $\A'$ be a full dg subcategory of $\A$ and $\B$ be a small dg category. 
Let $F:\rep(\B,\A')\rightarrow \rep(\B,\A)$ be the functor induced by the inclusion $\A'\rightarrow \A$.
Then an object $X\in \rep(\B,\A)$ lies in the essential image of $F$ if and only if for each object $B$ in $\B$, $X(-,B)$ is quasi-isomorphic to a dg $\A$-module represented by an object in $\A'$.
\end{corollary}
\begin{proof}
Let $\A''$ be the closure of $\A'$ under homotopy equivalences in $\A$. 
The inclusion of 
$\A'$ into $\A''$ is a quasi-equivalence and therefore induces a quasi-equivalence in
$\rep_{\dg}(\B,?)$. 
By the above theorem, it also induces a quasi-equivalence in
$\Fun_{A_{\infty}}(\B,?)$. 
Thus, we may assume that $\A'$ is stable under homotopy equivalences
in $\A$. 
Now the inclusion $\Fun_{A_{\infty}}(\B,\A') \to \Fun_{A_{\infty}}(\B,\A)$ is clearly an isomorphism
onto the full subcategory of $A_\infty$-functors whose values on objects lie in $\A'$. 
Thanks to the theorem, this implies the statement.
\end{proof}

\begin{definition}[{\cite[Definition 4.1]{Chen24}}]\label{exactdgstructure}
An {\em exact structure} on $\A$ is a class $\mathcal{S}\subseteq \mathcal H_{3t}(\A)$ stable under isomorphisms, consisting of homotopy short exact sequences (called {\em conflations})
\[
\begin{tikzcd}
A\ar[r, tail,"i"]\ar[rr,bend right=8ex,"h"swap]&B\ar[r,two heads, "p"]&C\\
\end{tikzcd} 
\]
where $i$ is called an {\em inflation} and $p$ is called a {\em deflation}, such that the following axioms are satisfied
\begin{itemize}
\item[Ex0]$\Id_{0}$ is a deflation.
\item[{Ex}1]Compositions of deflations are deflations.
\item[{Ex}2]Given a deflation $p:B\rightarrow C$ and any map $c: C'\rightarrow C$ in $Z^0(\A)$, the object
\begin{tikzcd}
B\ar[r,"p"]&C&C'\ar[l,"c"swap]
\end{tikzcd}
admits a homotopy pullback 
\[
\begin{tikzcd} 
{B'}\ar[r,"{p'}"]\ar[d,"{b}"swap]\ar[rd,"s"blue,blue]&{C'}\ar[d,"{c}"]\\
{B}\ar[r,"{p}"swap]&{C}
\end{tikzcd}
\]
and ${p'}$ is also a deflation.
\item[$\Ex2^{op}$]Given an inflation $i: A\rightarrow B$ and any map $a:A\rightarrow A'$ in $Z^0(\A)$, the object
\begin{tikzcd}
A'&A\ar[r,"i"]\ar[l,"a"swap]&B
\end{tikzcd}
admits a homotopy pushout 
\[
\begin{tikzcd}
{A}\ar[r,"{i}"]\ar[d,"{a}"swap]\ar[rd,"s"blue,blue]&{B}\ar[d,"{j}"]\\
{A'}\ar[r,"{i'}"swap]&{B'}
\end{tikzcd}
\]
and ${i'}$ is also an inflation.
\end{itemize}
We call $(\A,\mathcal {S})$ or simply $\A$ an {\em exact dg category}.
\end{definition} 
\begin{definition}[{\cite[Definition 4.3]{Chen24}}]\label{def:exactmorphism}
Let $(\A,\mathcal S)$ and $(\A',\mathcal S')$ be exact dg categories.
A morphism $F:\A\rightarrow \A'$ in $\Hqe$ is {\em exact} if the induced functor $\mathcal H_{3t}(\A)\rightarrow \mathcal H_{3t}(\A')$ sends objects in $\mathcal S$ to objects in $\mathcal S'$.
We denote by $\Hqe_{\ex}((\A,\S),(\A',\S'))$ the subset of $\Hqe(\A,\A')$ consisiting of exact morphisms.
We denote by $\Hqe_{\ex}$ the subcategory of $\Hqe$ consisting of exact dg categories with exact morphisms.
An exact morphism $(\A,\S)\rightarrow (\A',\S')$ is an {\em exact quasi-equivalence} if it is an isomorphism in $\Hqe_{\ex}$.
\end{definition}

\begin{examples}\label{exm:exactdg}
Let $\A$ be an additive dg category.
\begin{itemize}
\item[1)] If $\A$ is pretriangulated, then the class of all homotopy short exact sequences determines an exact structure on $\A$, cf.~\cite[Example 4.7]{Chen24}.
\item[2)] If $\A$ is concentrated in degree zero, the exact dg structures on $\A$ are the same as the Quillen exact structures on $\A$, cf.~\cite[Example 4.6]{Chen24}. 
 \item[3)] The notion of exact dg structure is self-dual, cf.~\cite[Proposition 4.9 d)]{Chen24}. 
\item[4)] Let $F:\A\rightarrow \A'$ be a quasi-equivalence of dg categories. 
Then the quasi-equivalence $F$ induces a bijection
between the class of exact dg structures on $\A$ and that on $\A'$.
Moreover, for each exact structure on $\A$ and the exact structure on $\A'$ induced by this bijection, the dg functor $F$ induces an equivalence of extriangulated categories between $(H^0(\A),\mathbb E,\mathfrak s)$ and $(H^0(\A'),\mathbb E',\mathfrak s')$, cf.~\cite[Remark 4.5 b)]{Chen24}.
 \item[5)] Any additive dg category $\A$ with the closure under isomorphisms of the class of all split exact sequences is an exact dg category. 
 We call this the {\em trivial exact structure}. 
 Let $\B$ be another exact dg category. 
 Then any dg functor $F:\A\rightarrow \B$ is exact when $\A$ is endowed with the trivial exact structure.
\item[6)] A full dg subcategory $\A'$ of an exact dg category $\A$ is {\em extension-closed} provided that for each conflation in $\A$
\[
\begin{tikzcd}
A\ar[r,"f"]\ar[rr,bend right=8ex,"h"swap]&B\ar[r,"j"]&C,
\end{tikzcd}
\]
if $A$ and $C$ belong to $\A'$, then so does $B$. 
If $\A$ is an exact dg category and $\B$ is an extension-closed subcategory, then $\B$ inherits an exact structure, cf.~\cite[Example-Definition 5.7]{Chen23}. 
\end{itemize}
\end{examples}




  
\section{Embedding theorem for connective exact dg categories}\label{sec:embedding}
Throughout this section, let $(\A,\S)$ be an exact dg category. 
It is {\em connective} if the underlying dg category $\A$ is connective, i.e.,~the Hom complexes $\Hom_{\A}(X,Y)$ has cohomology concentrated in non-positive degrees for all $X$, $Y$ in $\A$.
By Example~\ref{exm:exactdg} 4), we may assume that $\A$ has cofibrant Hom complexes.
One of the main aims of this section is to show the following theorem.

\begin{theorem}\label{main}
Let $(\A,\S)$ be an exact dg category. 
There exists a universal exact morphism $F:\A\rightarrow \D^b_{\dg}(\A,\S)$ in $\Hqe$ from $\A$ to a pretriangulated dg category $\D^b_{\dg}(\A,\S)$. 

If $\A$ is connective, this morphism satisfies:
\begin{itemize}
\item[1)]It induces a quasi-equivalence from $\tau_{\leq 0}{\A}$ to $\tau_{\leq 0}\D'$ for an extension-closed dg subcategory $\D'$ of $\D^b_{\dg}(\A,\S)$;
\item[2)]It induces a natural bijection $\mathbb E(C,A)\xrightarrow{\sim} \Ext^1_{\D^b(\A,\S)}(FC,FA)$ for each pair of objects $C,A$ in $H^0(\A)$ where $\D^b(\A)=H^0(\D^b_{\dg}(\A,\S))$.
\end{itemize}
We call $\D^b_{\dg}(\A,\S)$ the {\em bounded dg derived category} of $(\A,\S)$.
\end{theorem}
We will prove Theorem~\ref{main} after Lemma~\ref{lem:tstructureff}.
To avoid the heavy use of the symbol $\wedge$, we will omit it for representable dg modules.
\begin{example} If $\A$ is a Quillen exact category, then $\D^b_{\dg}(\A,\S)$ is quasi-equivalent to the canonical dg enhancement of the bounded derived category of $\A$, whence the name.
\end{example}

\begin{remark} The analogous theorem for exact $\infty$-categories \cite{Barwick15} in the sense of Barwick is
due to Klemenc \cite{Klemenc22}.
\end{remark}

\subsection{Construction of the universal morphism}
Let $\mathcal N$ be the full triangulated subcategory of $\tr(\A)$ generated by the total dg modules $N$ of conflations 
\[
\begin{tikzcd}
A\ar[r,"f"]\ar[rr,"h"swap,bend right=8ex]&B\ar[r,"j"]&C
\end{tikzcd}. 
\]
So $N$ is defined by the following diagram in $\C_{\dg}(\A)$
\begin{equation}\label{TR4}\tag{$\bigstar$}
\begin{tikzcd}[every label/.append style={font=\tiny}]
A\ar[r,"f"]\ar[d,"\begin{bmatrix}-f\\-h\end{bmatrix}",swap]&B\ar[r,"\begin{bmatrix}0\\1\end{bmatrix}"]\ar[d,equal]&U\ar[rd,"s"red,red]\ar[r,"{[}1{,}\;0{]}"]\ar[d,"{[}-h{,}j{]}",swap]&\Sigma A\ar[d]\\ 
V\ar[r,"{[}-1\ 0{]}"swap]&B\ar[r,"j"swap]&C\ar[r,"\begin{bmatrix}0\\1\end{bmatrix}"swap]\ar[d,"\begin{bmatrix}0\\0\\1\end{bmatrix}",swap]&\Sigma V\ar[d]\\ 
& &N\ar[r,equal]\ar[d]&N\ar[d]\\
&&\Sigma U\ar[r]&\Sigma^{2}A
\end{tikzcd}
\end{equation}
where we omit the symbol $\wedge$ for representable dg modules and where 
\[
s=\begin{bmatrix}0&0\\-1&0\end{bmatrix} \ko U=\Cone(f)\, \ko V=\Sigma^{-1}\Cone(j) \mbox{ and }
N=\Cone([-h,\;j]).
\]
Let $\mathcal N_{\dg}$ be the full dg subcategory of $\pretr(\A)$ consisting of the objects in $\N$. 
Let $F$ be the canonical morphism from $\A$ to the Drinfeld dg quotient $\pretr(\A)/\mathcal N_{\dg}$
(recall that we assume the morphism complexes of $\A$ to be cofibrant hence flat). We consider
$F$ as a morphism in the category $\Hqe$ obtained from the category of small dg categories
by localizing at the class of quasi-equivalences.

\begin{lemma}\label{univer} The morphism
$F:\A\rightarrow \pretr(\A)/\mathcal N_{\dg}$ is the universal exact morphism (cf.~Definition~\ref{def:exactmorphism}) from $\A$ to a pretriangulated dg category.
\end{lemma}
\begin{proof}Put $\D^b(\A,\S)=H^0(\pretr(A)/\mathcal N_{\dg})$.
We first show that $F$ is exact.
Let $X$ be a conflation as follows
\[
\begin{tikzcd}
A\ar[r,"f"]\ar[rr,"h"swap,bend right=8ex]&B\ar[r,"j"]&C
\end{tikzcd}.
\]
We consider the following roof $[1,0]/[-h,j]$: 
\[
\begin{tikzcd}
&U\ar[ld,Rightarrow,"{[}-h{,}j{]}"swap]\ar[rd,"{[}1{,}0{]}"]&\\
C&&\Sigma A
\end{tikzcd}
\]
where $U$ is the cone of $f:A\rightarrow B$.
We have the following triangle in $\D^b(\A,\S)$:
\[
\begin{tikzcd}
A\ar[r,"f/1"]&B\ar[r,"j/1"]&C\ \ar[r,"{[1}{,}0{]}{/}{[}-h{,}j{]}"]&\ \Sigma A.
\end{tikzcd}
\]
So the morphism $F$ is exact.

We show that $F$ is universal.
Let $G:\A\rightarrow \B$ be an exact morphism in $\mathrm{Hqe}$ from $\A$ to a pretriangulated dg category $\B$.
By the universal property of $\A\rightarrow \pretr(\A)$, there exists a unique morphism from $\pretr(\A)$ to $\B$ in $\mathrm{Hqe}$ which extends $G$.
Since $G$ is an exact morphism, $H^0(G)$ sends the totalization $N$ of $X$ to a zero object in $H^0(\B)$.
By the universal property of Drinfeld dg quotient \cite{Tabuada10}, there exists a unique morphism in 
$\Hqe$ from $\pretr(A)/\mathcal N_{\dg}$ to $\B$.
\end{proof}
\begin{proposition}\label{prop:Grothendieckgroups}
For a connective exact dg category $\A$, the universal exact morphism $F:\A\rightarrow\D^b_{\dg}(\A,\S)$ induces an isomorphism $K_0(H^0(\A))\rightarrow K_0(\D^b(\A,\S))$, where $K_0(H^0(\A))$ denotes the Grothendieck group of the extriangulated category $(H^0(\A),\mathbb E,\mathfrak s)$.
\end{proposition}
\begin{proof}
We assume that $\A$ is strictly connective.
The dg functor $\A\rightarrow H^0(\A)$ induces a dg functor between the pretriangulated hulls $\pretr(\A)\rightarrow \pretr(H^0(\A))$.
We have the following commutative diagram
\[
\begin{tikzcd}
\A\ar[r]\ar[d,hook]&H^0(\A)\ar[d,hook]\\
\pretr(\A)\ar[r]&\pretr(H^0(\A))
\end{tikzcd}
\begin{tikzcd}
H^0(\A)\ar[r,equal]\ar[d,hook]&H^0(\A)\ar[d,hook]\\
\tr(\A)\ar[r]&\tr(H^0(\A))=\H^b(H^0(\A))
\end{tikzcd}
\]
By \cite[Corollary B]{ChenLiZhangZhao24}, we have $K_0^{\sp}(H^0(\A))\iso K_0(\tr(H^0(\A)))$. 
Hence the canonical map $K_0^{\sp}(H^0(\A))\rightarrow K_0(\tr(\A))$ is injective.
Since $\tr(\A)$ is generated as a triangulated category by the objects in $H^0(\A)$, it is also surjective. 
Hence we have $K_0^{\sp}(H^0(\A))\iso K_0(\tr(\A))$.

By definition, we have $\D^b(\A,\S)\iso \tr(\A)/\N$ where $\N$ is the full triangulated subcategory of $\tr(\A)$ generated by the totalizations of the conflations.
By \cite[VIII 3.1]{Illusie77}, we have that $K_0(\D^b(\A,\S))$ is the quotient of $K_0(\tr(\A))$ by the subgroup generated by $[N]\in K_0(\tr(\A))$ where $N$ runs through the totalizations of conflations in $(\A,\S)$.
Note that in (\ref{TR4}) we have $[N]=[C]-[U]=[C]-[B]+[A]$ in $K_0(\tr(\A))$.
It follows that the isomorphism $K_0^{\sp}(H^0(\A))\iso K_0(\tr(\A))$ induces a canonical isomorphism 
\[
K_0(H^0(\A),\mathbb E,\mathfrak s)\iso K_0(\D^b(\A,\S)).
\]
\end{proof}
Denote by $\mathcal H$ the heart of the canonical t-structure $(\D(\A)^{\leq 0},\D(\A)^{\geq 0})$, cf.~Lemma~\ref{lemma:tstructure}. 
We denote by $H^0:\D(\A)\rightarrow \mathcal H$ the associated homological functor.
Recall that we omit the symbol $\wedge$ for representable dg modules.

\begin{lemma}\label{leftexactsequence}
A 3-term h-complex
\[
\begin{tikzcd}
A\ar[r,"f"]\ar[rr,bend right=8ex,"h"swap]&B\ar[r,"j"]&C
\end{tikzcd}
\]
is exact if and only if its totalization $N$ lies in $\mathcal H$. In this case, the totalization $N$ is the cokernel of $H^0(B)\xrightarrow{H^0(j)}H^0(C)$ in $\mathcal H$.
\end{lemma}
\begin{proof}
Consider the diagram (\ref{TR4}).
For $A'\in\A$ and $i\leq 0$, we have 
\[
\Hom_{\D(\A)}(A',\Sigma^{i}A)\xrightarrow{\sim}\Hom_{\D(\A)}(A',\Sigma^{i}V)
\]
and so we have $\Hom_{\D(\A)}(A',\Sigma^{i-1}N)=0$ and hence $N\in \D(\A)^{\geq 0}$. 
Since representable dg modules lie in $\D(\A)^{\leq 0}$, 
it is clear from the construction that $N\in \D(\A)^{\leq 0}$. 
Therefore, we have $N\in\mathcal H$.

From the rightmost column of the diagram (\ref{TR4}), we have $H^1(V)\xrightarrow{\sim}N$. 
From the second row of (\ref{TR4}), we have an exact sequence 
\[
\begin{tikzcd}H^0(B)\ar[r,"H^0(j)"]&H^0(C)\ar[r]&H^1(V)\ar[r]&H^1(B)=0.\end{tikzcd}
\]
So the second statement follows.
\end{proof}
\begin{remark}
We have an exact sequence in $\mathcal H$:
\[
\begin{tikzcd}
H^0(A)\ar[r,"H^0(f)"]&H^0(B)\ar[r,"H^0(j)"]&H^0(C)\ar[r]&N\ar[r]&0
\end{tikzcd}
\]
where, with the notations of  (\ref{TR4}), the image of the morphism $H^0(j)$ is $H^0(U)$.
Using the canonical extriangulated category $(H^0(\A),\mathbb E,\mathfrak s)$, this is equivalent to saying that $N$ is the {\em defect} (cf.~\cite[Definition 2.4]{Ogawa21}) of the conflation $
A\xrightarrow{\overline{f}}B\xrightarrow{\overline \jmath}C 
$
in $(H^0(\A),\mathbb E,\mathfrak s)$.
\end{remark}
\begin{notation}\label{not:N}
We denote by $\tilde{\mathcal N}\subseteq \mathcal H$ the essential image of $\mathcal N$ under the homological functor $H^0:\D(\A)\rightarrow \mathcal H$.
\end{notation}
\begin{definition}\label{def:defective}
An object in $\mathcal H$ is {\em defective}, if it is the cokernel of a map 
\[
H^0(B)\xrightarrow{H^0(\overline{j})}H^0(C)\ko
\]
where $j$ is a deflation, or equivalently if it is the totalization of some conflation.
\end{definition}
Let $X_i$, $i=1,2$, be homotopy short exact sequences
\[
\begin{tikzcd}
A_i\ar[r,"f_i"]\ar[rr,"h_i"swap,bend right=8ex]&B_i\ar[r,"j_i"]&C_i
\end{tikzcd}
\]
and $\alpha: X_1 \rightarrow X_2$ a morphism of ${\mathcal H}_{3t}(\A)$. Let
$X_3$ be homotopy square in $\A$ 
\[
\begin{tikzcd}
A_1 \ar[r,"f_1"] \ar[d] \ar[rd] & B_1 \ar[d]\\
A_2 \ar[r,"f_2"swap] & B_2
\end{tikzcd}
\]
obtained by restricting a representative of $\alpha$.
Notice that by \cite[Lemma 3.22]{Chen24}, up to (non canonical) isomorphism,
the object $X_3$ is independent of the choice of representative.
Suppose that $\alpha$ restricts to an isomorphism $C_1\rightarrow C_2$ in $H^0(\A)$. 
Then $X_3$ is homotopy cartesian by \cite[Proposition 3.31]{Chen24}. 
Let $N_i$ be the totalization of $X_i$ for $i=1$, $2$, $3$. 

\begin{lemma}\label{exact}
Keep the notations as above. The following statements hold.
\begin{itemize}
\item[a)] There exists a triangle in $\D(\A)$
\[
N_3\rightarrow N_1\rightarrow N_2\rightarrow \Sigma N_3
\]
\item[b)] If $X_1$ is a conflation, the triangle in a) induces a short exact sequence in $\mathcal H$
\[
0\rightarrow N_3\rightarrow N_1\xrightarrow{g} N_2\rightarrow 0
\]
where $g$ is the morphism induced by $\alpha$.
\end{itemize}
\end{lemma}
Notice that if $X_1$ is a conflation, then $X_2$ is a conflation by the dual of \cite[Proposition 4.9 b)]{Chen24}
and hence $X_3$ is homotopy bicartesian by the dual of \cite[Lemma 4.12]{Chen24}. 
\begin{proof} a)
Suppose the morphism $\alpha:X_1\rightarrow X_2$ is given by the following diagram in $\A$:
\[
\begin{tikzcd}
A_1\ar[r,"f_1"]\ar[rd,"s_1"{red,swap},red]\ar[rrd,"t"{blue,near start},blue]\ar[rr,bend left=8ex,"h_1"]\ar[d,"a"swap]&B_1\ar[rd,"s_2"red,red]\ar[r,"j_1"]\ar[d,"b"{swap,near end}]&C_1\ar[d,"c"]\\
A_2\ar[rr,bend right=8ex,"h_2"swap]\ar[r,"f_2"swap]&B_2\ar[r,"j_2"swap]&C_2
\end{tikzcd}.
\]
Since $\overline{c}$ is an isomorphism in $H^0(\A)$, we may replace $C_1$ by $C_2$, $j_1:B_1\rightarrow C_1$ by $c\circ j_1$ and $h_1$ by $c\circ h_1$. 
Thus we may assume $C_1=C_2$ and $c=\Id_{C_1}$ in $Z^0(\A)$.
Then we have the following diagram in $Z^0(\A)$:
\[
\begin{tikzcd}
A_1\ar[rrd,"-t"blue,blue]\ar[rr,bend left=8ex,"s_1"]\ar[r,"\phi"]\ar[d,equal]& B_1\oplus A_2\ar[rd,"\psi"{red},red]\ar[r,"{[}b{,}-f_2{]}"]\ar[d,"{[}1{,}0{]}"{swap,near end}]& B_2\ar[d,"j_2"]\\
A_1\ar[rrd,"t"{swap,near end,blue},blue]\ar[rd,"s_1"{red,swap},red]\ar[r,"f_1"]\ar[d,"a"swap]& B_1\ar[rd,"s_2"{red},red]\ar[r,"j_1"]\ar[d,"b"{swap,near end}]& C_1\ar[d,equal]\\
A_2\ar[rr,bend right=6ex,"h_2"swap]\ar[r,"f_2"swap]&B_2\ar[r,"j_2"swap]&C_1
\end{tikzcd}
\]
where we have omitted a zero diagonal map and the homotopy $h_1:A_1\rightarrow C_1$ and where $\phi=\begin{bmatrix}f_1\\a\end{bmatrix}$, $\psi={[}-s_2{,}-h_2{]}$.
It induces a canonical diagram in $\C(\A)$
\[
\begin{tikzcd}
A_2\ar[r,"\begin{bmatrix}0\\0\\1\end{bmatrix}"]\ar[d,"\begin{bmatrix}-f_2\\-h_2\end{bmatrix}"swap]&\Cone(\phi)\ar[rd,"s'"{red},red]\ar[r,"\begin{bmatrix}1\ 0\ 0\\0\ 1\ 0\end{bmatrix}"]\ar[d,"u"]&\Cone(f_1)\ar[d,"{[}-h_1{,}j_1{]}"]\\
\Sigma^{-1}\Cone(j_2)\ar[r,"{[}1{,}0{]}"swap]&B_2\ar[r,"j_2"swap]&C_1
\end{tikzcd}
\]
where $u={[}-s_1{,}b{,}-f_2{]}$, $s'=[t,s_2,h_2]$ and where the first row is a graded-split short exact sequence in $\C_{\dg}(\A)$. 
By the $3\times 3$-lemma, we have a triangle in $\D(\A)$ 
\[
N_3\rightarrow N_1\rightarrow N_2\rightarrow \Sigma N_3
\]
which gives the short exact sequence of b) in $\mathcal{H}$ when $X_1$ is a conflation.
\end{proof}
\subsection{The canonical t-structure on the category $\mathcal N$}\label{subsec:N}

In the rest of this section, we keep the assumption that $\A$ is connective. In this case, we have 

\begin{itemize}
\item The left aisle $\D(\A)^{\leq 0}$ is formed by those dg modules whose cohomology is concentrated in non-positive degrees.
\item The heart $\mathcal H$ of the t-structure $(\D(\A)^{\leq 0},\D(\A)^{\geq 0})$ is equivalent to $\Mod H^0(\A)$ and the cohomological functor $H^0$ sends a dg module $M$ to $H^0(M):A\mapsto H^0(M(A))$.
\item In particular $H^0$ is fully faithful on the full subcategory of quasi-representable dg modules and $H^0(M)$ is projective in the heart for each quasi-representable dg module $M$.
\end{itemize}

Recall that a full subcategory of an abelian category is {\em wide} if it is stable under extensions,
kernels and cokernels. In particular, each wide subcategory is abelian and its inclusion is fully 
faithful and fully exact.
\begin{lemma}\label{wide}
The full subcategory of $\mathcal H$, consisting of defective objects (cf.~Definition~\ref{def:defective}), is a wide subcategoy.
\end{lemma}

\begin{corollary}\label{t}
An object in $\mathcal H$ lies in $\tilde{\mathcal N}$ (cf.~Notation~\ref{not:N}) if and only if it is defective. 
The t-structure $(\D(\A)^{\leq 0},\D(\A)^{\geq  0})$ on $\D(\A)$ restricts to a bounded t-structure on $\N$.
\end{corollary}

\begin{proof}[Proof of the Corollary] 
Let $\mathcal N'$ be the full subcategory of $\N$ consisting of objects $N$ such that $H^i(N)$ is defective for each $i\in\mathbb Z$ and only finitely many of them are nonzero. 
By Lemma \ref{wide}, this is a triangulated subcategory of $\mathcal N$. It certainly contains totalizations of conflations and thus is equal to $\N$. 
Therefore the objects $N$ in $\N$ have the property that $H^i(N)$ is defective for each $i\in \mathbb Z$ and only finitely many of them is nonzero.
\end{proof}
\begin{proof}[Proof of Lemma \ref{wide}]
Denote by $\overline{\mathcal N}$ the full subcategory of $\mathcal H$ consisting of defective objects. Consider a 
three-term complex
\begin{align}
P\rightarrow Q\rightarrow R \label{short}
\end{align}
in $\Mod H^0(\A)$ which is exact in the middle. 

(1) We first show that $\overline{\mathcal N}$ is closed under extensions.
Assume the sequence (\ref{short}) is a short exact sequence with $P$ and $R$ defective.
Suppose $P$ and $R$ are totalizations of $X_1$ and $X_2$ respectively, where $X_i$, $i=1,2$, is of the form
\[
\begin{tikzcd}
A_i\ar[r,"f_i"]\ar[rr,"h_i"swap,bend right=8ex]&B_i\ar[r,"j_i"]&C_i
\end{tikzcd}.
\]
Then $P$ is the cokernel of $H^0B_1\xrightarrow{H^0(\overline{\jmath}_1)} H^0C_1$ and $R$ is the cokernel of $H^0B_2\xrightarrow{H^0(\overline{\jmath}_2)} H^0C_2$.
By the horseshoe lemma, $Q$ is the cokernel of 
\[
H^0(B_1\oplus B_2)\xrightarrow{\begin{bmatrix}\overline{\jmath}_1\ \ \overline{h}\\0\ \ \overline{\jmath}_2\end{bmatrix}}H^0(C_1\oplus C_2)
\]
in $\mathcal H$, where $h:B_2\rightarrow C_1$ is a morphism in $Z^0(\A)$.
Since the morphism 
\[
B_1\oplus B_2\xrightarrow{\begin{bmatrix}{j}_1\ \ h\\0\ \ {j}_2\end{bmatrix}}C_1\oplus C_2
\]
 can be written as the following composition
\[
\begin{tikzcd}
B_1\oplus B_2\ar[r,"{\begin{bmatrix}j_1\ \ 0\\0\ \ 1\end{bmatrix}}"] &C_1\oplus B_2\ar[r,"{\begin{bmatrix}1\ \ h\\0\ \ 1\end{bmatrix}}"]&C_1\oplus B_2\ar[r,"{\begin{bmatrix}1\ \ 0\\0\ \ j_2\end{bmatrix}}"]&C_1\oplus C_2 \ko
\end{tikzcd}
\]
it is a deflation by the axioms ${\Ex0}$ and ${\Ex2}$. Therefore $Q$ is also defective.

(2) We show that $\overline{\mathcal N}$ is closed under kernels.
Assume the sequence (\ref{short}) is a left short exact sequence in $\mathcal H$ with $Q$ and $R$ defective.
Suppose $Q$ and $R$ are totalizations of conflations $X_1$ and $X_2$ respectively, as described in $(1)$.
The morphism $Q\rightarrow R$ induces a commutative diagram in $\mathcal H$
\[
\begin{tikzcd}
H^0B_1\ar[r,"H^0(\overline{\jmath_1})"]\ar[d,dotted]&H^0C_1\ar[r]\ar[d,dotted,"H^0(\overline{c})"]&Q\ar[d]\ar[r]&0\\
H^0B_2\ar[r,"H^0(\overline{\jmath_2})"swap]&H^0C_2\ar[r]&R\ar[r]&0
\end{tikzcd}
\]
which in turn gives a commutative diagram in $\D(\A)$. 
This diagram lifts to a morphism $\mu: j_1\rightarrow j_2$ in $H^0(\Mor(\A))$ which induces a morphism $\alpha: X_1\rightarrow X_2$ in $\mathcal H_{3t}(\A)$. 
Let $\overline{a}:A_1\rightarrow A_2$ be the restriction of $\alpha$ to $A_1$. 
By \cite[Lemma 4.10]{Chen24}, the morphism $\alpha$ can be written as a composition $X_1\xrightarrow{\beta} {X}_3\xrightarrow{\gamma} X_2$ where $X_3$ is a conflation with ends $A_2$ and $C_1$, and the morphism $\beta$ restricts to $\overline a:A_1\rightarrow A_2$ and $\Id_{C_1}$, and the morphism $\gamma$ restricts to $\Id_{A_2}$ and $\overline{c}:C_1\rightarrow C_2$.
Now we apply Lemma \ref{exact} and its dual to the morphisms $X_1\rightarrow {X_3}$ and ${X_3}\rightarrow X_2$ respectively.
We get short exact sequences 
$
0\rightarrow N'\rightarrow Q\rightarrow N''\rightarrow0 
$, and  
$
0\rightarrow N''\rightarrow R\rightarrow N'''\rightarrow 0
$
in $\mathcal H$,
where, by Axiom ${\Ex2}$, $N'$ is the totalization of a conflation. 
This shows that $P$ is isomorphic to  $N'$ and thus is defective.

(3) We show that $\overline{\mathcal N}$ is closed under cokernels.
Assume we have an exact sequence 
$
Q\rightarrow R\rightarrow S\rightarrow 0
$
in $\overline{\mathcal N}$,
where $Q$ and $R$ are defective.
Using the notation in $(2)$, we find that $S$ is isomorphic to $N'''$. 
By Proposition \cite[Proposition 4.9 b)]{Chen24}, $N'''$ is the totalization of a conflation. 
This shows that $S$ is also defective.
\end{proof}

The following lemma is standard.
\begin{lemma}[\cite{IyamaYang20}, Lemma 3.3]\label{lem:tstructureff}
Let $\T$ be a triangulated category with a thick subcategory $\mathcal S$. 
Suppose that $\mathcal S$ admits a torsion pair (\cite[Definition 2.2]{IyamaYoshino08}) $\mathcal S=(\X, \Y)$.
Put $\P=\X\cap \Sigma^{-1}\Y$.
Then for any $A\in {^{\perp}\Y}$ and $B\in \Sigma^{-1}\X^{\perp}$, where the symbol $^{\perp}$ stands for the Hom orthogonal, the canonical map 
\[
\Hom_{\T/[\P]}(A,B)\rightarrow \Hom_{\T/\mathcal S}(A,B)
\]
is a bijection.
\end{lemma}
In the above lemma, we have $\P=0$ when $\mathcal S=(\X,\Y)$ is in particular a t-structure.
\begin{proof}[Proof of Theorem~\ref{main}]
By construction, the dg category $\D^b_{\dg}(\A,\S)$ is the dg quotient \linebreak $\pretr(\A)/\mathcal N_{\dg}$ and $F:\A\rightarrow \D^b_{\dg}(\A,\S)$ is the canonical morphism in $\Hqe$. 
This morphism is universally exact by Lemma \ref{univer}. 
Put $\D^b(\A,\S)=H^0(\D^b_{\dg}(\A,\S))$. 
For simplicity, for objects $A\in H^0(\A)$, we keep the same notation $A$ for the corresponding object in $\D^b(\A,\S)$.
Note that for any object $A\in\A$ and any totalization $N$ of a conflation in $\A$, we have $\Hom_{\tr(\A)}(A, \Sigma^{i}N)=0$ and $\Hom_{\tr(\A)}(N,\Sigma^{i+2} A)=0$ for $i\leq -1$. 
It follows from Corollary~\ref{t} and Lemma~\ref{lem:tstructureff} that for objects $A$ and $C$ in $H^0(\A)$, we have
\[
\Hom_{\tr(\A)}(C,\Sigma^{i}A)\iso \Hom_{\D^b(\A,\S)}(C,\Sigma^{i}A)
\]
for $i\leq 0$.
Consider a triangle 
$
\begin{tikzcd}
A\ar[r] &E\ar[r]&C\ar[r,"u{/}t"]&\Sigma A
\end{tikzcd}
$ in $\D^b(\A,\S)$
where $u/t$ is the following roof
\[
\begin{tikzcd}
&W\ar[ld,Rightarrow,"t"swap]\ar[rd,"u"]&\\
C&&\Sigma A
\end{tikzcd}.
\]
Similarly as above, we may assume that the mapping cone $N'=\Cone(t)$ belongs to $\mathcal N^{\leq 0}$. The composition 
$
\begin{tikzcd}
\Sigma^{-1}N'\ar[r] &W\ar[r,"u"]& \Sigma A
\end{tikzcd}
$
factors through $\Sigma^{-1}H^0(N')$. Therefore we may assume that $N'$ lies in the heart of the t-structure on $\mathcal N$.
So $N'$ is the totalization of a conflation $X'$ of the form
\[
\begin{tikzcd}
A'\ar[r,"f'"]\ar[rr,bend right=8ex,"h'"swap]&B'\ar[r,"j'"]&C'.
\end{tikzcd}
\]
Then the morphisms $C\rightarrow N'$ and $\Sigma^{-1}N'\rightarrow \Sigma A$ give rise to morphisms $c:C\rightarrow C'$ and $a:A'\rightarrow A$ in $H^0(\A)$.
Put $[X'']=c^*[X']$ and $[X''']=a_*[X'']\in \mathbb E(C,A)$. 
Let $N'',N'''$ be the totalizations of $X''$ and $X'''$ respectively. 
Consider the corresponding diagrams (\ref{TR4})  for these conflations.
We have the following commutative diagram in $\tr(\A)$
\[
\begin{tikzcd}
&C\ar[rr]\ar[ld]&&N''\ar[ld]\ar[r]\ar[dd]&N'''\ar[dd]\\
C'\ar[rr]&&N'\ar[rd]&&\\
&&&\Sigma^2 A'\ar[r]&\Sigma^2 A
\end{tikzcd}
\]
So the roof $u/t$ is equivalent to the roof 
\[
\begin{tikzcd} & U'''\ar[ld,Rightarrow]\ar[rd]\\
C&&\Sigma A
\end{tikzcd}
\]
which is given by the conflation $X'''$
\[
\begin{tikzcd}
A\ar[r,"f'''"]\ar[rr,"h'''"swap,bend right=8ex]&B'''\ar[r,"j'''"]&C.
\end{tikzcd}
\]
So the object $E$ is isomorphic to $B'''$ and the image of $H^0(\A)$ under the functor $H^0(F)$ is an extension-closed subcategory of $\D^b(\A,\S)$. 
This proves 1).

It remains to show that $H^0(F)$ induces isomorphisms $\mathbb E(C,A)\xrightarrow{\sim} \Ext^1_{\D^b(\A,\S)}(C,A)$ for objects $C$ and $A$ in $H^0(\A)$.
We observe that if $[X]=[X']\in \mathbb E(C,A)$,  then $X$ and $X'$ give equivalent roofs. 
So we have a well-defined map 
$
\mathbb E(C,A)\rightarrow \Ext^1_{\D^b(\A,\S)}(C,A)
$.
From the above proof, this map is a surjection.
From the definition of the bifunctor $\mathbb E$, it is clear that this map is natural in the variables $C$ and $A$.
From the definition of the addition operation on $\mathbb E(C,A)$, it is clear that this map is additive.
Now suppose that $[X]$ is an element in $\mathbb E(C,A)$ which is sent to zero in $\Ext^1_{\D^b(\A,\S)}(C,A)$.  
Suppose that $X$ is of the form
\[
\begin{tikzcd}
A\ar[r,"f"]\ar[rr,bend right=8ex,"h"swap]&B\ar[r,"j"]&C.
\end{tikzcd}
\]
Then the morphism $\overline{f}/1:A\rightarrow B$ in $\D^b(\A,\S)$ is a split monomorphism.
Since $H^0(F)$ is fully faithful, the morphism $\overline{f}:A\rightarrow B$ is a split monomorphism in $H^0(\A)$. 
By $2)\Rightarrow 1)$ in \cite[Proposition 4.19]{Chen24} , the element $[X]$ is zero in $\mathbb E(C,A)$. 
This shows 2).
\end{proof}






\subsection{Algebraic extriangulated categories}
\begin{proposition-definition}\label{algebraic}Let $\C$ be an extriangulated category. The following statements are equivalent:
\begin{itemize}
\item[1)]$\C$ is equivalent, as an extriangulated category, to a full extension-closed subcategory of an algebraic triangulated category.
\item[2)]$\C$ is equivalent, as an extriangulated category, to $\B/(\mathcal P_0)$ for a Quillen exact category $\B$ and a class $\mathcal P_0$ of projective-injective objects.
\item[3)]$\C$ is equivalent, as an extriangulated category, to $(H^0(\A),\mathbb E,\mathfrak s)$ for an exact dg category $(\A,\S)$.
\end{itemize}
If one of the above equivalent conditions holds, then $\C$ is called an {\em algebraic extriangulated category}.
In this case, the exact dg category $(\A,\S)$ in 3) is called a {\em dg enhancement} for $\C$.
\end{proposition-definition}
\begin{proof}
Let us prove that 1) implies 2). Suppose the extriangulated category $(\C,\mathbb E,\mathfrak s)$ is equivalent to a full extension-closed subcategory of $\underline{\E}$ where $\E$ is a Frobenius Quillen exact category. 
Let $F:\C\rightarrow \underline{\E}$ be the fully exact functor which is an equivalence onto its essential image.
Let $\pi:\E\rightarrow \underline{\E}$ be the canonical quotient functor.
We form the following diagram
\[
\begin{tikzcd}
\B\ar[r,dashed,hook]\ar[d,dashed,two heads]&\E\ar[d,two heads,"\pi"]\\
\C\ar[r,hook,"F"swap]&\underline{\E}\mathrlap{\;.}
\end{tikzcd}
\]
Let $\B$ be the full subcategory of $\E$ consisting of objects $X$ such that $\pi(X)$ belongs to the essential image of $F$.
The category $\B$ is extension-closed in $\E$ since the essential image of $F$ is closed under extension and $\pi$ is a fully exact functor. 
Note also that $\B$ contains all projective-injective objects in $\E$.
So $\B$ is canonically exact and the full subcategory
\[
\P_0=\{P\in\B| \text{$P$ is projective-injective in $\E$}\}
\]
consists of projective-injective objects of $\B$ (but perhaps not all).
By \cite[Proposition 3.30]{NakaokaPalu19}, $\B/[\P_0]$ has the structure of an extriangulated category, induced from that of $\B$.
The functor $F$ induces a canonical equivalence of categories 
\[
G:\C\iso \B/[\P_0],\;\; A\mapsto F(A)
\]
which, by definition, is a fully exact functor.

Let us prove that conversely, condition 2) implies 1). Suppose $\C$ is equivalent to $\B/[\P_0]$ as extriangulated categories.
Then we have a fully exact fully faithful functor
\[
\C\iso\B/[\P_0]\hookrightarrow \D^b(\B)/\mathrm{thick}(\P_0).
\]
The triangle quotient $\D^b(\B)/\mathrm{thick}(\P_0)$ is an algebraic triangulated category and $\C$ identifies with an extension-closed subcategory of $\D^b(\B)/\mathrm{thick}(\P_0)$.

Let us show that 1) implies 3). Suppose $\C$ is equivalent to a full extension-closed subcategory of $H^0(\A')$ for a pretriangulated dg category $\A'$.
Let $F:\C\rightarrow H^0(\A')$ be a fully exact functor which is an equivalence onto its essential image.
Let $\A$ be the full dg subcategory of $\A'$ whose objects are those in the essential image of $F$.
Then $\A$ is an extension-closed subcategory of $\A'$, since $F$ is a fully exact functor. 
By Example \ref{exm:exactdg} 6), $\A$ inherits an exact dg structure from that of $\A'$.
By the definition of the biadditive functor $\mathbb E$ associated with an exact dg structure, 
we see that the inclusion $H^0(\A)\hookrightarrow H^0(\A')$ is a fully exact fully faithful functor.
Thus $\C$ is equivalent to $H^0(\A)$ as an extriangulated category where $\A$ is an exact dg category.

Let us show that conversely, condition 3) implies 1).
Suppose $\C$ is equivalent to $H^0(\A)$ for an exact dg category $\A$.
By \cite[Remark 4.25]{Chen24}, we may replace $\A$ with $\tau_{\leq 0}\A$ and assume that $\A$ is a connective dg category.
Then by Theorem \ref{main}, we have a fully exact fully embedding $H^0(\A)\hookrightarrow \D^b(\A,\S)$ where $\D^b(\A,\S)$ is an algebraic triangulated category.
\end{proof}
\begin{remark}
\begin{itemize}
\item[a)] By \cite[Remark 4.25]{Chen24}, each algebraic extriangulated category admits a connective dg enhancement.
\item[b)] Item 1) shows that the class of algebraic extriangulated categories is closed under forming extension-closed subcategories.
\item[c)] Item 2) shows that the class of algebraic extriangulated categories is closed under forming ideal quotients by projective-injective objects. 
\item[d)] For an ideal quotient category $\B/(\P_0)$, where $\B$ is a Quillen exact category and $\P_0$ is a class of projective-injective objects, it is a Frobenius extriangulated category \cite[Definition 7.1]{NakaokaPalu19} if and only if $\B$ is a Frobenius exact category.
In particular, we have that the stable category of an algebraic Frobenius extriangulated category is an {\em algebraic} triangulated category.
Notice that although the stable category of a Frobenius extriangulated category is a triangulated category, it is not necessarily an algebraic triangulated category.  
\end{itemize}
\end{remark}

In fact, the class of extriangulated categories can also be described using a seemingly more
general version of item 2) in Proposition-Definition~\ref{algebraic}.
This generalisation is suggested by \cite[Theorem 2.8]{FangGorskyPaluPlamondonPressland23a}, 
stating that if $\C$ is an extriangulated category and $\I$ is an ideal generated by a family of morphisms 
$f:I\rightarrow P$ from injectives to projectives, then $\C/\I$ is naturally extriangulated.
In the context of algebraic extriangulated categories, we have the following result.
\begin{proposition}\label{prop:idealquotientinjectiveprojective}
Let $\C$ be an algebraic extriangulated category. 
Let $\{f:I\rightarrow P\}$ be a family of morphisms from injectives to projectives in $\C$. 
Denote by $[I\rightarrow P]$ the ideal generated by this family of morphisms.
Then $\C/[I\rightarrow P]$ is still an algebraic extriangulated category.
\end{proposition}
\begin{proof} By item 2) of the preceding Proposition-Definition, we may assume
that $\C$ is a Quillen exact category.
In \cite[Proposition 1.7]{DraxlerReitenSmaloSolberg99}, it is shown that given a full subcategory $\N\subset \E$ of a Quillen exact category $\E$, the class of conflations which make the objects in $\N$ projective-injective 
is a Quillen exact structure on $\E$.
We endow $\Mor(\C)$ with the componentwise exact structure and view $\{f:I\rightarrow P\}$ as a 
full subcategory of $\Mor(\C)$.
Let $\F$ be the exact category whose underlying category is $\Mor(\C)$ and whose exact structure is 
obtained from $\Mor(\C)$ by making the objects $f:I\rightarrow P$ projective-injective.
Then $\F$ is an exact category.
We have the inclusion functor 
$
\theta:\C\rightarrow \F
$
sending an object $X$ to $(\Id:X\rightarrow X)$. Clearly, it induces a functor
$
\phi:\C/[I\rightarrow P]\rightarrow \F/(I\rightarrow P)
$.
We have the following observations:
\begin{itemize}
\item $\theta$ is an exact functor. 
Indeed, a conflation in $\C$ 
\[
\begin{tikzcd}
0\ar[r]&X\ar[r]&Y\ar[r]&Z\ar[r]&0
\end{tikzcd}
\]
is sent, by $\theta$, to a conflation in $\Mor(\C)$
\[
\begin{tikzcd}[row sep=small]
&&&I\ar[dd,blue]\ar[rd,blue]\ar[ld,dashed,red]&&\\
0\ar[r]&X\ar[r]\ar[dd,equal]&Y\ar[rr]\ar[dd,equal]&&Z\ar[r]\ar[dd,equal]&0\\
&&&P\ar[rd,blue]\ar[ld,dashed,red]&&\\
0\ar[r]&X\ar[r]&Y\ar[rr]&&Z\ar[r]&0
\end{tikzcd}
\]
For any morphism from $I\rightarrow P$ to $\Id:Z\rightarrow Z$, the morphism $P\rightarrow Z$ factors through $Y$ since $P$ is projective in $\C$.
This shows that the morphism from $I\rightarrow P$ to $\Id:Z\rightarrow Z$ factors through $\Id:Y\rightarrow Y$.
Hence the objects $I\rightarrow P$ are projective with respect to the conflations in the essential image of $\theta$. 
Dually, one shows that they are injective. 
Hence the functor $\theta$ sends conflations to conflations in $\F$.
\item The induced functor $\phi$ is also fully faithful. This follows by definition.
\item The full embedding $\phi:\C/[I\rightarrow P]\rightarrow \F/(I\rightarrow P)$ is fully exact.
Let $\Id_X:X\rightarrow X$ and $\Id_Z:Z\rightarrow Z$ be two objects in the image of the functor $\phi$.
Consider a conflation in $\F$
\[
\begin{tikzcd}
0\ar[r]&X\ar[r]\ar[d,equal]&Y\ar[r]\ar[d,"\alpha"]&Z\ar[r]\ar[d,equal]&0\\
0\ar[r]&X\ar[r]&Y'\ar[r]&Z\ar[r]&0\mathrlap{.}
\end{tikzcd}
\]
The map $\alpha$ is an isomorphism in $\C$. 
Hence $\alpha:Y\rightarrow Y'$ is isomorphic to $\Id:Y\rightarrow Y$ and the above conflation is equivalent to a conflation in the essential image of $\theta$.
\item By definition, the category $\F/(I\rightarrow P)$ is an algebraic extriangulated category.
\end{itemize}
Therefore, the functor $\phi$ is a fully exact fully faithful functor from $\C/[I\rightarrow P]$ to an algebraic extriangulated category and hence $\C/[I\rightarrow P]$ is also algebraic.
\end{proof}
\subsection{Higher extensions}
Let $(\C,\mathbb E,\mathfrak s)$ be a small extriangulated category.
In their paper \cite{GorskyNakaokaPalu21}, for $n\geq 1$, Gorsky--Nakaoka--Palu defined the higher 
extension bimodule $\mathbb E^{n}(?,-)$ on $\C$ as the $n$th tensor power over $\C$ of $\mathbb E$.
Recall that a right $\C$-module $F:\C^{op}\rightarrow \Ab$ is {\em weakly effaceable} if for each $C\in \C$ and each element $x\in F(C)$, there exists a deflation $p:B\rightarrow C$ such that $F(p)(x)=0$, cf.~\cite[A.2]{Keller90} and also~\cite[Definition 2.6]{KaledinLowen15}.
Dually one has the notion of weakly effaceable left module.
A $\C$-$\C$-bimodule $G:\C^{op}\otimes \C\rightarrow \Ab$ is {\em weakly effaceable}, if for each 
$C\in\C$, the left module $G(C,-)$ and the right module $G(?,C)$ are both weakly effaceable.
By \cite[Corollary 3.5]{NakaokaPalu19}, for each $A\in \C$, the right $\C$-module 
$\mathbb E(?,A):\C^{op}\rightarrow \Ab$ is weakly effaceable. It follows that the tensor powers
$\mathbb E^n(?,-)$ are weakly effaceable.

In the rest of this section, we will omit the symbol $^{\wedge}$ in the representable dg module $A^{\wedge}$ for $A\in\A$.
Recall the extriangulated category $(H^0(\A),\mathbb E,\mathfrak s)$ associated with the exact dg category $\A$.
Clearly $(\Ext^n_{\D^b(\A)}(?,-))_{n\geq 0}$ together with the connecting morphisms given by triangles in $\D^b(\A)$ induced by $\mathbb E$-triangles in $H^0(\A)$, is a $\delta$-functor  (cf.~Definition~\ref{def:deltafunctor}) for $H^0(\A)$.
It clearly satisfies the assumptions~\ref{assumption:delta2} and~\ref{assumption:delta3}.
\begin{proposition}\label{higher}
Let $\A$ be a connective exact dg category and $F:\A\rightarrow \D^{b}_{\dg}(\A)$ the universal exact morphism from $\A$ 
to a pretriangulated dg category. 
Then we have a canonical isomorphism of $\delta$-functors $\alpha:\mathbb E^n(?,-)\xrightarrow{\sim} \Ext^n_{\D^b(\A)}(?,-)$.\end{proposition}

\begin{proof}
Put $\C=H^0(\A)$. 
Let us show the $\C$-$\C$-bimodule $\Ext^n_{\D^b(\A)}(?,-)$ is effaceable for each $n\geq 1$.
By construction, we have $\D^b(\A)=\tr(\A)/\mathcal N$. 
Consider a morphism from $C$ to $\Sigma^{n}(A)$ in $\D^b(\A)$ given  by the roof 
\[
\begin{tikzcd}
&M\ar[Rightarrow,ld,"s",swap]\ar[rd,"w"]&\\
C&&\Sigma^n (A)
\end{tikzcd}
\]
where $s$ and $w$ are morphisms in $\tr(\A)$ and $N=\Cone(s)$ lies in $\N$.  
We may assume $N\in \N^{\leq 0}$. 
By Corollary \ref{t}, $N$ is an extension of its homology which are defective objects in $\mathcal H$. 
For defective objects $\tilde{N}$ we know that $\RHom_{\A}(\tilde{N}, A')$ is concentrated in degree $2$ for $A'\in \A$.  
So we may assume $N\in \mathcal N^{[-n+1,0]}$ and write $N=(\Sigma^{n-1}N_{n-1})*(\Sigma^{n-2}N_{n-2})*\cdots*N_{0}$ where $N_i$ is defective for each $i$. 
Suppose for each $i$, $N_{i}$ is the totalization of a conflation $X_i$:
\[
\begin{tikzcd}
A_i\ar[r,"f_i"]\ar[rr,"h_i"swap,bend right=6ex]&B_i\ar[r,"j_i"]&C_i.
\end{tikzcd}
\]
We have a canonical map $p:N_{n-2}\rightarrow\Sigma^{2}N_{n-1}$. 
Consider the following diagram in $\tr(\A)$:
\[
\begin{tikzcd}
\Sigma A_{n-2}\ar[r]&\Sigma V_{n-2}\ar[r]&N_{n-2}\ar[r]\ar[d,"p"near start, swap]&\Sigma^2 A_{n-2}\ar[ld,dotted]\ar[lld,dotted]\\
\Sigma^2 U_{n-1}\ar[r]&\Sigma^2 C_{n-1}\ar[r]&\Sigma^2 N_{n-1}\ar[r]&\Sigma^3U_{n-1}.
\end{tikzcd}
\]
Since $\A$ is connective, we have $\Hom_{\D(\A)}(V_{n-2},\Sigma N_{n-1})=0$ and $\Hom_{\D(\A)}(A_{n-2},\Sigma U_{n-1})=0$. 
Thus we have a morphism $b: A_{n-2}\rightarrow C_{n-1}$ which induces the morphism $p$. 

The map from $\Sigma^{-1}N$ to $\Sigma^n A$ gives rise to a morphism $N_{n-1}\rightarrow \Sigma^2(A)$. 
It induces a map $a:A_{n-1}\rightarrow A$ which is well-defined up to a morphism factoring through $\overline{f}_{n-1}$. 
Thus, we get an element $a_*[X_{n-1}]=\delta_{n-1}\in\mathbb E(C_{n-1},A)$.
We denote by $u$ the composition $C\rightarrow N\rightarrow N_{\geq -n+2}$. 
The map $\Sigma^{n-2}N_{n-2}\rightarrow N_{\geq -n+2}$ gives rise to a canonical map $N_{\geq -n+2}\rightarrow \Sigma^{n}A_{n-2}$. 
It induces a canonical map $v:\Cone(u)\rightarrow \Sigma^{n} A_{n-2}$. 
We have the following commuative diagram:
\[
\begin{tikzcd}
\Sigma^{-1}N_{n-1}\ar[d,equal]\ar[r]&\Sigma^{-n}N\ar[r]\ar[d]&\Sigma^{-n}N_{\geq -n+2} \ar[r]\ar[d]&N_{n-1}\ar[d,equal]\\
\Sigma^{-1}N_{n-1}\ar[r]\ar[d,equal]&\Sigma^{-(n-1)}M\ar[r]\ar[d,dashed,"x"]&\Sigma^{-n}\Cone(u)\ar[r]\ar[d,"y"]&N_{n-1}\ar[d,equal]\\
\Sigma^{-1}N_{n-1}\ar[r]&\Sigma A_{n-1}\ar[r]&\Sigma V_{n-1}\ar[r]&N_{n-1}
\end{tikzcd}
\]
where $y$ is the composition $\Sigma^{-n}\Cone(u)\rightarrow A_{n-2}\rightarrow C_{n-1}\rightarrow \Sigma V_{n-1}$.
Since the spaces $\Hom_{\D(\A)}(\Sigma^{-n}N_{\geq -n+2},\Sigma A)$ and $\Hom_{\D(\A)}(\Sigma^{-n+1}C,\Sigma A)$ both vanish, the space 
\[
\Hom_{\D(\A)}(\Sigma^{-n}\Cone(u),\Sigma A)
\]
 also vanishes and hence the composition $\Sigma^{-n+1}M\xrightarrow{x}\Sigma A_{n-1}\xrightarrow{\Sigma a} \Sigma A$ equals to the morphism $\Sigma^{-n+1}(w)$.
 We have the following diagram in $\D(\A)$:
 \[
 \begin{tikzcd}
 &&&\Sigma^{-n+1}M\ar[llldd,Rightarrow,"\Sigma^{-n+1}(s)"swap,bend right=6ex]\ar[rrdd,"x",bend left=9ex]\ar[rrrdd,"\Sigma^{-n+1}(w)",bend left=9ex]\ar[ld,Rightarrow]&&&\\
 &&\Sigma^{-n}\Cone(u)\ar[lld,Rightarrow,"\Sigma^{-n+1}(t)"]\ar[d,"\Sigma^{-n}(v)"]&&\Sigma V_{n-1}&&\\
 \Sigma^{-n+1}C&&A_{n-2}\ar[r,"b"]&C_{n-1}\ar[ru]&&\Sigma A_{n-1}\ar[lu,Rightarrow]\ar[r,"\Sigma(a)"]&\Sigma A.
 \end{tikzcd}
 \]
 Therefore, the morphism $w/s$ is equal to the composition of $\Sigma^{-1}(v)/t$ and a morphism $\Sigma^{n-1}A_{n-2}\rightarrow \Sigma^{n}A$. 
 Hence the left $\C$-module $\Ext^n_{\D^b(\A)}(C,-)$ is effaceable for each $n\geq 1$ and $C\in \C$. 
 Similarly one shows that the right $\C$-module $\Ext^n_{\D^b(\A)}(?,A)$ is effaceable for each $n\geq 1$ and $A\in \C$.
 The conclusion follows from Corollary~\ref{cor:effaceablebimodule}.
\end{proof}

\subsection{Quotients with respect to projective-injective objects}

\begin{lemma}
Let $\C$ be a connective dg category and $\C'\subseteq\C$ a full dg subcategory. 
Put $\B=\tr(\C)/\tr(\C')$. 
Then for $A,B\in \C$, we have 
$
\Hom_{\B}(A,\Sigma^n B)=0
$
 for $n>0$ and 
 \[
 \Hom_{\B}(A,B)\iso (H^0(\C)/[\C'])(A,B).
 \] 
\end{lemma}
\begin{proof}
Put $\T=\tr(\C)$.
Since $\C$ is connective, by Proposition \ref{cot} we have a canonical co-t-structure $(\T^{\geq 0}, \T^{\leq 0})$ on $\T$, which restricts to a co-t-structure on $\tr(\C')$. 
We denote by $\T^{-n}$ the $\Sigma^n$ of the coheart of the co-t-structure.

Suppose we have a morphism in $\B$ as follows
\[
\begin{tikzcd}&M\ar[ld,Rightarrow,"s"swap]\ar[rd,"b"]&\\
A&&\Sigma^nB
\end{tikzcd}
\]
where $s$ and $b$ are morphisms in $\tr(\C)$ and $P=\Cone(s)$ lies in $\tr(\C')$. 
Take a weight decomposition of $P: \sigma_{>-1}(P)\rightarrow P\rightarrow \sigma_{\leq -1}P\rightarrow \Sigma \sigma_{\geq -1}P$. 
Since $A\in H^0(\C)$ lies in the coheart of the co-t-structure, $A\rightarrow P$ factors through $\sigma_{>-1}(P)$. 
Hence we may assume $P\in \T^{>-1}$. 
Note that for $n\geq 0$ we have $\Sigma^{n}B\in \T^{-n}$ and hence the map $\Sigma^{-1}P\rightarrow\Sigma^{n}B$ vanishes.
Therefore, the map $b$ factors through $A$ for $n\geq 0$. 
So for $n>0$, we have $\Hom_{\B}(A,\Sigma^{n}B)=0$. 

Now assume $n=0$. 
If a map $f:A\rightarrow B$ in $\tr(\C)$ factors through an object in $\tr(\C')$, then it factors through an object in the coheart, i.e.,~a retract of an object $W\in \C'$ in $\tr(\C')$. 
Therefore, it factors through an object in $\C'$. 
So we have 
\[
\Hom_{\B}(A,B)=(H^0(\C)/[\C'])(A,B).
\]  
\end{proof}

\begin{definition}
An object $P$ in $(H^0(\A),\mathbb E,\mathfrak s)$ is said to be {\em projective} if for each conflation 
\[
\begin{tikzcd}
X\ar[r,"f"]\ar[rr,"h"swap, bend right=8ex]&Y\ar[r,"j"]&Z
\end{tikzcd} 
\]
the morphism $\overline{\jmath}$ induces a surjection $\Hom_{\D(\A)}(P,Y)\rightarrow \Hom_{\D(\A)}(P,Z)$, or equivalently, if $P$ is projective in the extriangulated category $(H^0(\A),\mathbb E,\mathfrak s)$.
Dually, we define {\em injective} objects in $(H^0(\A),\mathbb E,\mathfrak s)$.
\end{definition}

\begin{lemma}\label{lem:leftorthogonal}
Let $(\A,\S)$ be a connective exact dg category.
Then an object $P$ in $H^0(\A)$ is projective if and only if $\Hom_{\tr(\A)}(P,\mathcal N)=0$ where $\mathcal N$ is the full triangulated subcategory of $\tr(\A)$ generated by the total dg modules of conflations.
\end{lemma}
\begin{proof}Notice that the totalization $N$ of a conflation
\[
\begin{tikzcd}
X\ar[r,"f"]\ar[rr,"h"swap,bend right=6ex]&Y\ar[r,"j"]\ar[r]&Z
\end{tikzcd} 
\]
lies in the heart of the canonical t-structure of $\D(\A)$ and is the cokernel of the map
\[
\Hom_{\D(\A)}(-,Y)\rightarrow \Hom_{\D(\A)}(-,Z).
\]
Also for objects $A'$ in $H^0(\A)$, we have $\Hom_{\tr(\A)}(A', \Sigma^i N)=0$ for $i\neq 0$. 
So $P$ is projective if and only if $\Hom_{\D(\A)}(P,N)=0$ for each totalization $N$ of a conflation.
\end{proof}

Let $F:\A\rightarrow \D^b_{\dg}(\A,\S)$ be the universal exact morphism from a connective exact dg category $\A$ into a pretriangulated dg category $\D^b_{\dg}(\A,\S)$. 
Recall that we put $\D^b(\A,\S)=H^0(\D^b_{\dg}(\A,\S))$. 
Let $(\T^{\geq 0}, \T^{\leq 0})$ be the canonical co-t-structure  on $\T=\tr(\A)$ introduced in Proposition \ref{cot}.
Since we have
\[
\T^{\geq 0}=\bigcup_{i\geq0}\mathcal{H}[-i]*\mathcal{H}[-i+1]*\cdots*\mathcal{H}
\]
 where $\mathcal{H}$ is the coheart of the co-t-structure consisting of direct summands of $A$ in $\tr(\A)$ for $A\in\A$, it follows from Theorem~\ref{main} that $\tau_{\leq 0}\RHom_{\A}(A, M)$ is quasi-isomorphic to $\tau_{\leq 0}\Hom_{\D^b_{\dg}(\A,\S)}(A,M)$ for $A\in \A$ and $M\in \T^{\geq 0}$. 
If $P$ is a projective object in the extriangulated category $H^0(\A)$, then by Lemma~\ref{lem:leftorthogonal} $P$ lies in the left orthogonal of $\N$ and we have 
$
\Hom_{\tr(\A)}(P,X)\xrightarrow{\sim} \Hom_{\D^b(\A,\S)}(FP, X)
$
 for any $X\in\pretr(\A)$. 

\begin{lemma}\label{lem:quotientprojectiveinjective}
Let $\P$ be a full dg subcategory of $\A$ consisting of projective-injective objects in $H^0(\A)$. 
Let $\B=\tr(\A)/\tr(\P)$ and $\B'=\D^b(\A,\S)/\tr(\P)$.  
Then for $A,B\in\A$, the following statements hold:
\begin{itemize}
\item[1)]The canonical map $\Hom_{\B}(A, \Sigma^n B)\rightarrow\Hom_{\B'}(A,\Sigma^n B)$ is bijective for $n\leq 0$.
\item[2)]The canonical map $\Hom_{\D^b(\A,\S)}(A,\Sigma^{n}B)\rightarrow\Hom_{\B'}(A,\Sigma^nB)$ is bijective for {$n\geq 1$}.
\end{itemize}
\end{lemma}
\begin{proof}Consider a morphism in $\B'$
\[
\begin{tikzcd}&M\ar[ld,Rightarrow,"s",swap]\ar[rd,"b"]&\\
A&&\Sigma^nB
\end{tikzcd}
\]
where $s$ and $b$ are morphisms in $\D^b(\A,\S)$ and $P=\Cone(s)\in \tr(\P)$. 
Put $\T=\tr(\A)$. Since objects in $\P$ are injective, the quotient functor $F:\T\rightarrow \D^b(\A,\S)\iso \T/\mathcal N$ induces bijections $\Hom_{\T}(X,Q)\xrightarrow{\sim} \Hom_{\D^b(\A,\S)}(X,Q)$ for $X\in \T$ and $Q\in \tr(\P)$. We may assume $s=F(s')$ for some morphism $s'$ in $\T$. 
Let $(\T^{\geq 0}, \T^{\leq 0})$ be the canonical co-t-structure on $\T$. 
Since $A\in \T^{\geq 0}$, we may assume $P\in \T^{\geq 0}$. 

1) Assume $n\leq 0$. Consider the triangle $\Sigma^{-1}P\rightarrow M\xrightarrow{s} A\rightarrow P$ in $\T$ and the induced long exact sequence
\[
\begin{tikzcd}
\cdots\ar[r]&\T(M,\Sigma^{-1}B)\ar[r]\ar[d]&\T(P,B)\ar[r]\ar[d,"\sim"]&\T(A,B)\ar[r,two heads]\ar[d,"\sim"]&\T(M,B)\ar[d]\\
\cdots\ar[r]&(M,\Sigma^{-1}B)\ar[r]&(P,B)\ar[r]&(A,B)\ar[r,two heads]&(M,B)
\end{tikzcd}
\]
where the Hom spaces in the second row are in $\D^b(\A,\S)$.
By the Five-Lemma, we have 
\begin{align}
\Hom_{\T}(M,\Sigma^{n}B)\xrightarrow{\sim} \Hom_{\D^b(\A,\S)}(M,\Sigma^{n}B) \label{projinj}
\end{align}
for $n\leq 0$. 
So we have $b=F(b')$ for some morphism $b'$ in $\T$. 
This shows that the canonical map
$
\Hom_{\B}(A, \Sigma^n B)\rightarrow\Hom_{\B'}(A,\Sigma^n B)
$
 is a surjection for $n\leq 0$. 
Assume that the morphism $b=F(b'):M\rightarrow \Sigma^n B$ in $\D^b(\A,\S)$ factors through an object in $\tr(\P)$.  
By bijection (\ref{projinj}) and since $\P$ consists of projective-injective objects, the morphism $b'$ in $\T$ also factors through an object in $\tr(\P)$.
 So the canonical map $\Hom_{\B}(A, \Sigma^n B)\rightarrow\Hom_{\B'}(A,\Sigma^n B)$ is a bijection for each $n\leq 0$.

2) Assume $n\geq 1$. 
Since $\Hom_{\D^b(\A,\S)}(\Sigma^{-1}P, \Sigma^{n} B)=0$ for $P\in \T^{\geq 0}$, the morphism $b:M\rightarrow \Sigma^{n}B$ factors through $A$. 
So the map $\Hom_{\D^b(\A,\S)}(A,\Sigma^{n}B)\rightarrow\Hom_{\B'}(A,\Sigma^nB)$ is a surjection.
Let $f:A\rightarrow \Sigma^n B$ be a morphism in $\D^b(\A,\S)$ which factors through an object $Q$ in $\tr(\P)$. 
We may assume $Q\in \tr(\P)^{\geq 0}$. 
Then we have $\Hom_{\D^b(\A,\S)}(Q,\Sigma^n B)=0$ and hence the morphism $f$ is zero.
So the canonical map 
\[
\Hom_{\D^b(\A,\S)}(A,\Sigma^{n}B)\rightarrow\Hom_{\B'}(A,\Sigma^nB)
\]
 is a bijection for $n\geq 1$.
\end{proof}

Let $\P$ be a full dg subcategory of a connective exact dg category $\A$ consisting of projective-injective objects. 
Then we have the following diagram
\[
\begin{tikzcd}&\tr(\A)\ar[r,"F_1"]\ar[d,"G_1"swap]&\D^b(\A,\S)\ar[d,"G_2"]\\
H^0(\A)/[H^0(\P)]\ar[r, hook,"H"swap]&\tr(\A)/\tr(\P)\ar[r,"F_2"swap]&\D^b(\A,\S)/\tr(\P).
\end{tikzcd}
\]
In summary, we have
\begin{proposition}\label{prop:dgsingularitycategory}
Let $A,B\in \A$. The following statements hold:
\begin{itemize}
\item[1)]$H$ is fully faithful.
\item[2)]For $n\geq 1$, $\Hom_{\tr(\A)/\tr(\P)}(A,\Sigma^n B)=0$.
\item[3)]For $n\leq 0$, $\Hom_{\tr(\A)/\tr(\P)}(A,\Sigma^n B)\xrightarrow{\sim}\Hom_{\D^b(\A,\S)/\tr(\P)}(A,\Sigma^n B)$.
\item[4)]For $n\geq 1$, $\Hom_{\D^b(\A,\S)}(A,\Sigma^n B)\xrightarrow{\sim}\Hom_{\D^b(\A,\S)/\tr(\P)}(A,\Sigma^n B)$. 
In particular, the essential image of $F_2\circ H$ is closed under extension in $\D^b(\A,\S)/\tr(\P)$. 
\item[5)]For $n\leq 0$, $\Hom_{\tr(\A)}(A,\Sigma^n B)\xrightarrow{\sim} \Hom_{\D^{b}(\A,\S)}(A,\Sigma^n B)$.
\end{itemize}
\end{proposition}
For an extriangulated category $\C$ with a subcategory $\P$ consisting of (not necessarily all) projective-injective objects in $\C$, Nakaoka--Palu showed that the ideal quotient $\C/[\P]$ has the structure of an extriangulated category, induced from that of $\C$, cf.~\cite[Proposition 3.30]{NakaokaPalu19}. In the context of exact dg categories, we have the following theorem. 
\begin{theorem}\label{quot}
Let $(\A,\S)$ be a connective exact dg category and $\P$ a full dg subcategory of $\A$ consisting of projective-injective objects in $\A$.
Let $\mathcal T_{\dg}$ be the canonical dg enhancement of $\D^b(\A,\S)/\tr(\P)$.
Then the dg quotient $\A/\P$ carries a canonical exact structure $(\A/\P,\overline{\S})$ induced from $(\A,\S)$ and its dg derived category is quasi-equivalent to $\mathcal T_{\dg}$.
\end{theorem}
\begin{proof}
By the items $(1)$ and $(3)$ of Proposition~\ref{prop:dgsingularitycategory}, we have $H^0(\A/\P)=H^0(\A)/[\P]$ and the canonical morphism $\varphi:\A/\P\rightarrow \mathcal T_{\dg}$ in $\Hqe$ is quasi-fully faithful. 
By item (4), the quasi-essential image of $\varphi$ is closed under extensions in $\T_{\dg}$ and therefore $\A/\P$ carries a canonical exact dg structure.
Again by item $(4)$ and Proposition \ref{higher}, the bounded dg derived category of $\A/\P$ is quasi-equivalent to $\mathcal T_{\dg}$.
\end{proof}
In particular, we obtain a short exact sequence of triangulated categories
\[
\begin{tikzcd}
0\ar[r]&\tr(\P)\ar[r]&\D^b(\A,\S)\ar[r]&\D^b(\A/\P,\overline{\S})\ar[r]&0.
\end{tikzcd}
\]

\begin{remark}
Let $(\E,\mathbb E,\mathfrak s)$ be an algebraic extriangulated category. 
By Proposition-Definition~\ref{algebraic} we have that $\E$ is equivalent, as an extriangulated category, to $\B/(\mathcal P_0)$ for a Quillen exact category $\B$ and a class $\mathcal P_0$ of projective-injective objects. 
Then by Proposition~\ref{prop:dgsingularitycategory}, we could take the dg quotient $\B/\P_0$ as a dg enhancement of $(\E,\mathbb E,\mathfrak s)$.
\end{remark}
\begin{corollary}\label{cor:connectivequotient}
Each connective exact dg category is exactly quasi-equivalent to a dg quotient $\E/\P_0$, where $\E$ is a Quillen exact category and $\P_0$ is a full subcategory of projective-injectives in $\E$, with the exact dg structure in Theorem~\ref{quot}. In particular, each connective additive dg category is quasi-equivalent to a dg quotient $\E/\P_0$ of an additive category $\E$ by a full subcategory $\P_0$.
\end{corollary}
\begin{proof}
Let $\A$ be a connective exact dg category. 
By Theorem~\ref{main}, we have a quasi-equivalence $\A\iso \tau_{\leq 0}\D'$ where $\D'$ is an extension-closed subcategory of a (strictly) pretriangulated dg category $\D=\D_{\dg}^b(\A)$.
Then $Z^0(\D)$ is a Frobenius exact category with the contractible objects being the projective-injectives, when endowed with the graded-split exact structure. 
Its stable category is isomorphic to the triangulated category $H^0(\D)$.
Note that $\Ext^1_{Z^0(\D)}(?,-)$ of this exact structure is naturally isomorphic to $\Hom_{H^0(\D)}(?,\Sigma(-))$.
Since $Z^0(\D')$ is extension-closed in $Z^0(\D)$, it inherits a Quillen exact structure with respect to which all the contractibles are projective-injective.
We denote by $\ctr(\D)\subset Z^0(\D')$ the full subcategory consisting of all the contractibles in $\D$.
We have the following diagram in $\Hqe$
\[
\begin{tikzcd}
Z^0(\D')\ar[r]\ar[d]&\pretr(Z^0(\D'))\ar[r]\ar[d,dashed]&\pretr(Z^0(\D'))/\ctr(\D)\ar[r]\ar[ld,dashed,"G"swap]&\D^b_{\dg}(Z^0(\D'))/\ctr(\D)\ar[lld,dashed,"F"]\\
\D'\ar[r,hook]&\D&&
\end{tikzcd}.
\]
Put $\T=H^0(\D^b_{\dg}(Z^0(\D'))/\ctr(\D))$.
By Proposition~\ref{higher} and Proposition~\ref{prop:dgsingularitycategory}, the morphism $F$ induces a canonical isomorphism
\[
\Hom_{\T}(X,\Sigma^iY)\iso \Hom_{H^0(\D)}(X,\Sigma^iY)
\]
for $X$ and $Y$ in $Z^0(\D')$ and $i\geq 1$. 
Let us show that it is also true for $i\leq 0$.
For this we work with $\T'=\tr(Z^0(\D'))/\tr(\ctr(\D))$ instead of $\T$ (and the morphism $G$ instead of $F$), by Proposition~\ref{prop:dgsingularitycategory} (3).
Suppose we have a morphism $f:\Sigma^{-i}X\rightarrow Y$ in $H^0(\D)$ for $i\leq 0$. 
Recall that we denote by $IX=\Cone(\Id_{X})$.
We have the following sequence in $Z^0(\D')$
\[
\begin{tikzcd}[column sep=small]
X\ar[rr,tail]&&IX\ar[rd,two heads,dashed]\ar[rr]&&I^2X\ar[rd,two heads,dashed]\ar[rr]&&I^3X\ar[rr]&&\ldots\ar[rr]&&I^{-i}X\ar[rd,two heads,dashed]&\\
&&&\Sigma X\ar[ru, tail,dashed]&&\Sigma^2 X\ar[ru,tail,dashed]&&&&&&\Sigma^{-i}X
\end{tikzcd}
\]
The morphism $f:\Sigma^{-i}X\rightarrow Y$ is represented by a cocylce $g:\Sigma^{-i}X\rightarrow Y$.
Then $g$ together with the projection of the above sequence to $X$ yields a morphism $\Sigma^{-i}X\rightarrow Y$ in $\T'$.
It is clear that its image under the functor $H^0(G):\T\rightarrow H^0(\D)$ is $f$.

Suppose we have a morphism $X\rightarrow \Sigma^iY$
\[
\begin{tikzcd}
&Z\ar[rd,"a"]\ar[ld,Rightarrow,"s"swap]&\\
X&&\Sigma^i Y
\end{tikzcd}
\]
 in $\T'$ whose image under $H^0(G)$ is zero in $H^0(\D)$.
 We may assume that $Z$ is of the form
 \[
 \begin{tikzcd}
 \ldots\ar[r]&0\ar[r]&X\ar[r]&IX\ar[r]&I^2X\ar[r]&\ldots\ar[r]&I^{-i}X\ar[r]&\ldots
 \end{tikzcd}
 \]
 where $X$ is in degree zero, and that the morphism $s:Z\rightarrow X$ is given by the projection onto $X$.
 Then the morphism $a$ is induced by a morphism $I^{-i}X\rightarrow Y$
 whose precomposition with $I^{-i-1}X\rightarrow I^{-i}X$ is zero.
 So it induces a cocyle $\Sigma^{-i}X\rightarrow Y$ which by assumption is a coboundary. 
 Hence the morphism $a$ in $\D^b(Z^0(\D'))$ factors through $I^{-i+1}X$ and the morphism $a/s:X\rightarrow \Sigma^iY$ is zero in $\T'$.
 
 Since the objects in $D'$ form a set of generators for the triangulated categories $\T$ and $H^0(\D)$, it follows that $F$ is a quasi-equivalence.
 It follows that $\tau_{\leq 0}\D'$ is quasi-equivalent to the dg quotient $Z^0(\D')/\ctr(\D)$. 
 For the second statement, we simply notice that each additive dg category 
 admits the trivial exact structure given by the closure under isomorphisms of the class of all split exact sequences.
\end{proof}
\subsection{Examples}
We illustrate the results in this section through the following classes of examples.
\subsubsection{Stable dg categories}\label{exm:exactdgpretriangulated}
Let $(\A,\mathcal S)$ be a small connective exact dg category such that the associated extriangulated category $H^0(\A)$ is triangulated. 
By \cite[Theorem 6.5]{Chen24}, $\A$ is a stable dg category with the canonical exact structure $\S$.
We show that $\A$ is quasi-equivalent to the $\tau_{\leq 0}$-truncation of a pretriangulated dg category.
\begin{proposition}\label{prop:stable}
The connective dg category $\A$ is quasi-equivalent to $\tau_{\leq 0}\D^b_{\dg}(\A,\S)$.
\end{proposition}
\begin{proof}
The canonical functor $H^0(\A)\rightarrow \D^b(\A,\S)=H^0(\D^b_{\dg}(\A,\S))$ is a triangle functor and it is fully faithful by Theorem~\ref{main}. 
Its essential image contains a family of generators for $\D^b(\A,\S)$ and hence it a triangle equivalence. 
By Theorem~\ref{main}, we have that $\A$ is quasi-equivalent to $\tau_{\leq 0}\D^b_{\dg}(\A,\S)$.
\end{proof}

Similarly, one shows that if $\A$ is the $\tau_{\leq 0}$-truncation of a pretriangulated dg category $\T$ with the induced exact structure, then $\T$ is the bounded dg derived category of $\A$ via the canonical dg functor $\A\rightarrow \T$. 
\begin{corollary}
Let $(\A,\S)$ be a connective Frobenius exact dg category, i.e.,~the associated extriangulated category $(H^0(\A),\mathbb E,\mathfrak s)$ is Frobenius.
Let $\P$ be its full dg subcategory of projectives.
Endow $\A/\P$ with the exact structure $\tilde{\S}$ given by Theorem~\ref{quot}.
Then we have an exact quasi-equivalence $\A/\P\iso \tau_{\leq 0}(\D^b_{\dg}(\A,\S)/\pretr(\P))$.
\end{corollary}
\begin{proof}
Since $H^0(\A/\P)=H^0(\A)/[H^0(\P)]$ is a triangulated category, $(\A/\P,\tilde{\S})$ is a stable dg category with the canonical exact structure. 
By Proposition~{prop:stable} and Theorem~\ref{quot}, we have 
\[
\A/\P\iso \tau_{\leq 0}\D^b_{\dg}(\A/\P,\tilde{\S})\iso \tau_{\leq 0}(\D^b_{\dg}(\A,\S)/\pretr(\P)).
\]
\end{proof}
\subsubsection{Cohen--Macaulay dg modules}\label{exm:CMdgmodule}
This example is taken from \cite{Jin20}.
Let $k$ be a field. 
A dg $k$-algebra $A$ is {\em proper} if $\sum_{i\in \mathbb Z}\dim_{k}{H^i(A)<\infty}$. 
It is {\em Gorenstein} if the thick subcategory $\per(A)$ of the derived category $\D(A)$ 
generated by $A$ coincides with the thick subcategory $\thick(DA)$ generated by $DA$, where $D=\Hom_k(-,k)$ is the $k\mbox{-dual}$.
Let $A$ be a connective proper Gorenstein dg algebra.
A dg $A\mbox{-}$module is {\em perfectly valued} if its total cohomology is finite-dimensional and we denote by $\pvd(A)$ the triangulated category of perfectly valued dg $A\mbox{-}$modules.
A dg $A\mbox{-}$module $M$ in $\pvd(A)$ is {\em {Cohen--Macaulay}} if $H^i(M)=0$ and $\Hom_{\D(A)}(M,\Sigma^i A)=0$ for $i>0$.
Let $\CM A$ be the subcategory of $\pvd(A)$ consisting of Cohen-Macaulay dg $A\mbox{-}$modules.
By \cite[Theorem 2.4]{Jin20}, the category $\CM A$ is an $\Ext\mbox{-}$finite Frobenius extriangulated category with $\proj(\CM A)=\add A$. 
Let $\pvd_{\dg}(A)$ be the canonical dg enhancement of $\pvd(A)$ and $\A\subseteq \pvd(A)$ the full dg subcategory consisting of Cohen--Macaulay dg modules. 
Put $\CM_{\dg}(A)=\tau_{\leq 0}\A$. 
Since $\A$ is extension-closed in $\pvd_{\dg}(A)$, the dg category $\CM_{\dg}(A)$ inherits a canonical exact structure $\S$ whose corresponding extriangulated category is $\CM A$. 
\begin{proposition} \label{prop:CMderived}
The bounded dg derived category of $(\CM_{\dg}(A),\S)$ is quasi-equivalent to $\pvd_{\dg}(A)$. 
\end{proposition}
\begin{proof}
Clearly, the canonical dg functor $\CM_{\dg}(A)\rightarrow \pvd_{\dg}(A)$ is exact. 
By Theorem~\ref{main}, it induces a canonical morphism $\D^b_{\dg}(\CM_{\dg}(A),\S)\rightarrow \pvd_{\dg}(A)$ in $\Hqe$.
For objects $X$ and $Y$ in $\CM(A)$, we have a canonical bijection $\mathbb E(X,Y)\iso \Hom_{\pvd(A)}(X,\Sigma Y)$.
Note that for projective objects $P\in \add(A)$, the Hom spaces $\Hom_{\D(A)}(P,\Sigma^{i} M)$ vanish for $i>0$ and $M\in\CM A$, by the definition of Cohen--Macaulay dg modules.
Since the exact dg category $\CM_{\dg}(A)$ has enough projectives, we have bijections $\mathbb E^{i}(X,Y)\iso \Hom_{\pvd(A)}(X,\Sigma^iY)$ for $i\geq 1$.
By \cite[Lemma 3.9 (2)]{Jin20}, the triangulated category $\pvd(A)$ is generated by $\CM A$.
Therefore the above morphism $\D^b_{\dg}(\CM_{\dg}(A),\S)\rightarrow \pvd_{\dg}(A)$ is a quasi-equivalence.
\end{proof}
 This equivalence induces an equivalence from the singularity category of $\CM_{\dg}(A)$ onto that of $A$ in complete analogy with
the situation for a finite-dimensional Iwanaga–Gorenstein algebra concentrated in degree
$0$.

\subsubsection{t-structures}
Let $\T$ be a pretriangulated dg category and $\A'$ a full dg subcategory such that $H^0(\A')$ is an aisle in the triangulated category $H^0(\T)$.
Clearly $\A'$ is stable under extensions in $\T$ and hence inherits a canonical exact structure $(\A',\S)$, cf.~Example~\ref{exm:exactdg} 6).
Put $\A=\tau_{\leq 0}\A'$. 
By \cite[Remark 4.5 a)]{Chen24}, we obtain a canonical exact structure $(\A,\S)$.
The composite dg functor $\A\rightarrow \A'\hookrightarrow \T$ is exact and hence induces a canonical morphism $\D^b_{\dg}(\A,\S)\rightarrow \T$ in $\Hqe$.
\begin{proposition}\label{prop:t-structure}
The canonical morphism $\D^b_{\dg}(\A,\S)\rightarrow \T$ is a quasi-equivalence.
\end{proposition}
\begin{proof}
Let $X$ be an object in $\A$.
We view it as an object in the triangulated category $H^0(\T)$.
We have the triangle $X\rightarrow 0\rightarrow \Sigma X\xrightarrow {\Id_{\Sigma X}}\Sigma X$ in $H^0(\A)$.
Since $\Sigma X$ is also an object in $H^0(\A)$, this means that $H^0(\A)$ is an extriangulated category with enough injectives and the injectives are the zero objects.
Now, the same reasoning as in the proof of Proposition~\ref{prop:CMderived} shows that the morphism $\D^b_{\dg}(\A,\S)\rightarrow \T$ is a quasi-equivalence.
\end{proof}
\subsubsection{Relative cluster categories as derived categories}\label{subsection:higgscategories}
One of the motivations for our work is Yilin Wu's construction of the Higgs category associated with a finite ice quiver with potential (see \cite{Wu23a}\cite{Wu21} for the Jacobi-finite case and \cite{KellerWu23} for the Jacobi-infinite case). 

Let $(Q,F,W)$ be a Jacobi-finite ice quiver with potential. 
Denote by $\Gamma_{\rel}$ the relative Ginzburg dg algebra $\Gamma_{\rel}(Q,F,W)$. 
Let $e=\sum_{i\in F}e_i$ be the idempotent associated with all frozen vertices. 
Let $\pvd_{e}(\Gamma_{\rel})$ be the full subcategory of $\pvd(\Gamma_{\rel})$ of the dg $\Gamma_{\rel}$-modules whose restriction to frozen vertices is acyclic.
Then the {\em relative cluster category} $\C=\C(Q,F,W)$ associated to $(Q,F,W)$ is defined as the Verdier quotient of triangulated categories
$
\per(\Gamma_{\rel})/\pvd_{e}(\Gamma_{\rel}).
$
The {\em relative fundamental domain} $\mathcal F^{\rel}_{\Gamma_{\rel}}=\mathcal F^{\rel}$ associated to $(Q,F,W)$ is defined as the following subcategory of $\per(\Gamma_{\rel})$
\[
\mathcal F^{\rel}{\coloneqq}\{\Cone(X_1\xrightarrow{f}X_0) | X_i\in\add(\Gamma_{\rel}) \text{ and } \Hom(f,I) \text{ is surjective}, \forall I\in \mathcal P{\coloneqq}\add(e\Gamma_{\rel})\}.
\]
It is clear that for objects $X$ and $Y$ in $\F^{\rel}$, every morphism $f:X\rightarrow \Sigma Y$ in $\per(\Gamma_{\rel})$ factors through an object in $\add(\Sigma \Gamma_{\rel})$.
Let $\pi^{\rel}:\per(\Gamma_{\rel})\rightarrow \C$ be the canonical quotient functor. 
By \cite[Proposition 5.20]{Wu23a}, the restriction of the quotient functor $\pi^{\rel}$ to $\mathcal F^{\rel}$ is fully faithful. 
In particular for objects $X$ and $ Y$ in $\mathcal F^{\rel}$, the functor $\pi^{\rel}$ induces an injection 
$
\Ext^1_{\per(\Gamma_{\rel})}(X,Y)\rightarrow\Ext^1_{\C}(\pi^{\rel}X,\pi^{\rel}Y).
$
The {\em Higgs category} (\cite[Definition 5.22]{Wu23a}) $\mathcal H$ is the image of $\mathcal F^{\rel}$ in $\C$ under the quotient functor $\pi^{\rel}$. 
Indeed, we have the following
\begin{lemma}\label{equi}Let $X,Y$ be objects in $\mathcal F^{\rel}$.
Then $\pi^{\rel}$ induces a canonical isomorphism 
\[
\Hom_{\per(\Gamma_{\rel})}(X,\Sigma^{i}Y)\xrightarrow{\sim}\Hom_{\C}(\pi^{\rel}X,\Sigma^{i}\pi^{\rel}Y).
\] for any $i\leq 0$ or for $Y\in \add\Gamma_{\rel}$ and $i\leq 1$.

\end{lemma}
\begin{proof}
We first show the case when $Y=\Gamma_{\rel}$ and $i\leq 1$. 
Recall the relative t-structure $(\D(\Gamma_{\rel})^{\rel}_{\leq 0}, \D(\Gamma_{\rel})^{\rel}_{\geq 0})$ on $\D(\Gamma_{\rel})$ which restricts to t-structures on $\per (\Gamma_{\rel})$ and $\pvd_{e}(\Gamma_{\rel})$, cf.~\cite[Propositions 4.10, 4.11]{Wu23a}.
We show the injectivity of the map.
Let $f:X\rightarrow \Sigma^{i}\Gamma_{\rel}$ be a map in $\per(\Gamma_{\rel})$ which factors through an object $M$ in $\pvd_{e}(\Gamma_{\rel})$. 
We want to show that $f$ is indeed a zero map.
By definition, we have $\Gamma_{\rel}\in \per(\Gamma_{\rel})^{\rel}_{\leq 0}$ and hence $X\in \add \Gamma_{\rel}*\add \Sigma\Gamma_{\rel}\subset \per(\Gamma_{\rel})^{\rel}_{\leq 0}$.
So we may suppose $M\in \pvd_{e}(\Gamma_{\rel})^{\rel}_{\leq 0}$.
By \cite[Corollary 3.13]{Wu23a}, we have 
$
\Hom_{\Gamma_{\rel}}(M,\Sigma^i\Gamma_{\rel})\xrightarrow{\sim}D\Hom_{\Gamma_{\rel}}(\Gamma_{\rel}, \Sigma^{3-i}M)=0
$ 
where $D=\Hom(-,k)$. 
This shows the injectivity of the morphism. 
We now show the surjectivity of the map. Consider a roof $b/s$ 
\[
\begin{tikzcd}
&N\ar[ld,Rightarrow,"s"swap]\ar[rd,"b"]&\\
X&&\Sigma^i\Gamma_{\rel}
\end{tikzcd}.
\]
Since $X\in\per(\Gamma_{\rel})^{\rel}_{\leq 0}$, we may assume $\Cone(s)\in \pvd(\Gamma_{\rel})^{\rel}_{\leq 0}$.
Then since we have
\[
\Hom_{\Gamma_{\rel}}(\Sigma^{-1}\Cone(s),\Sigma^{i} \Gamma_{\rel})\iso D\Hom_{\Gamma_{\rel}}(\Gamma,\Sigma^{2-i} \Cone(s))=0,
\] 
the morphism $b$ factors through $s$ and this shows the surjectivity of the map.

For general $Y\in \mathcal F^{\rel}$ and $i\leq 0$, it follows from the above case by applying the homological functor $\Hom(X,-)$ to the defining triangle for $Y$.
\end{proof}
Let $\per_{\dg}(\Gamma_{\rel})$ respectively $\pvd_{e}(\Gamma_{\rel})_{\dg}$ be the canonical dg enhancement of $\per(\Gamma_{\rel})$ respectively $\pvd_{e}(\Gamma_{\rel})$. 
Let $\C_{\dg}$ be the canonical dg enhancement of $\C$ given by the dg quotient $\per_{\dg}(\Gamma_{\rel})/\pvd_{e}(\Gamma_{\rel})_{\dg}$. 
Let $\mathcal H_{\dg}$ be the full dg subcategory of $\C_{\dg}$ consisting of objects in $\mathcal H$.
Let $\F^{\rel}_{\dg}$ be the dg full subcategory of $\per_{\dg}(\Gamma_{\rel})$ consisting of objects in $\F^{\rel}$.
By \cite[Definition 5.12, Proposition 5.39]{Wu23a}, the dg category $\F^{\rel}_{\dg}$ (resp.,~$\mathcal H_{\dg}$) is extension-closed in $\per_{\dg}(\Gamma_{\rel})$ (resp.,~$\C_{\dg}$) and thus inherits a canonical exact dg structure. 
By Lemma \ref{equi}, we obtain the following.
\begin{corollary}
The canonical morphism $\per_{\dg}(\Gamma_{\rel})\rightarrow \C_{\dg}$ induces a quasi-equivalence 
$
\tau_{\leq 0}\F_{\dg}^{\rel}\rightarrow \tau_{\leq 0} \mathcal H_{\dg}
$.
\end{corollary}

\begin{proposition}
Via the canonical quasi-equivalence $\tau_{\leq 0}\F^{\rel}_{\dg}\rightarrow \tau_{\leq 0} \mathcal H_{\dg}$, the exact dg structure on $\tau_{\leq 0}\F^{\rel}_{\dg}$ is identified with the exact substructure on $\tau_{\leq 0}\mathcal H_{\dg}$ which makes the objects in $\add(\Gamma_{\rel})$ projective.
Moreover, the bounded dg derived category of $\tau_{\leq 0}\mathcal H_{\dg}$ is $\C_{\dg}$ and
the bounded dg derived category of $\tau_{\leq 0}\F_{\dg}^{\rel}$ is $\per_{\dg}(\Gamma_{\rel})$.
\end{proposition}
\begin{proof}
Let $X$ and $Y$ be two objects in $\F^{\rel}$ and $g:\pi^{\rel}X\rightarrow \Sigma \pi^{\rel}Y$ a morphism in $\mathcal C$. 
Then by Lemma \ref{equi}, $g$ is of the form $\pi^{\rel}(f)$ for some $f:X\rightarrow \Sigma Y$ if and only if $g$ factors though an object in $\add (\pi\Sigma \Gamma_{\rel})$. 
This is exactly the exact substructure which makes the objects in $\add(\Gamma_{\rel})$ projective.

By \cite[Section 5.8]{Wu23a}, Theorem~\ref{main} and Proposition~\ref{higher}, the canonical exact morphism 
$
\tau_{\leq 0}\mathcal H_{\dg}\rightarrow \C_{\dg}
$
characterises $\C_{\dg}$ as the bounded dg derived category of $\tau_{\leq 0}\H_{\dg}$.

Since $\Gamma_{\rel}$ is projective in the extriangulated category $\F^{\rel}$, higher extensions $\mathbb E^i$ of $\F^{\rel}$ vanish for $i\geq 2$.
Also, we have $\Ext^{i}_{\per(\Gamma_{\rel})}(X,Y)=0$ for $X,Y\in\F^{\rel}$ and $i\geq 2$.
Thus by  Theorem \ref{main} and Proposition \ref{higher}, the canonical exact morphism 
$
\tau_{\leq 0}\mathcal F^{\rel}_{\dg}\rightarrow \per_{\dg}(\Gamma_{\rel})
$
characterises $\per_{\dg}(\Gamma_{\rel})$ as the bounded dg derived category of $\tau_{\leq 0}\mathcal F^{\rel}_{\dg}$.

\end{proof}

\section{Reproduction of exact dg categories}\label{sec:reproduction}
In this section, we show that ``functor dg categories" $\rep_{\dg}(\B,\A)$ with exact target $\A$ and connective source $\B$ carry canonical exact structures, cf.~Theorem~\ref{fun}. 
As an example, we show that when $\B$ is a certain path $k$-category with radical square zero relations, and $\A$ carries the trivial exact structure, we obtain a Frobenius exact structure, cf.~Proposition~\ref{prop:2-periodic}.
We define {\em biexact} morphisms from the tensor product of two exact dg categories to another exact dg category and show the existence of a universal biexact morphism. 
We show that in this way we obtain a closed monoidal category with internal hom whose value on a pair of small connective exact dg categories $(\B,\A)$ given by extension-closed dg subcategory of $\rep_{\dg}(\B,\A)$ consisting of exact morphisms. 

In this section, it is convenient to work with $\rep(k\Sq,\A)$ instead of $\mathcal H_{3t}(\A)$, to be consistent with $\rep_{\dg}(\B,\A)$. 
Recall that by \cite[3.2]{Chen24}, we have a fully faithful functor $\mathcal H_{3t}(\A)\hookrightarrow \rep(k\Sq,\A)$ whose essential image consists of the squares in $\C(\A)$
\begin{equation}\label{squa:X}
\begin{tikzcd}
X\ar[r]\ar[d]&Y\ar[d]\\
N\ar[r]&Z
\end{tikzcd}
\end{equation}
where $N$ is acyclic.
In particular, for such squares (\ref{squa:X}), we may assume that $X$, $Y$, $Z$ are representable dg $\A$-modules and  $N=\Cone(\Id_{X})$.
An exact dg structure on an additive dg category $\A$ can then be defined to be a class $\mathcal S\subset \rep(k\Sq,\A)$ of homotopy bicartesian squares (\ref{squa:X}), where $N$ is acyclic, which satisfies similar axioms as those of Definition~\ref{exactdgstructure}.
Exceptionally, we will employ this definition in this section.

\subsection{Functor dg categories}\label{subsec:functordg}
A famous slogan by Amnon Neeman is: triangulated categories do not reproduce. 
A key feature of exact dg categories is that they do reproduce, as the following theorem shows.
\begin{theorem}\label{fun}
Let $\A$ be a small exact dg category. 
If $\B$ is a small connective dg category, then $\rep_{\dg}(\B,\A)$ is canonically an exact dg category. 
\end{theorem}
We will prove Theorem~\ref{fun} after Lemma~\ref{pointwise}.
For a dg category $\A$, we denote by $\tilde{\A}$ the full dg subcategory of $\C_{\dg}(\A)$ consisting of cofibrant quasi-representable dg $\A$-modules. 
 The Yoneda dg functor $Y:\A\rightarrow \C_{\dg}(\A)$ induces a quasi-equivalence from $\A$ to $\tilde{\A}$, which we still denote by $Y$.
 We may assume that $\B$ is cofibrant as a dg $k$-category. 
 For each $B\in \B$, let $R_B:\rep_{\dg}(\B,\A) \rightarrow \tilde{\A}$ be the dg functor given by $X\mapsto X(-,B)$. 
 We have a morphism in $\mathrm{Hqe}$ given by the following roof of dg functors 
 \[
 \begin{tikzcd}
 &\tilde{\A}&\\
 \rep_{\dg}(\B,\A)\ar[ru,"R_B"]&&\A\ar[lu,"Y"swap]
 \end{tikzcd}.
 \]
 It then induces a functor $\rep(k\mathcal I, \rep_{\dg}(\B,\A))\rightarrow \rep(k\mathcal{I},\A)$ for each small category $\mathcal I$, which we denote by $F^{\mathcal I}_B$ or $F_B$ if there is no ambiguity.
 \begin{lemma}\label{suff}
 Let $X$ be an object in $\rep(k\mathrm{Sq},\rep_{\dg}(\B,\A))$. 
If for each $B\in\B$, $F_B(X)\in \rep(k\mathrm{Sq},\A)$ is a homotopy cocartesian (resp.,~cartesian) square, 
then $X$ is also a homotopy cocartesian (resp.,~cartesian) square. 
 \end{lemma}
 \begin{proof}
We show the case when $F_B(X)$ is homotopy cocartesian for each $B\in\B$. The other case can be shown similarly. 
We may suppose that $X$ is of the form
 \[
 \begin{tikzcd}
 X_{0}^{\wedge}\ar[r,"f^{\wedge}"]\ar[d]&X_{1}^{\wedge}\ar[d,"j^{\wedge}"]\\
 IX_0^{\wedge}\ar[r]&X_{2}^{\wedge}
 \end{tikzcd}
 \]
 where $X_i\in \rep_{\dg}(\B,\A)$ for $i=0,1,2$, and $IX_0^{\wedge}=\Cone(\Id_{X_0^{\wedge}})$.
 For a morphism $g$ in $\C(\B^{op}\otimes \A)$, we denote by $C(g)$ its cone.
 Consider the following diagram in $\C(\B^{op}\otimes \A)$
 \[
 \begin{tikzcd}
 X_0\ar[r,"f"]\ar[d,"s=\begin{bmatrix}-f\\0\end{bmatrix}"swap]&X_1\ar[r,"\begin{bmatrix}0\\1\end{bmatrix}"]\ar[d,equal]&C(f)\ar[d,"r={[}0{,}j{]}"]\\
 \Sigma^{-1}C(j)\ar[r,"{[}-1{,}0{]}"swap]&X_1\ar[r,"j"swap]&X_2
 \end{tikzcd}.
 \]
 Let $W$ be an object in $\rep_{\dg}(\B,\A)$.
 Since $F_{B}(X)$ is homotopy cocartesian for $B\in \B$, for each pair of objects $B_1,B_2$ in $\B$, and $i\leq 0$, we have
 \[
 \Hom_{\D(\A)}(X_2(-,B_1),\Sigma^{i}W(-,B_2))\xrightarrow{\sim}\Hom_{\D(\A)}(C(f)(-,B_1),\Sigma^{i}W(-,B_2)),
 \]
 and hence for $j\leq 0$ we have
 \[
 \tau_{\leq 0}\RHom_{\A}(X_2,\Sigma^{j}W)\xrightarrow{\sim}\tau_{\leq 0}\RHom_{\A}(C(f),\Sigma^{j}W).
 \]
 Thus we have 
\[
\begin{aligned}
\Hom_{\D(\rep_{\dg}(\B,\A))}(X_2^{\wedge}, \Sigma^{j}W^{\wedge})&\xrightarrow{\sim}\Hom_{\D(\B^{op}\otimes \A)}(X_2, \Sigma^{j}W)\\
&\xrightarrow{\sim}\Hom_{\D(\B^{e})}(\B,\RHom_{\A}(X_2,\Sigma^jW))\\
&\xleftarrow{\sim}\Hom_{\D(\B^{e})}(\B,\tau_{\leq 0}\RHom_{\A}(X_2,\Sigma^{j}W))\\
&\xrightarrow{\sim}\Hom_{\D(\B^{e})}(\B,\tau_{\leq 0}\RHom_{\A}(C(f),\Sigma^{j}W))\\
&\xrightarrow{\sim}\Hom_{\D(\B^e)}(\B,\RHom_{\A}(C(f),\Sigma^{j}W))\\
&\xrightarrow{\sim}\Hom_{\D(\B^{op}\otimes\A)}(C(f),\Sigma^{j}W)\\
&\xrightarrow{\sim}\Hom_{\D(\rep_{\dg}(\B,\A))}(C(f^{\wedge}),\Sigma^{j}W^{\wedge}).
\end{aligned}
\]
This shows that $X$ is a homotopy cocartesian square.
 \end{proof}
 \begin{lemma}\label{pointwise}
 Let $S$ be an object in $\rep(k\mathrm{Cosp}, \rep_{\dg}(\B,\A))$ (resp.,~$\rep(k\mathrm{Sp}, \rep_{\dg}(\B,\A))$).
If for each $B\in \B$, the object $F_B(S)$ in $\rep(k\mathrm{Cosp}, \A)$ (resp.,~$\rep(k\mathrm{Sp}, \A)$) admits a homotopy pullback (resp.,~homotopy pushout). 
Then $S$ admits a homotopy pullback (resp.,~homotopy pushout).
 \end{lemma}
 \begin{proof}
 We show the case when $S\in\rep(k\mathrm{Cosp}, \rep_{\dg}(\B,\A))$. 
 The other case can be shown dually. 
 We may assume $S$ is of the form
 \[
 \begin{tikzcd}
 &X_1^{\wedge}\ar[d,"j^{\wedge}"]\\
 PX_2^{\wedge}\ar[r]&X_2^{\wedge}
 \end{tikzcd}
 \]
 where $PX_2^{\wedge}=\Cone(\Id_{\Sigma^{-1}X_2^{\wedge}})$, and $X_i\in \rep_{\dg}(\B,\A)$, $i=1$, $2$, and $j$ is a graded-split surjection. 
  Since $j$ and $PX_2\rightarrow X_2$ are graded-split surjections, by assumption, 
 for each $B\in\B$, $F_B(S)$ admits a homotopy pullback
 \[
 \begin{tikzcd}
 H(B)\ar[r,""]\ar[d]&X_1(-,B)\ar[d,"j_{(-,B)}"]\\
 PX_2(-,B)\ar[r]&X_2(-,B)
 \end{tikzcd}
 \]
 where $H(B)\in H^0(\A)$.
 Let $V$ be the kernel of the map $X_1\oplus PX_2\rightarrow X_2$ in $\C(\B^{op}\otimes \A)$. 
 Then for each $B\in\B$, we have a canonical map $s_B: H(B)\rightarrow V(-,B)$.
Put $G(B)=\Sigma\Cone(s_B)$. By definition $G(B)$ lies in the heart of the canonical t-structure on $\D(\A)$, 
which is equivalent to $\Mod(H^0(\A))$. 
 For each morphism $b:B\rightarrow B'$ in $H^0(\B)$, we have a canonical morphism $H(b):H(B)\rightarrow H(B')$ in $H^0(\B)$ which makes the following diagram in $\D(\A)$ commute 
 \[
 \begin{tikzcd}
 H(B)\ar[r]\ar[d]&H(B')\ar[d]\\
 V(-,B)\ar[r]&V(-,B').
 \end{tikzcd}
 \]
Since $\Hom_{\D(\A)}(\Sigma H(B),\Sigma^{-1}G(B'))=0$, it induces a unique morphism $G(b):G(B)\rightarrow G(B')$ in $\D(\A)$ rendering the following commutative diagram in $\D(\A)$
\[
\begin{tikzcd}
H(B)\ar[d,"H(b)"swap]\ar[r]&V(-,B)\ar[r]\ar[d,"V(-{,}b)"]&\Sigma^{-1}G(B)\ar[d,dashed,"\Sigma^{-1}G(b)"]\ar[r]&\Sigma H(B)\ar[d,"\Sigma H(b)"]\\
H(B')\ar[r]&V(-,B')\ar[r]&\Sigma^{-1}G(B')\ar[r]&\Sigma H(B').
\end{tikzcd}
\]
Therefore we have a canonical functor $G: H^0(\B)\rightarrow \Mod(H^0(\A))$ which sends $B$ to $G(B)$ and $b:B\rightarrow B'$ to $G(b):G(B)\rightarrow G(B')$.
By adjunction, the functor $G$ corresponds to an object in $\Mod(H^0(\B)^{op}\otimes H^0(\A))$, which is equivalent to the heart of the canonical t-structure on $\D(\B^{op}\otimes \A)$.
We still denote by $G$ the corresponding object in $\D(\B^{op}\otimes \A)$. We may assume $G$ is cofibrant.

 Next, we will define a morphism from $\Sigma V$ to $G$ in $\D(\B^{op}\otimes \A)$.
We have the canonical isomorphism
\begin{align}
\Hom_{\D(\B^{op}\otimes\A)}(\Sigma V, G)&\xrightarrow{\sim}\Hom_{\D(\B^{e})}(\B, \RHom_{\A}(\Sigma V, G)).\label{bimodule1}\tag{41.1}
\end{align}
We have that $\RHom_{\A}(\Sigma V, G)$ is concentrated in degrees 0 and 1.
So we have
\begin{align}
\Hom_{\D(\B^{e})}(\B, \RHom_{\A}(\Sigma V, G))\xrightarrow{\sim} \Hom_{(H^0\B)^e}(H^0(\B), H^0\RHom_{\A}(\Sigma V, G)). \label{bimodule2}\tag{41.2}
\end{align}
For each pair of objects $B,B'\in H^0(\B)$, we have the map
\[
\alpha_{B,B'}:\Hom_{H^0(\B)}(B,B')\rightarrow \Hom_{\D(\A)}(\Sigma V(-,B), G(-,B'))
\]
where $b:B\rightarrow B'$ is sent to $\Sigma V(-,B)\rightarrow G(-,B)\xrightarrow{G(-,b)} G(-,B')$ where the first map is the canonical map.
It is not hard to check 
\[
\alpha_{-,-}: \Hom_{H^0(\B)}(-,-)\rightarrow \Hom_{\D(\A)}(\Sigma V, G)
\]
defines an $H^0(\B)$-$H^0(\B)$-bimodule morphism.  
By the canonical isomorphisms (\ref{bimodule1}) and (\ref{bimodule2}), the morphism $\alpha_{-,-}$ defines a morphism $u:\Sigma V\rightarrow G$ in $\D(\B^{op}\otimes \A)$. By definition, $u_{(-,B)}$ is the canonical morphism $\Sigma V(-,B)\rightarrow G(-,B)$.

Put $X_0=\Sigma^{-1}\Cone(u)$. 
Then $X_0(-,B)$ is quasi-isomorphic to $H(B)$ for each $B\in\B$ and hence $X_0\in \rep_{\dg}(\B, \A)$. 
Consider the composition morphism $X_0\rightarrow V\rightarrow X_1\oplus PX_2$.
The sequence $X_0\rightarrow X_1\oplus PX_2\rightarrow X_2$ then corresponds to a square
\[
\begin{tikzcd}
X_0^{\wedge}\ar[r]\ar[d]&X_1^{\wedge}\ar[d]\\
PX_2^{\wedge}\ar[r]&X_2^{\wedge}
\end{tikzcd}
\]
which is homotopy cartesian and hence a homotopy pullback of $S$.
\end{proof}
 \begin{proof}[Proof of Theorem \ref{fun}]
 Let $\tilde{\S}$ be the class of homotopy bicartesian squares $X$ in $ \rep_{\dg}(\B,\A)$ such that $F_B(X)\in\rep(k\Sq,\A)$ is a conflation for each $B\in\B$. 
 By Lemma~\ref{suff} and Lemma~\ref{pointwise}, one checks that $(\rep_{\dg}(\B,\A),\tilde{\S})$ is an exact dg category.
\end{proof}

In Theorem~\ref{fun}, when $\A$ and $\B$ are concentrated in degree zero so that $\A$ is a Quillen exact category by Example~\ref{exm:exactdg} 2), we recover the pointwise exact structure on the category $\Fun_{k}(\B,\A)$ of $k$-linear functors.
\begin{lemma}
Let $\B$ be a $k$-category 
and $\A$ a Quillen exact category considered as an exact dg category. 
Then $\tau_{\leq 0}\rep_{\dg}(\B,\A)$ is concentrated in degree zero.
Moreover, we have a canonical equivalence of exact categories $\Fun_{k}(\B,\A)\iso \rep(\B,\A)$, where $\Fun_{k}(\B,\A)$ is endowed with the pointwise exact structure, and $\rep(\B,\A)$ is endowed with the exact structure from Theorem~\ref{fun}.
\end{lemma}
\begin{proof}
For objects $X$ and $Y$ in $\rep(\B,\A)$, we have a canonical isomorphism
\[
\Hom_{\D(\A\otimes\B^{op})}(X,\Sigma^i Y)\iso \Hom_{\D(\B^{e})}(\B,\RHom_{\A}(X,\Sigma^i Y)).
\]
Note that the space $\tau_{\leq 0}\RHom_{\A}(X,\Sigma^i Y)$ is zero if $i<0$, since the dg category $\A$ is concentrated in degree $0$.
Thus we have a quasi-equivalence of dg categories
\[
\tau_{\leq 0}\rep_{\dg}(\B,\A)\iso H^0(\rep_{\dg}(\B,\A))\iso \rep(\B,\A).
\]
Therefore, the canonical exact dg structure on $\rep_{\dg}(\B,\A)$ provided by Theorem \ref{fun} induces a Quillen exact structure on $\rep(\B,\A)$. 
 
 The heart of the canonical t-structure on $\D(\A\otimes \B^{op})$ is equivalent to $\Mod(\A\otimes \B^{op})$.
 Since the objects $X$ and $Y$ both have cohomology concentrated in degree 0, they belong to the heart.
 So they correspond to $k$-linear functors $\A^{op}\otimes \B\rightarrow \Mod k$. 
 Furthermore, they are given by functors $F,G:\B\rightarrow \A$ because $X(-,B)$ and $Y(-,B)$ are quasi-representable for each $B\in\B$.
 So we have an inclusion functor
 \begin{align*}
  \Fun_k(\B,\A)&\hookrightarrow \Fun_{k}(\B,\Mod \A)\iso\Fun_{k}(\A^{op}\otimes \B,\Mod k)\\
 &\iso \Mod(\A\otimes \B^{op})\hookrightarrow \D(\A\otimes \B^{op})
 \end{align*}
whose essential image is exactly $\rep(\B,\A)$.
If we identify the categories $\Fun_{k}(\B,\A)$ and $\rep(\B,\A)$ via the above functor, we see that the Quillen exact structure on $\rep(\B,\A)$ is identified with the componentwise Quillen exact structure on $\Fun_{k}(\B,\A)$.
\end{proof}

\begin{remark}\label{rmk:funexact}
From the definition of the canonical exact dg structure on $\rep_{\dg}(\B,\A)$, it is clear that an exact morphism $F:\A\rightarrow \A'$ in $\Hqe$ induces an exact morphism $F_*:\rep_{\dg}(\B,\A)\rightarrow\rep_{\dg}(\B,\A')$ in $\Hqe$.
\end{remark}
\begin{remark}
Let $\B$ be a connective dg category. 
Let $F:\A\rightarrow \D^b_{\dg}(\A,\S)$ be the universal exact morphism from $\A$ to a pretriangulated dg category. 
By Remark~\ref{rmk:funexact}, the canonical morphism 
$
F_*:\rep_{\dg}(\B,\A)\rightarrow\rep_{\dg}(\B,\D^b_{\dg}(\A,\S))
$
is exact and hence induces a canonical morphism 
$
\phi: \D^b_{\dg}(\rep_{\dg}(\B,\A),\tilde{\S})\rightarrow \rep_{\dg}(\B,\D^b_{\dg}(\A,\S))
$
in $\Hqe$.
For objects $X$ and $Y$ in $\rep_{\dg}(\B,\A)$, the canonical map
$
\tau_{\leq 0}\RHom(X,Y)\iso \tau_{\leq 0}\RHom(F_*X,F_*Y).
$
is an isomorphism in $\D(k)$.
Also, by Corollary~\ref{strictify}, the canonical map
\begin{equation}\label{bij:repderived}
\mathbb E(X, Y)\iso \Hom_{\rep(\B, \D^b_{\dg}(\A))}(F_*X, \Sigma F_* Y)
\end{equation}
is a bijection. 
We do not know whether the same is true for higher extension groups (and hence do not know whether $H^0(\phi)$ is a triangle equivalence).
\end{remark}
We end this subsection with an example which generalises the category of $2$-periodic complexes over an additive category.
Let $\B$ be the path $k$-category of the following quiver 
\[
\begin{tikzcd}
0\ar[r,shift left=1ex,"u"]&1\ar[l,shift left=1ex,"v"]
\end{tikzcd}
\]
with relations $uv=0$ and $vu=0$.
Let $\A$ be a connective additive dg category with the trivial exact structure $\S$, cf.~Example~\ref{exm:exactdg} 5). 
\begin{proposition}\label{prop:2-periodic}
The dg category $\rep_{\dg}(\B,\A)$ with the exact structure $\tilde{\S}$ given by Theorem~\ref{fun} is Frobenius.
\end{proposition}
\begin{proof}
We first show that objects of the forms $\begin{tikzcd}X\ar[r,shift left=1ex, "\Id_{X}"]&X\ar[l,shift left=1ex,"0"]\end{tikzcd}$, and $\begin{tikzcd}X\ar[r,shift left=1ex, "0"]&X\ar[l,shift left=1ex,"\Id_{X}"]\end{tikzcd}$, where $X\in H^0(\A)$, are projective-injective in $\rep_{\dg}(\B,\A)$.
We denote by $I$ the dg functor $k\rightarrow \B$ sending the unique object of $k$ to the object $0$ in $\B$.
We obtain the following adjoint triple:
\[
\begin{tikzcd}
\D(\A)\ar[rr, shift left=2ex,"U"{description}]\ar[rr, shift right= 2ex,"R"{description}]&&\D(\A\otimes \B^{op})\ar[ll,"L"{description}]
\end{tikzcd}
\]
induced by $I$.
The functor $L$ sends an object $\begin{tikzcd}X\ar[r,shift left=1ex,"f"] &Y\ar[l,shift left=1ex, "g"]\end{tikzcd}$ in $\D(\A\otimes \B^{op})$ to the object $X$.
The functor $R$ sends an object $X$ to the object $\begin{tikzcd}X\ar[r,shift left=1ex, "\Id_{X}"]&X\ar[l,shift left=1ex,"0"]\end{tikzcd}$.
So we have 
\[
\Hom_{\D(\A\otimes \B^{op})}(\begin{tikzcd}X\ar[r,shift left=1ex,"f"] &Y\ar[l,shift left=1ex, "g"]\end{tikzcd}, R(Z))\iso \Hom_{\D(\A)}(X,Z).
\]
Therefore, objects of the form $\begin{tikzcd}X\ar[r,shift left=1ex, "\Id_{X}"]&X\ar[l,shift left=1ex,"0"]\end{tikzcd}$ are injective in $(\rep_{\dg}(\B,\A),\tilde{\S})$.
We denote by $J$ the dg functor $B\rightarrow k$ sending $0$ and $1$ to the unique object of $k$, and $u$ to $\Id$ and $v$ to $0$.
Then we obtain the following adjoint triple
\[
\begin{tikzcd}
\D(\A\otimes \B^{op})\ar[rr, shift left=2ex,"U'"{description}]\ar[rr, shift right= 2ex,"R'"{description}]&&\D(\A)\ar[ll,"L'"{description}]
\end{tikzcd}
\]
induced by $J$.
The functor $L'$ sends an object $X\in \D(\A)$ to the object $\begin{tikzcd}X\ar[r,shift left=1ex, "\Id_{X}"]&X\ar[l,shift left=1ex,"0"]\end{tikzcd}$ in $\D(\A\otimes \B^{op})$.
The functor $R'$ sends an object $\begin{tikzcd}X\ar[r,shift left=1ex,"f"] &Y\ar[l,shift left=1ex, "g"]\end{tikzcd}$ in $\D(\A\otimes \B^{op})$ to $X$.
So we have 
\[
\Hom_{\D(\A\otimes \B^{op})}(L'(Z), \begin{tikzcd}X\ar[r,shift left=1ex,"f"] &Y\ar[l,shift left=1ex, "g"]\end{tikzcd})\iso \Hom_{\D(\A)}(Z,X).
\]
Therefore, objects of the form $\begin{tikzcd}X\ar[r,shift left=1ex, "\Id_{X}"]&X\ar[l,shift left=1ex,"0"]\end{tikzcd}$ are projective in $(\rep_{\dg}(\B,\A),\tilde{\S})$.
Dually, objects of the form $\begin{tikzcd}X\ar[r,shift left=1ex, "0"]&X\ar[l,shift left=1ex,"\Id_{X}"]\end{tikzcd}$ are also projective-injective in $(\rep_{\dg}(\B,\A),\tilde{\S})$.
For each object $\begin{tikzcd}X\ar[r,shift left=1ex,"f"] &Y\ar[l,shift left=1ex, "g"]\end{tikzcd}$ in $\rep_{\dg}(\B,\A)$, we have a conflation
\[
\begin{tikzcd}
X\ar[rr,"{[}1{,}f{]^{\intercal}}"{description}]\ar[d,"f"swap]&& X\oplus Y\ar[d,"\begin{bmatrix}0\ 0\\0\ 1\end{bmatrix}"]\ar[rr,"{[}-f{,}1{]}"{description}]&&Y\ar[d,"g"]\\
Y\ar[rr,"{[}g{,}1{]^{\intercal}}"{description}]\ar[d,"g"swap]&& X\oplus Y\ar[d,"\begin{bmatrix}1\ 0\\0\ 0\end{bmatrix}"]\ar[rr,"{[}-1{,}g{]}"{description}]&&X\ar[d,"f"]\\
X\ar[rr,"{[}1{,}f{]^{\intercal}}"{swap,description}]&& X\oplus Y\ar[rr,"{[}-f{,}1{]}"{swap,description}]&&Y
\end{tikzcd}
\]
Hence the exact dg category $(\rep_{\dg}(\B,\A),\tilde{\S})$ has enough injectives. Dually, it also has enough projectives and therefore it is a Frobenius exact dg category.
\end{proof}
\subsection{Closed monoidal structure on $\Hqe^{\cn}_{\ex}$}
Let $(\A,\mathcal S)$ and $(\A',\mathcal S')$ be small exact dg categories. 
Recall that a morphism $\A\rightarrow \A'$ in $\Hqe$ is {exact} if the induced functor $\mathcal H_{3t}(\A)\rightarrow \mathcal H_{3t}(\A')$ sends objects in $\mathcal S$ to objects in $\mathcal S'$.
Recall that we denote by $\Hqe_{\ex}((\A,\S),(\A',\S'))$ the subset of $\Hqe(\A,\A')$ consisiting of exact morphisms. 
Let $(\C,\S'')$ be an exact dg category. 
A morphism $\mu: \A\otimes \A'\rightarrow \C$ in $\Hqe$ is {\em biexact}, if for each object $A\in \A$ and $A'\in \A'$, the induced morphisms 
\[
\mu_{A,-}:\A'\rightarrow \C,\;\;\mu_{-,A'}:\A\rightarrow \C
\]
are both exact morphisms.
Suppose $\A$ and $\A'$ are both connective. 
By Theorem~\ref{main}, we have universal fully exact embeddings $\A\rightarrow \D^b_{\dg}(\A,\S)$ and $\A'\rightarrow \D^b_{\dg}(\A',\S')$ into pretriangulated dg categories.
We denote by $\rep_{\dg}^{\ex}(\A,\A')$ the full dg subcategory of $\rep_{\dg}(\A,\A')$ consisting of the exact morphisms.
\begin{lemma}\label{lem:internalhom}
Suppose $\A$ and $\A'$ are small connective exact dg categories.
The dg category $\rep_{\dg}^{\ex}(\A,\A')$ is stable under extensions in $\rep_{\dg}(\A,\A')$. 
In particular, it inherits a canonical exact dg structure.
\end{lemma}
\begin{proof}
Suppose we have a conflation in $\rep_{\dg}(\A,\A')$
\[
F\rightarrow G\rightarrow H
\]
where we omit the homotopy, and where $F$ and $H$ are exact morphisms.
We would like to show that $G$ is also exact.
Assume that we are given a conflation in $\A$
\[
A\rightarrow B\rightarrow C
\]
where again we omit the homotopy. 
Then we have the following diagram in $\C(\A')$
\begin{equation}\label{dia:conflations}
\begin{tikzcd}
F(-,A)\ar[r,tail]\ar[d,tail]&F(-,B)\ar[d,tail]\ar[r,two heads]&F(-,C)\ar[d,tail]\\
G(-,A)\ar[r,blue]\ar[d,two heads]&G(-,B)\ar[d,two heads]\ar[r,blue]&G(-,C)\ar[d, two heads]\\
H(-,A)\ar[r,tail]&H(-,B)\ar[r,two heads]&H(-,C)
\end{tikzcd}
\end{equation}
where we have omitted the homotopies.
The totalizations of the first and the third rows are in $\N$ associated with $(\A',\S')$, cf.~Subsection~\ref{subsec:N}. 
 The totalization of the second row, being an extension of the totalizations of the first 
 and the third rows, is thus also in $\N$.
 Therefore the second row is  homotopy bicartesian in $\A'$.
  Consider the image of the diagram (\ref{dia:conflations}) under the universal exact morphism $F:\A'\rightarrow \D^b_{\dg}(\A',\S')$.
  By Theorem~\ref{main}, we see that the image of the second row under $F$ is isomorphic to the image of a conflation under $F$.
  Therefore, the second row in (\ref{dia:conflations}) is a conflation.
 It follows that $G$ is exact.
\end{proof}

Let $\A\boxtimes \A'$ be the $\tau_{\leq 0 }$-truncation of the extension closure of the quasi-essential image of the morphism in $\Hqe$
\[
F: \A\otimes \A'\rightarrow \D^b_{\dg}(\A)\otimes\D^b_{\dg}(\A')\rightarrow \pretr(\D^b_{\dg}(\A)\otimes\D^b_{\dg}(\A'))
\]
with the inherited exact structure.
\begin{proposition}\label{prop:universalbilinear}
 The natural morphism $\A\otimes \A'\rightarrow \A\boxtimes \A'$ is the universal biexact morphism in $\Hqe$ from $\A\otimes \A'$ to an exact dg category.
\end{proposition}
\begin{proof}
Let $(\C,\S'')$ be an exact dg category and $G:\A\otimes \A'\rightarrow \C$ a biexact morphism in $\Hqe$. 
Since $\A\otimes \A'$ is connective, the morphism $G$ factors through $\tau_{\leq 0}\C$.
We still denote the induced exact structure on $\tau_{\leq 0}\C$ by $\S''$.
By adunction, we obtain a morphism $\A\rightarrow \rep_{\dg}^{\ex}(\A',\D^b_{\dg}(\tau_{\leq 0}\C,\S''))$.
We endow $ \rep_{\dg}^{\ex}(\A',\D^b_{\dg}(\tau_{\leq 0}\C,\S''))$ with the exact structure given by Lemma~\ref{lem:internalhom}.
Then the above morphism is exact and hence we obtain an induced morphism $\D^b_{\dg}(\A,\S)\rightarrow \rep_{\dg}^{\ex}(\A',\D^b_{\dg}(\tau_{\leq 0}\C,\S''))$.
Again by adjunction we obtain an exact morphism $\A'\rightarrow \rep_{\dg}(\D^b_{\dg}(\A,\S),\D^b_{\dg}(\tau_{\leq 0}\C,\S''))$ which yields a canonical morphism $\D^b_{\dg}(\A',\S')\rightarrow \rep_{\dg}(\D^b_{\dg}(\A,\S),\D^b_{\dg}(\tau_{\leq 0}\C,\S''))$.
So we have the following diagram
\[
\begin{tikzcd}
&\A\otimes \A'\ar[r,"F"]\ar[d]\ar[dd,bend right=10ex]&\pretr(\D^b_{\dg}(\A,\S)\otimes \D^b_{\dg}(\A',\S'))\ar[dd,dashed,bend left=10ex]\\
&\Im(F)\ar[r]\ar[ru,hook]&\A\boxtimes \A'\ar[u,hook]\ar[ld,dashed,red]\\
\C&\tau_{\leq 0}\C\ar[l,hook]\ar[r,hook]&\D^b_{\dg}(\tau_{\leq 0}\C,\S'')
\end{tikzcd}.
\] 
Since the essential image of the inclusion $\tau_{\leq 0}{\C}\rightarrow \D^b_{\dg}(\tau_{\leq 0}\C)$ is extension-closed, the conclusion follows immediately.
\end{proof}

\begin{corollary}\label{cor:exactadjunction}
Let $\A$ and $\B$ be small connective exact dg categories. 
Let $\C$ be a small exact dg category. 
Then we have the following natural bijection
\[
\Hqe_{\ex}(\A\boxtimes \B,\C)\iso \Hqe_{\ex}(\A,\rep_{\dg}^{\ex}(\B,\C)).
\]
\end{corollary}
\begin{proof}
The required bijection follows from the following composition of bijections
\[
\Hqe_{\ex}(\A\boxtimes\B,\C)\iso \Hqe_{\mathrm{biexact}}(\A\otimes \B,\C)\iso\Hqe_{\ex}(\A,\rep_{\dg}^{\ex}(\B,\C))
\]
where $\Hqe_{\mathrm{exact}}(\A\otimes\B,\C )$ denotes the set of biexact morphisms.
\end{proof}
Let $\Hqe^{\cn}_{\ex}$ be the full subcategory of $\Hqe_{\ex}$ consisting of connective exact dg categories. 
Let $\I$ be the category of free $k$-modules of finite rank which we consider as a dg category concentrated in degree zero.
We endow $\I$ with the trivial exact structure.
It is clear that the tensor product $\A\boxtimes \A'$ extends to a bifunctor $-\boxtimes - :\Hqe_{\ex}^{\cn}\times \Hqe_{\ex}^{\cn}\rightarrow \Hqe_{\ex}^{\cn}$. 
\begin{proposition}
The bifunctor $-\boxtimes -$ defines a symmetric monoidal structure on $\Hqe^{\cn}_{\ex}$ with $\I$ being the tensor unit. 
\end{proposition}
\begin{proof}
Let $\A$, $\B$ and $\C$ be small connective exact dg categories. 
The tensor products $\A\boxtimes (\B\boxtimes \C)$ and $(\A\boxtimes \B)\boxtimes \C$ are both solutions of universal {\em triexact} morphisms from $\A\otimes \B\otimes \C$ to an exact dg category.
Therefore they are naturally isomorphic in $\Hqe_{\ex}^{\cn}$.
It is straightforward to check that this defines a monoidal structure on $\Hqe_{\ex}^{\cn}$.

We have $\A\otimes \B\iso \B\otimes \A$ in $\Hqe$ which induces a canonical isomorphism $\A\boxtimes \B\iso \B\boxtimes \A$
in $\Hqe_{\ex}^{\cn}$. 
Therefore the monoidal structure is symmetric.

We may assume that the small connective exact dg category $\A$ is such that $Z^0(\A)$ is additive. 
We define a  morphism in $\Hqe$
\[
G: \I\otimes \A\rightarrow \A,\;\; (\bigoplus_{\text{$J$ finite}} k, A)\mapsto \bigoplus_{\text{$J$ finite}} A.
\]
By \cite[Proposition 4.9 c)]{Chen24}, it is biexact.
Let $F:\I\otimes \A\rightarrow \E$ be the universal biexact morphism.
Then we have $F(\bigoplus_{\text{$J$ finite}} k, A)\iso F(k,\bigoplus_{\text{$J$ finite}} A)$.
Hence, the morphism $G$ is universal biexact. Therefore,  $\I$ is the tensor unit.
\end{proof}
By Corollary~\ref{cor:exactadjunction}, this monoidal structure is {\em closed} with internal Hom $\tau_{\leq 0}\rep_{\dg}^{\ex}(?,-)$.
\begin{corollary}\label{cor:closedmonoidal}
Let $\A$, $\B$ and $\C$ be small connective exact dg categories. 
We have the following quasi-equivalence of exact dg categories
\[
\tau_{\leq 0}\rep^{\ex}_{\dg}(\A\boxtimes\B,\C)\iso\tau_{\leq 0}\rep^{\ex}_{\dg}(\A,\tau_{\leq 0}\rep_{\dg}^{\ex}(\B,\C)).
\]
\end{corollary}

\section{Application to $0$-Auslander extriangulated categories}\label{sec:applications}
An extriangulated category $(\C,\mathbb E,\mathfrak s)$ is {\em $0$-Auslander} \cite{GorskyNakaokaPalu23} if 
\begin{itemize}
\item it has enough projectives;
\item it has positive global dimension at most 1;
\item it has dominant dimension at least 1.
\end{itemize}
It is {\em reduced} if moreover the only projective-injectives are the zeros.
An exact dg category $\A$ is {\em $0$-Auslander} if $H^0(\A)$ is $0$-Auslander.
For a pretriangulated category $\T$ with full dg subcategories $\C$ and $\D$, we write $\C\ast\D$ the full dg subcategory of $\T$ consisting of objects in the subcategory 
\begin{align*}
H^0(\C)\ast H^0(\D)=\{X\in H^0(\T)&\mid \text{$\exists$ triangle $C\rightarrow X\rightarrow D\rightarrow \Sigma C$ in $H^0(\T)$, where} \\
&\text{ $C\in H^0(\C)$, $D\in H^0(\D)$}\}
\end{align*}
 of the triangulated category $H^0(\T)$.
Inspired by the talk by Matthew Pressland on the recent work of
Fang--Gorsky--Palu--Plamondon--Pressland \cite{FangGorskyPaluPlamondonPressland23}, we have the following $0$-Auslander correspondence.
\begin{theorem}[0-Auslander correspondence]\label{thm:0-Auslander correspondence}
There is a bijective correspondence between the following:
\begin{itemize}
\item[(1)] equivalence classes of connective exact dg categories $\A$ which are $0$-Auslander;
\item[(2)] equivalence classes of pairs $(\P,\I)$ consisting of 
\subitem{$\bullet$} a connective additive dg category $\P$, and
\subitem{$\bullet$} an additive dg subcategory $\I$ such that $H^0(\I)$ is covariantly finite in $H^0(\P)$.
\end{itemize}
The bijection from $(1)$ to $(2)$ sends $\A$ to the pair $(\P_{\A},\I_{\A})$ formed by  the full dg subcategory $\P_{\A}$ on the projectives of $\A$ and its full dg subcategory $\I_{\A}$ of projective-injectives. 
The inverse bijection sends $(\P,\I)$ to the $\tau_{\leq 0}$-truncation of the  dg subcategory $(\P*\Sigma\P)\cap \ker\, \Ext^1(-,\I)$ of $\pretr(\P)$.
\end{theorem}
We will prove the theorem after Proposition~\ref{prop:0-Auslander}.
The map from (2) to (1) is obtained independently in \cite[Theorem 4.1]{FangGorskyPaluPlamondonPressland23a}.
Clearly, under the correspondence of the theorem,
\begin{itemize}
\item[a)] the exact dg categories in $(1)$ whose corresponding extriangulated category is moreover reduced correspond to the pairs in $(2)$ where $\I=0$; 
\item[b)] the exact dg categories in $(1)$ with the trivial exact structure (so that each object is both projective and injective) correspond to the pairs in $(2)$ where $\I=\P$.
\end{itemize}

For a pair $(\P,\I)$ in $(2)$, we denote by $\A_{\P,\I}$ the dg subcategory $\P*\Sigma\P\cap \ker(\Ext^1(-,\I))$ of $\pretr(\P)$ and by $\Q_{\P,\I}$ (resp.,~$\J_{\P,\I}$) be the full dg subcategory of $\A_{\P,\I}$ whose objects are the direct summands in $H^0(\A_{\P,\I})$ of objects in $H^0(\P)\subset H^0(\A_{\P,\I})$ (resp.,~$H^0(\I)\subset H^0(\A_{\P,\I})$).
Clearly $\Q_{\P,\I}$ is stable under kernels of retractions in $\pretr(\Q_{\P,\I})\iso \pretr(\P)$.
Below we will show that $H^0(\A_{\P,\I})$ is {\em weakly idempotent complete}, 
i.e.,~retractions of epimorphisms possess kernels, cf.~Lemma~\ref{lem:weakly idempotent complete}.

For a connective exact dg category $\A$, we denote by $\overline{\A}$ the full dg subcategory 
of $\D^{b}_{\dg}(\A,\S)$ consisting of objects in the closure under kernels of retractions of $H^0(\A)$ in $\D^{b}(\A,\S)$. 
Clearly $\overline{\A}$ is stable under extensions in $\D^b_{\dg}(\A,\S)$ and 
the thus inherited exact structure $\overline{\S}$ induces a quasi-equivalence $\D^b_{\dg}(\overline{\A},\overline{\S})\iso \D^b_{\dg}(\A,\S)$.

\begin{definition}\label{def:equ}
Two connective exact dg categories $(\A,\S)$ and $(\B,\S')$ are {\em equivalent} 
if there is an exact quasi-equivalence $\overline{\A}\iso \overline{\B}$.
 
Two pairs $(\P,\I)$ and $(\P',\I')$ in Theorem~\ref{thm:0-Auslander correspondence} (2) are {\em equivalent} if there is a quasi-equivalence $\Q_{\P,\I}\iso \Q_{\P',\I'}$ inducing a quasi-equivalence $ \J_{\P,\I}\iso \J_{\P',\I'}$.
\end{definition}
It is straightforward to verify the following fact.
\begin{lemma}\label{rmk: equivalence 0-Auslander}
A connective exact dg category $(\A,\S)$ is $0$-Auslander if and only if $(\overline{\A},\overline{\S})$ is $0$-Auslander.
\end{lemma}
\begin{lemma}\label{lem:weakly idempotent complete}
For a pair $(\P,\I)$ in Theorem~\ref{thm:0-Auslander correspondence} (2), the category $H^0(\A_{\P,\I})$ is weakly idempotent complete.
\end{lemma}
\begin{proof}
It is enough to show that $\P\ast\Sigma \P$ is stable under kernels of retractions in $\tr(\P)$.
Suppose we have $Y\iso X\oplus Z$ in $\tr(\P)$, where both $Y$ and $Z$ are in $\P\ast\Sigma\P$.
Then $X$ is an extension of $\Sigma^{-1} Z$ and $Y$ and so $X\in \Sigma^{-1}\P\ast\P\ast\Sigma\P$.
Since $\Hom_{\tr(\P)}(\Sigma^{-1}\P,X)=0$, we have $X\in \P\ast\Sigma\P\ast\Sigma\P=\P\ast\Sigma \P$.
\end{proof}
Now we show that for a pair $(\P,\I)$ in Theorem~\ref{thm:0-Auslander correspondence} (2), the dg subcategory $\A_{\P,\I}$ of $\pretr(\P)$ inherits a canonical exact dg structure whose corresponding extriangulated category is 0-Auslander.
\begin{lemma}\label{lem:0-Auslander}
Let $\P$ be a connective additive dg category and $\I\subseteq \P$ a full additive dg subcategory such that $H^0(\I)$ is covariantly finite in $H^0(\P)$.
The following statements hold.
\begin{itemize}
\item[(1)]  
The full subcategory $H^0(\A_{\P,\I})$ of $\tr(\P)$ is extension-closed and the inherited extriangulated structure is 0-Auslander. 
Put $\A=\tau_{\leq 0}\A_{\P,\I}$ and $\S$ the corresponding exact structure on $\A$.
\item [(2)] 
We have that $\Q_{\P,\I}$ is the full dg subcategory of projectives in $\A_{\P,\I}$ and $\J_{\P,\I}$ is the full dg subcategory of projective-injectives in $\A_{\P,\I}$. Moreover, the pair $(\Q_{\P,\I},\J_{\P,\I})$ is equivalent to $(\P,\I)$.
\item [(3)]The bounded dg derived category of $(\A,\S)$ is $\pretr(\P)$.
\end{itemize}
\end{lemma}
\begin{proof}
(1) Since $\P$ is connective, the full dg subcqtegory $\P\ast\Sigma \P$ is extension-closed in $\pretr(\P)$ and hence $\A_{\P,\I}$ is also extension-closed in $\pretr(\P)$.

(2) It is clear that the objects in $\Q_{\P,\I}$ are projective.
 Let $X$ be an object in $\A$. 
 Then we have the following triangle in $\tr(\P)$:
 \begin{equation}\label{tri:X}
P_1\xrightarrow{f} P_2\xrightarrow{g} X\xrightarrow{h} \Sigma P_1.
\end{equation}
Suppose that $X$ is projective in $\A$. 
Then triangle (\ref{tri:X}) splits 
and hence $X\oplus P_1\iso P_2$ and $X\in \Q_{\P,\I}$.
Therefore, $\Q_{\P,\I}$ is the full dg subcategory of projective objects.
It is also clear that the objects in $\I$ are projective-injective.
 Let us show that $H^0(\A)$ has dominant dimension at least $1$: if $X$ is projective in $\A$, then the triangle (\ref{tri:X}) splits.
Since $H^0(\I)$ is covariantly finite in $H^0(\P)$, we have a triangle 
$P_2\xrightarrow{u} I\xrightarrow{v} Y\xrightarrow{w} \Sigma X$ in $\tr(\P)$,
where $I\in\I$ and $u$ is a left $\I$-approximation of $P_2$. 
Then it is clear that $Y$ lies in $\A$. 
 We have $\Hom_{\tr(\P)}(A,\Sigma^2A')=0$ for $A$ and $A'$ in $\A$. 
 By definition, we have that $I$ is injective in $H^0(\A)$.
 Therefore $Y$ is injective in $H^0(\A)$.
We have the following diagram in $\tr(\P)$:
\begin{equation}
\begin{tikzcd}\label{dia:X}
X\ar[r,equal]\ar[d]&X\ar[d]&&\\
P_2\ar[r,"u"]\ar[d]&I\ar[r,"v"]\ar[d]&Y\ar[d,equal]\ar[r]&\Sigma P_2\ar[d]\\
P_1\ar[r]&U\ar[r]&Y\ar[r]&\Sigma P_1.
\end{tikzcd}
\end{equation}
Since $U$ is an extension of $P_1$ and $Y$, it is injective in $\A$.
Therefore, from the middle column of diagram (\ref{dia:X}) we deduce that the dominant dimension of $X$ is a least 1.
Therefore, $H^0(\A_{\P,\I})$ is a $0$-Auslander extriangulated category.
If $X$ is moreover injective in $\A$, then the triangle in the middle column splits and hence $X$ lies in $\J_{\Q,\I}$.
So $\J_{\P,\I}$ is the full dg subcategory of projective-injectives in $\A_{\P,\I}$.
It is straightforward to check that the pair $(\Q_{\P,\I},\J_{\P,\I})$ is equivalent to $(\P,\I)$.

(3) The natural dg functor $\A\rightarrow \A_{\P,\I}\rightarrow \pretr(\P)$ is clearly exact.
So we have the following diagram in $\Hqe$:
\[
\begin{tikzcd}
\P\ar[r,hook]&\A\ar[r,hook]\ar[d,hook]&\A_{\P,\I}\ar[r,hook]&\pretr(\P)\\
&\D^b_{\dg}(\A,\S).\ar[rru,dashed,""]
\end{tikzcd}
\]
Note that the objects in $\P$ are projective in $\A$ and generate $\D^b(\A,\S)$ as a triangulated category.
The claim follows from Theorem~\ref{main} and Proposition~\ref{higher}.
\end{proof}

 \begin{proposition}\label{prop:0-Auslander}
Let $(\A',\S')$ be a connective exact dg category such that the corresponding extriangulated category $H^0(\A')$ is $0$-Auslander. 
Let $\P\subset \A'$ be the full dg subcategory of projectives and 
$\I\subset \P$ the full dg subcategory of projective-injectives (so $H^0(\I)$ is covariantly finite in $H^0(\P)$).
Then $(\A',\S')$ is equivalent (in the sense of Definition~\ref{def:equ}) to $(\A,\S)$ where $\A=\tau_{\leq 0}(\A_{\P,\I})$ and $\S$ is the exact structure given by Lemma~\ref{lem:0-Auslander}.
When $\I=0$, then $(\A',\S')$ is exactly quasi-equivalent to $(\A,\S)$.
\end{proposition}
\begin{proof}
Since $H^0(\A')$ is $0$-Auslander, the subcategory $H^0(\I)$ is covariantly finite in $H^0(\P)$.
So $(\P,\I)$ forms a pair in Theorem~\ref{thm:0-Auslander correspondence} (2).
We have the following diagram in $\Hqe$:
\[
\begin{tikzcd}
&\P\ar[r]\ar[ld,hook]&\A'\ar[r,hook]&\D^b_{\dg}(\A',\S')\\
\A_{\P,\I}\ar[r,hook]&\pretr(\P)\ar[rru,"\mu"swap]&&
\end{tikzcd}
\]
By Theorem~\ref{main} and Proposition~\ref{higher}, the morphism $\P\rightarrow \D^b_{\dg}(\A',\S')$ is quasi-fully faithful and $\D^b(\A',\S')$ is generated as a triangulated category by the objects in $\P$. 
Therefore, the morphism $\mu:\pretr(\P)\rightarrow \D^b_{\dg}(\A',\S')$ is a quasi-equivalence.
By Lemma~\ref{lem:0-Auslander} 3), $\pretr(\P)$ is the dg derived category of $(\A,\S)$.
It is enough to show that the quasi-essential image of $\A\subset \pretr(\P)$ under the morphism $\mu$ is in $\overline{\A'}$.
Let $X\in \D^b(\A',\S')$ be an object with a triangle (\ref{tri:X}) such that $\Hom_{\D^b(\A',\S')}(X,\Sigma \I)=0$. 
We claim that $X\in \overline{\A'}$. 
Indeed, the object $P_1$ in the triangle (\ref{tri:X}) admits a triangle
$
P_1\rightarrow I\rightarrow K\rightarrow \Sigma P_1
$
where $I\in \I$ and $K$ is injective in $H^0(\A')$.
We have the following diagram
\begin{equation}
\begin{tikzcd}\label{dia:Y}
P_1\ar[r,"f"]\ar[d]&P_2\ar[r,"g"]\ar[d]&X\ar[d,equal]\ar[r]&\Sigma P_1\ar[d]\\
I\ar[r]\ar[d]&U\ar[r]\ar[d]&X\ar[r]&\Sigma I\\
K\ar[r,equal]&K.& &
\end{tikzcd}
\end{equation}
By assumption, we have $\Hom(X,\Sigma I)=0$ and therefore $X\oplus I$ is isomorphic to $U$, an extension of $P_2$ and $K$ (so $U\in \A'$). 
Hence $X$ lies in $\overline{\A'}$.
On the other hand, it is clear that each object in $\A'$ is in the quasi-essential image of $\A$ under the morphism $\mu$.
If $\I=0$, then $X\iso U$ lies in $\A'$ and hence $\A'$ is quasi-equivalent to $\A$.
\end{proof}
\begin{proof}[Proof of Theorem~\ref{thm:0-Auslander correspondence}]\label{proof:main theorem}
By Lemma~\ref{rmk: equivalence 0-Auslander}, the property of being 0-Auslander is invariant under the equivalence relation introduced in Definition~\ref{def:equ}.
For a connective 0-Auslander exact dg category $(\A,\S)$, the inclusion $\A\rightarrow \overline{\A}$ induces a quasi-equivalence $\D^b_{\dg}(\A,\S)\iso \D^b_{\dg}(\overline{\A},\overline{\S})$ and $\D^b_{\dg}(\A,\S)$ is quasi-equivalent to $\pretr(\P_{\A})$.
Therefore, the map sending the equivalence class of $(\A,\S)$ to the equivalence class of $(\P_{\A},\I_{\A})$ is well-defined.

Let $(\P,\I)$ be a pair in Theorem~\ref{thm:0-Auslander correspondence} (2).
Put $\A=\tau_{\leq 0}\A_{\P,\I}$.
By Lemma~\ref{lem:0-Auslander},  $(\A,\S)$ is a $0$-Auslander exact dg category. 
By Lemma~\ref{lem:weakly idempotent complete} we have that  $H^0(\A_{\P,\I})$ is weakly idempotent complete. 
Therefore, the map sending the equivalence class of a pair $(\P,\I)$ to the equivalence class of $(\A,\S)$ is well-defined.

Combining Lemma~\ref{lem:0-Auslander} and Proposition~\ref{prop:0-Auslander}, we see that the above maps are inverse to each other.
\end{proof}

For an extriangulated category $(\C,\mathbb E,\mathfrak s)$, 
we denote by $\proj \C$ respectively $\inj \C$ its full subcategory of projectives respectively injectives.
Inspired by recent results of Fang--Gorsky--Palu--Plamondon--Pressland~\cite{FangGorskyPaluPlamondonPressland23a}, we have the following result which relates algebraic 0-Auslander extriangulated categories to homotopy categories of two-term complexes.
It confirms a conjecture in~\cite{FangGorskyPaluPlamondonPressland23a} in the algebraic case.
\begin{theorem}\label{thm:FGPPP}
For each algebraic $0$-Auslander extriangulated category $\C$, we have
an equivalence of extriangulated categories
\begin{equation}\label{equ:0-Auslander}
\C/[\inj \C\rightarrow\proj \C]\iso \H^{[-1,0]}(\proj \C/[\I])
\end{equation}
where $\I$ is the full subcategory of $\C$ consisting of projective-injectives.
\end{theorem}
\begin{proof}
Let $(\A,\S)$ be a connective dg enhancement for $\C$ and $\I$ its full dg subcategory of projective-injectives.
Let $\A/\I$ be the dg quotient. 
By Theorem~\ref{quot} the dg quotient $\A/\I$ carries a canonical exact structure $(\A/\I,\overline{\S})$ whose corresponding extriangulated category is $H^0(\A)/[\I]$ which is again {0-Auslander}.
So we may replace $(\A,\S)$ by $(\A/\I,\overline{\S})$ and assume $\I=0$. 
Let $\P$ be the full dg subcategory of projectives in $\A$ and $\J$ the full dg subcategory of injectives in $\A$.
Then we have $\J=\Sigma \P$ in $\pretr(\P)$ and $\Hom_{H^0(\A)}(\P,\J)=0$.
By Proposition~\ref{prop:0-Auslander}, $\A$ is exactly quasi-equivalent to $\tau_{\leq 0}\A_{\P,0}$.
We keep the same notations $\P$ and $\J$ for the corresponding subcategories $H^0(\P)$ and $H^0(\J)$ of $H^0(\A)$.

The dg functor $\P\rightarrow H^0(\P)$ induces a canonical exact morphism 
$
\A_{\P,0}\rightarrow \A_{H^0(\P),0}
$
which is quasi-dense.
By taking 0-homology on both sides, we obtain a canonical functor
$
H^0(\A)\rightarrow \mathcal H^{[-1,0]}(H^0(\P))
$.
In the homotopy category $\mathcal H^{[-1,0]}(H^0(\P))$ of two-term complexes we have $\Hom(\Sigma \P,\P)=0$. 
Therefore, it induces a canonical functor 
\[
F: \C'=H^0(\A)/(\J\rightarrow \P)\iso \mathcal H^{[-1,0]}(H^0(\P))=\C''.
\]
We denote by $\mathbb E'$ respectively $\mathbb E''$ the bifunctor associated with the extriangulated category $\C'$ respectively $\C''$.
Let $X$ and $Y$ be objects in $\A$. 
Then we have triangles in $\tr(\P)$
\begin{equation}\label{tri:X'}
P_1\xrightarrow{f} P_0\rightarrow X\rightarrow \Sigma P_1,
\end{equation}
and 
\begin{equation}\label{tri:Y'}
Q_1\xrightarrow{g} Q_0\rightarrow Y\rightarrow \Sigma Q_1,
\end{equation}
where $P_i$ and $Q_i$ are in $\P$ for $i=0$, $1$.
Since $\Hom_{H^0(\A)}(\P,\J)=0$, we have 
\[
\Hom_{\C'}(\P,\P)\iso \Hom_{H^0(\A)}(\P,\P)\iso\Hom_{\C''}(\P,\P),
\]
and 
\[
\Hom_{\C'}(\P,\Sigma\P)\iso \Hom_{H^0(\A)}(\P,\Sigma \P)\iso \Hom_{\C''}(\P,\Sigma\P)=0.
\]
We apply the cohomological functor $\Hom_{\tr(\P)}(-,\P)$ to the triangle (\ref{tri:X'}) and we get a long exact sequence of Hom spaces in $\tr(\P)$
\[
\ldots\rightarrow (\Sigma P_1,\P)\rightarrow (X,\P)\rightarrow (P_0,\P)\xrightarrow{(f,\P)} (P_1,\P)\rightarrow (X,\Sigma \P)\rightarrow 0. 
\]
By direct inspection we see that $\Hom_{\C'}(X,\P)$ is isomorphic to the kernel of the map $\Hom(f,\P)$ and that $\Hom_{H^0(\A)}(X,\Sigma \P)\iso \Hom_{\C'}(X,\Sigma \P)$. 
Therefore, we have 
\[
\Hom_{\C'}(X,\P)\iso \Hom_{\C''}(FX,F \P)
\]
 and $\Hom_{\C'}(X,\Sigma \P)\iso \Hom_{\C''}(FX,F \Sigma \P)$.
 Now from the $\mathbb E'$-triangle given by (\ref{tri:Y'}), we obtain the following diagram:
\[
\begin{tikzcd}
\cdots\ar[r]& {\C'}(X,Q_0)\ar[r]\ar[d,"\sim"]& {\C'}(X,Y)\ar[d,dashed]\ar[r]& \mathbb E'(X,Q_1)\ar[r,"\mathbb E'{(}X{,}g{)}"]\ar[d,"\sim"]& \mathbb E'(X,Q_0)\ar[d,"\sim"]\\
\cdots\ar[r]& {\C''}(FX,FQ_0)\ar[r]& {\C''}(FX,FY)\ar[r]& \mathbb E''(FX,FQ_1)\ar[r,"\mathbb E''{(}FX{,}Fg{)}"swap]& \mathbb E''(FX,FQ_0).
\end{tikzcd}
\]
Since $(\C',\mathbb E',\mathfrak s')$ is hereditary, the space $\mathbb E'(X,Y)$ is isomorphic to the cokernel of $\mathbb E'(X,g)$. 
Similarly, the space $\mathbb E''(FX,FY)$ is isomorphic to the cokernel of $\mathbb E''(FX,Fg)$.
We have $\mathbb E(X,Q_1)\iso \Hom_{H^0(\A)}(X,\Sigma Q_1)\iso\Hom_{\C'}(X,\Sigma Q_1)$. 
By the Five-Lemma, we see that $\Hom_{\C'}(X,Y)\iso \Hom_{\C''}(FX,FY)$ and $\mathbb E'(X,Y)\iso \mathbb E''(FX,FY)$.
Therefore, the functor $F$ is fully faithful and is an equivalence of extriangulated categories.
\end{proof}
\begin{remark}
While the statement of Theorem~\ref{thm:FGPPP} is irrelevant with the dg enhancement, our construction of the equivalence (\ref{equ:0-Auslander}) does depend on the dg enhancement.
\end{remark}
	\def\cprime{$'$} \def\cprime{$'$}
	\providecommand{\bysame}{\leavevmode\hbox to3em{\hrulefill}\thinspace}
	\providecommand{\MR}{\relax\ifhmode\unskip\space\fi MR }
	\providecommand{\MRhref}[2]{%
		\href{http://www.ams.org/mathscinet-getitem?mr=#1}{#2}
	}
	\providecommand{\href}[2]{#2}
%


	\bibliographystyle{amsplain}
	\bibliography{stanKeller}

\end{document}